\documentclass{article}

\usepackage[T1]{fontenc}
\usepackage[utf8]{inputenc}
\usepackage[english]{babel}
\usepackage{newtxtext}

\usepackage{amsmath}
\usepackage{amssymb}
\usepackage{bm}
\usepackage{bbm}
\usepackage{spalign}
\usepackage{enumerate}
\usepackage{fix-cm}
\usepackage{paralist}
\usepackage{enumitem}
\usepackage{multicol}

\usepackage{graphicx} 
\graphicspath{ {../optDesign/csv/plots/} }
\usepackage{xcolor}
\usepackage{tikz}
\usepackage{pgfplots}
\pgfplotsset{compat=1.18}
\usepgfplotslibrary{groupplots}
\usepackage{caption}
\usepackage{float}
\usepackage{wrapfig}
\usepackage{subcaption}
\usepackage{longtable}
\usepackage{afterpage}

\usepackage{charter}
\newcommand{\package}{\texttt{Boscia.jl}}
\newcommand{\frankwolfe}{\texttt{FrankWolfe.jl}}
\newcommand{\scip}{\texttt{SCIP}}
\newcommand{\hypatia}{\texttt{Hypatia.jl}}
\newcommand{\pajarito}{\texttt{Pajarito.jl}}
\newcommand{\julia}{\texttt{Julia}}

\usepackage{url} 

\graphicspath{{images/}}

\usepackage{fancyvrb}

\usepackage[theme=grayscale]{jlcode}

\usepackage[section]{algorithm}
\usepackage{algorithmicx} 
\usepackage{algpseudocode}

\usepackage{xspace}

\usepackage[figuresright]{rotating}
\usepackage{csvsimple}
\usepackage{makecell}
\usepackage{caption}

\usepackage{xcolor} 
\usepackage{array}
\newcolumntype{H}{>{\setbox0=\hbox\bgroup}c<{\egroup}@{}}

\usepackage[colorinlistoftodos]{todonotes}
\usepackage[scaled=0.9]{FiraMono}

\usepackage{booktabs}

\usepackage{amsthm}

\newtheorem{proposition}{Proposition}
\newtheorem{theorem}{Theorem}
\newtheorem{definition}{Definition}
\newtheorem{lemma}{Lemma}
\newtheorem{corollary}{Corollary}
\newtheorem*{remark}{Remark}
\newtheorem{conjecture}{Conjecture}

\usepackage{chngcntr}
\usepackage{apptools}
\AtAppendix{\newtheorem{definition}{Definition}[section]
\newtheorem{theorem}[definition]{Theorem}
\newtheorem{proposition}[definition]{Proposition}
\newtheorem{lemma}[definition]{Lemma}
\newtheorem{corollary}[definition]{Corollary}}
 
\newif\ifarxiv
\arxivtrue

\title{Solving Optimal Experiment Design\\ with Mixed-Integer Convex Methods}

\author{\name Deborah Hendrych \email \href{mailto:Hendrych@zib.de}{hendrych@zib.de}\\
\addr Technische Universit\"at Berlin, Germany\\
\addr Zuse Institute Berlin, Germany
\AND
\name Mathieu Besançon \email \href{mailto:mathieu.besancon@inria.fr}{mathieu.besancon@inria.fr} \\
\addr Université Grenoble Alpes, Inria, Grenoble, France \\
\addr Zuse Institute Berlin, Germany
\AND
\name Sebastian Pokutta \email \href{mailto:pokutta@zib.de}{pokutta@zib.de} \\
\addr Technische Universit\"at Berlin, Germany\\
Zuse Institute Berlin, Germany
}

\DeclareMathOperator*{\argmax}{argmax}
\DeclareMathOperator*{\argmin}{argmin}
\DeclareMathOperator{\Tr}{Tr}
\DeclareMathOperator{\rank}{rank}
\DeclareMathOperator{\diag}{diag}
\DeclareMathOperator{\Diag}{Diag}

\DeclareMathOperator{\conv}{conv}

\newcommand{\innp}[2]{\left\langle #1, #2 \right\rangle}
\newcommand{\norm}[1]{\left\| #1 \right\|}
\newcommand{\vx}{\mathbf{x}}
\newcommand{\vvv}{\mathbf{v}}

\newcommand{\vy}{\mathbf{y}}
\newcommand{\vd}{\mathbf{d}}

\newcommand{\vl}{\mathbf{l}}
\newcommand{\vu}{\mathbf{u}}
\newcommand{\va}{\mathbf{a}}
\newcommand{\vs}{\mathbf{s}}
\newcommand{\vz}{\mathbf{z}}
\newcommand{\vw}{\mathbf{w}}
\newcommand{\ve}{\mathbf{e}}
\newcommand{\R}{ {\mathbb R} } 
\newcommand{\Z}{ {\mathbb Z} } 
\newcommand{\N}{ {\mathbb N} } 
\newcommand{\sdCone}{\mathbb{S}_+}
\newcommand{\pdCone}{\mathbb{S}_{++}}
\newcommand{\mRoot}[1]{ #1^\frac{1}{2}}
\newcommand{\invMRoot}[1]{ #1^{-\frac{1}{2}}}

\newcommand{\bcomma}{,\allowbreak}  

\usepackage{jmlr2e}
\usepackage[noabbrev,capitalize,nameinlink]{cleveref}[0.21]
\usepackage{natbib}
\setcitestyle{authoryear,round,citesep={;},aysep={,},yysep={;}}
\usepackage[vvarbb]{newtxmath}
\usepackage{multirow}

\usepackage{todonotes}

\begin{document}

\maketitle

\begin{abstract}
    We tackle the Optimal Experiment Design Problem, which consists of choosing experiments to run or observations to select from a finite set
    to estimate the parameters of a system.
    The objective is to maximize some measure of information gained about the system from the observations, leading to a convex integer optimization problem.
    We leverage \package, a recent algorithmic framework, which is based on a nonlinear branch-and-bound algorithm with node relaxations solved
    to approximate optimality using Frank-Wolfe algorithms.
    One particular advantage of the method is its efficient utilization of the polytope formed by the original constraints which is preserved by the method, unlike alternative methods relying on epigraph-based formulations.
    We assess the method against both generic and specialized convex mixed-integer approaches.
    Computational results highlight the performance of the proposed method, especially on large and challenging instances.
\end{abstract}

\section{Introduction}\label{sec:Intro}
The \emph{Optimal Experiment Design Problem (OEDP)} arises in statistical estimation and empirical studies in many applications areas from Engineering to Chemistry.
For OEDP, we assume we have a matrix $A$ consisting of the rows $\vvv_1\bcomma\dots,\vvv_m\in\R^n$ where each row represents an experiment. 
The ultimate aim is to fit a regression model:
\begin{equation}\label{eq:regressionModel} 
    \min_{\boldsymbol{\theta}\in\R^m} \| A\boldsymbol{\theta} - \vy\|,
\end{equation}
where $\vy$ encodes the responses of the experiments and $\boldsymbol{\theta}$ are the parameters to be estimated.
The set of parameters with size $n$ is assumed to be (significantly) smaller than the number of distinct experiments $m$. 
Furthermore, we assume that $A$ has full column rank, i.e.~the vectors $\vvv_1,\ldots, \vvv_m$ span $\R^n$.

The problem is running all experiments, potentially even multiple times to account for errors, is often not realistic because of time and cost constraints. 
Thus, OEDP deals with finding a subset of size $N$ of the experiments providing the ``most information'' about the experiment space \citep{pukelsheim2006optimal,de1995d}. 
In general, the \emph{number of allowed experiments} $N$ is in the interval $[n,m]$ in order to allow a solution to the regression model in \cref{eq:regressionModel}.

In \cref{sec:OptimalDesign}, we investigate the necessary conditions for a function to be a valid and useful information measure. 
Every information function leads to a different criterion. In this paper, we focus on two popular criteria, namely the A-criterion and D-criterion, see \cref{subsec:AandDProblems}.

In general, OEDP leads to a \emph{Mixed-Integer Non-Linear Problem (MINLP)}. There has been a lot of development in the last years in solving MINLP \citep{kronqvist2019review}.
Nevertheless, the capabilities of current MINLP solvers are far away from their linear counterparts, the \emph{Mixed-Integer Problem (MIP)} solvers \citep{bestuzheva2021scip}, especially concerning the magnitude of the problems that can be handled.
Therefore, instead of solving the actual MINLP formulation, a continuous version of OEDP, called the \emph{Limit Problem}, is often solved and the integer solution is created from the continuous solution by rounding schemes and other heuristics \citep{pukelsheim2006optimal}. 
This does not necessarily lead to optimal solutions, though, and the procedure is not always applicable to a given continuous solution either because it requires a certain sparsity of the continuous solution.

The goal of this paper is to compare the performance of different MINLP approaches for OEDP problems.
A special focus is put on the newly proposed framework Boscia \citep{hendrych2023convex} which can solve larger instances and significantly outperforms the other examined approaches.
The new method leverages Frank-Wolfe algorithms and a formulation different from the other approaches. In \cref{sec:PropertiesAndConvergence}, we establish the convergence of the Frank-Wolfe algorithm on the continuous relaxations of the considered OEDP problems.
The different solution methods are detailed in \cref{sec:SolutionMethods} and evaluation of the computational experiments can be found in \cref{sec:ComputationalExperiments}.

\subsection{Related work}
As mentioned, one established method of solution is the reduction to a simpler problem by removing the integer constraints and employing heuristics to generate an integer solution from the continuous solution. 
Recently, there have been more publications tackling the MINLP formulation of the Optimal Experiment Design Problem. 
These, however, concentrate on specific information measures, in particular the A-criterion, see \citep{nikolov2022proportional,ahipacsaouglu2015first}, and the D-criterion, \citep{welch1982branch,ponte2023branch,li2022d,xu2024submodular,ponte2023computing,fampa2022maximum,chen2024generalized,ponte2024convex}.
The most general solution approach known to the authors was introduced in \citet{ahipacsaouglu2021branch}. 
It considers OEDP under matrix means which, in particular, includes the A-criterion and D-criterion.
While the matrix means covers many information measures of interest, it still yields a restricting class of information functions.
For example, the G-criterion and V-criterion are not included in this class of functions \citep{de1995d}.
The newly proposed framework Boscia only requires the information measures to be either $L$-smooth, i.e.~the gradient is Lipschitz continuous, or generalized self-concordant, thereby covering a larger group of information functions.
In addition, Boscia does not suppose any prior knowledge about the structure of the problem, being thus more flexible in terms of problem formulations.
On the other hand, it is highly customizable, giving the user the ability to exploit the properties of their problems to speed up the solving process.
An in-depth and unified theory for the Optimal Experiment Design Problem can be found in \citet{pukelsheim2006optimal}.

\subsection{Contribution}
Our contribution can be summarized as follows
\paragraph{Unified view on experiment design formulations.} First, we propose a unified view of multiple experiment design formulations as the optimization of a nonlinear (not necessarily Lipschitz-smooth) information function over a truncated scaled probability simplex intersected with the set of integers.
Unlike most other formulations that replace the nonlinear objective with nonlinear and/or conic constraints, we preserve the original structure of the problem.
Additionally, we can easily handle special cases of OEDP without any reformulations and additional constraints on OEDP since we do not suppose a specific problem structure, unlike the approach in \citet{ahipacsaouglu2021branch}. 

\paragraph{Convergence analysis for BPCG.} Second, we establish convergence guarantees for Frank-Wolfe algorithms on continuous relaxations of these problems by exploiting smoothness or self-concordance, and sharpness.
This is necessary in order to solve the various experimental design problems with Boscia. 

\paragraph{Superior solution via Boscia framework.}
Finally, we use the recently proposed Boscia \citep{hendrych2023convex} framework to solve the integer optimization problem with a Frank-Wolfe method for the node relaxations of the Branch-and-Bound tree and show the effectiveness of the method on instances generated with various degrees of correlation between the parameters.

\subsection{Notation} 
In the following let $\lambda_i(A)$ denote the $i$-th eigenvalue of matrix $A$; we assume that these are sorted in increasing order.
Moreover, $\lambda_{\min}(A)$ and $\lambda_{\max}(A)$ denote the minimum and maximum eigenvalue of $A$, respectively.
We define and denote the \emph{condition number} of a positive semi-definite matrix $A$ as $\kappa(A) =\frac{\lambda_{\max}(A)}{\lambda_{\min}(A)}$.
Further, let $\log\det(\cdot)$ be the log-determinant of a positive definite matrix.
Given matrices $A$ and $B$ of same dimensions, $A \circ B$ denotes their Hadamard product.
Given a vector $\vx$, $\diag(\vx)$ denotes the diagonal matrix with $\vx$ on its diagonal 
and let $\Diag(A)$ denote the diagonal of the matrix $A$.
The cones of positive definite and positive semi-definite matrices in $\R^{n\times n}$ will be denoted by $\pdCone^n$ and $\sdCone^n$, respectively. 
We will refer to them as PD and PSD cones.
The set of symmetric $n\times n$ matrices will be denoted by $\mathbb{S}^n$.
Let $\N_0$ denote the set of positive integers including 0 and for $m\in\N$ let $[m] = \{1,2,\dots, m\}$. 
Lastly, we denote matrices with capital letters, e.g.~$X$, vectors with bold small letters, e.g.~$\vx$, and simple small letters are scalars, e.g.~$\lambda$.

\section{The optimal experiment design problem}\label{sec:OptimalDesign}

As stated in the introduction, OEDP aims to pick the subset of predefined size $N$ yielding the most information about the system.
So first, we have to answer the question of how to quantify information.
To that end, we introduce the information matrix
\begin{equation}\label{eq:InformationMatrix}
    X(\vx) = \sum\limits_{i=1}^m x_i \vvv_i \vvv_i^\intercal = A^\intercal \diag(\vx) A
\end{equation}
where $x_i\in\N_0$ denotes the number of times experiment $i$ is to be performed. 
We call $\vx \in\N_0^m$ a \emph{design}.
Throughout this paper, we will use both ways of expressing $X(\vx)$ but will favor the second representation. 
The inverse of the information matrix is the dispersion matrix
\begin{equation*}
    \label{eq:DispersionMatrixSum}
    D(\vx) = \left(\sum\limits_{i=1}^m x_i \vvv_i \vvv_i^\intercal\right)^{-1}.
\end{equation*}
It is a measure of the variance of the experiment parameters \citep{ahipacsaouglu2015first}. 
Our goal is to obtain the ``most'' information, i.e.~we want to maximize some measure of the information matrix.
This is equivalent to minimizing over the dispersion matrix \citep{pukelsheim2006optimal}, thereby minimizing the variance of the parameters.

The matrix $A^\intercal A$ has full rank, by assumption A has full column rank, and is positive definite\footnote{$\vz^\intercal A^\intercal A\vz = \norm{A\vz}_2 ^2 > 0$ for all $\vz\in\R^m\backslash \{0\}$}. 
Because of the non-negativity of $\vx$, the information matrix $X(\vx)$ is in the PSD cone.
In particular, $X(\vx)$ is positive definite for $\vx\in\N_0^m$ if the non-zero entries of $\vx$ correspond to at least $n$ linearly independent columns of $A$. 
To solve the regression problem \eqref{eq:regressionModel}, we need the chosen experiments to span the parameter space. 
The information matrix $X(\vx^*) $ corresponding to the solution $\vx^*$ of OEDP  should wherefore lie in the PD cone. 
\begin{remark}
Experiments can be run only once or be allowed to run multiple times to account for measurement errors. 
In the latter case, we will suppose non-trivial upper (and lower bounds) on the number of times a given experiment can be run. 
The sum of the upper bounds significantly exceeds $N$.
\end{remark}
If non-trivial lower bounds $\vl$ are present, their sum may not exceed $N$ otherwise there is no solution respecting the time and cost constraints. 
A special case of non-trivial lower bounds is obtained if $n$ linearly independent experiments have non-zero lower bounds. 
These experiments can be summarized in the matrix $C=A^\intercal \diag(\vl) A$. 
Notice that $C$ is positive definite. The information matrix then becomes 
\begin{equation}\label{eq:InformationMatrixFusionCase}
X_C(\vx) = C + A^\intercal \diag(\vx-\vl) A .
\end{equation}
We will refer to the resulting OEDP as the \emph{Fusion Experiment Design Problem} or short as the \emph{Fusion Problem}. 
The OEDP with $\vl=0$ will be referred to as \emph{Optimal Problem}.
These are the two types of problems we are focusing on in this paper.

Let us now answer the question posed earlier: How can we measure information?
We need a function $\phi$ receiving a positive definite matrix\footnote{As previously stated, the information matrix of the optimal solution $X(\vx^*)$ has to have full column rank. Thus it suffices to define $\phi$ on the PD cone.} as input and returning a number, that is $\phi: \pdCone^n \rightarrow \R$.
We will lose information by compressing a matrix to a single number.
Hence, the suitable choice of $\phi$ depends on the underlying problem. Nevertheless, there are some properties that any $\phi$ has to satisfy to qualify as an information measure.

\begin{definition}[Information Function \citep{pukelsheim2006optimal}]
    \label{def:informationFunction}
An information function $\phi$ on $\pdCone^n$ is a function $\phi: \pdCone^n \rightarrow \R$ that is positively homogeneous, concave, nonnegative, non-constant, upper semi-continuous and respects the Loewner ordering.
\end{definition}

\paragraph{Respecting Loewner ordering} 
Let $B,D\in\sdCone^n$. We say $D\succcurlyeq B$ if and only if $D-B\succcurlyeq 0$ which is equivalent to $D-B \in\sdCone^n$. 
The order thus imposed on $\sdCone^n$ is called the \emph{Loewner Ordering}.
A map $\phi$ respects the Loewner ordering iff $C\succcurlyeq B \Rightarrow \phi(C)\geq\phi(B) \; \forall C,B\in\sdCone^n$. 
In simple terms, carrying out more experiments should not result in information loss.

\paragraph{Concavity} The condition for concavity is
\[\phi(\gamma B + (1-\gamma)C) \geq \gamma \phi(B) + (1-\gamma)\phi(C) \quad \forall \gamma\in[0,1] \; \forall B,C\in\pdCone^n.\]
The intuition behind this is that we should not be able to achieve the best result by performing two sets of experiments and interpolating the result. 
There should be a design unifying the two sets achieving a better result.

\paragraph{Positive homogeneity} For $\gamma > 0$, we have $\phi(\gamma B) = \gamma\phi(B)$, i.e.~scaling has no influence on the overall information value. 
It has added advantage that later in the optimization, we may forgo any factors since they do not interfere with the ordering created by $\phi$. 

\paragraph{Nonnegativity} By convention, we have $\phi(B) \geq 0$ for all $B\in\sdCone^n$.

\paragraph{Upper semi-continuity} The upper level sets $\left\{M\in\pdCone\,|\, \phi(M)\geq \lambda\right\}$ are closed for all $\lambda\in\R$. 
This is to ensure "good" behavior at the boundary. The functions considered in this paper are all continuous.

\paragraph{}
The most frequently used information functions arise from matrix means $\phi_p$ (\citet{pukelsheim2006optimal}, \citet{ahipacsaouglu2015first}). 
The matrix means are defined as follows:
\begin{definition}[Matrix mean]
    \label{def:matrixMean}
Let $C\in\pdCone^n$ and let $\lambda(C)$ denote the eigenvalues of $C$. The \emph{matrix mean $\phi_p$} of $C$ is defined as
\begin{equation}
    \label{eq:matrixMeanRegular}
    \phi_p(C) = \begin{cases}\max\lambda(C), & \text{for } p = \infty \\ \left(\frac{1}{n} \Tr(C^p)\right)^\frac{1}{p}, & \text{for } p \neq 0, \pm\infty \\ \det(C)^\frac{1}{n}, & \text{for } p = 0 \\ \min\lambda(C), & \text{for } p = - \infty.\end{cases}
\end{equation}
If $C$ is semi-definite with $\rank(C) < n$, the matrix mean is defined as:
\begin{equation}
    \label{eq:matrixMeanSingular}
    \phi_p(C) = \begin{cases}\max\lambda(C), & \text{for } p = \infty \\ \left(\frac{1}{n} \Tr(C^p)\right)^\frac{1}{p}, & \text{for } p = (0, \infty) \\ 0, & \text{for } p = [0,- \infty].\end{cases}
\end{equation}
\end{definition}
Note that the function $\phi_p$ satisfies the requirements of \cref{def:informationFunction} only for $p\leq 1$ \citep{pukelsheim2006optimal}.
The different values of $p$ lead to different so-called criteria.
The \emph{Optimal Experiment Design Problem} using a matrix means function is then defined as
\begin{subequations}
\begin{align*}
	\tag{OEDP} \label{eq:GeneralOptDesignProblem}
    \max_{\vx} \, &  \log(\phi_p(X(\vx))) \\
    \text{s.t. }\, & \sum_{i=1}^m x_i = N \\
    & \vl \leq \vx \leq \vu \\
    & \vx \in \N_0^m,
\end{align*}
\end{subequations}
where $\vu$ and $\vl$ denote the upper and lower bounds, respectively.

\begin{remark}
    Using the logarithm formulation, one can recover Fenchel duality results \citep[Chapter~3]{sagnol2010optimal}.
    It can also by beneficial from a information theory point of view.
    Note, however, that the $\log$ does not necessarily preserve concavity.
\end{remark}
The resulting problem \eqref{eq:GeneralOptDesignProblem} is an integer optimization problem which, depending on the information function $\phi_p$, can be $\mathcal{NP}$-hard. 
The two information measures we will focus on lead to $\mathcal{NP}$-hard problems \citep{welch1982branch,nikolov2022proportional,li2025strong,li2022d}.
Thus, these OEDP are hard to solve. 
Often, therefore, one solves the so-called \emph{Limit Problem} obtained by letting the number of allowed experiments $N$ go to infinity. 
The resulting optimization problem is continuous.
\begin{subequations}
\begin{align*}
	\tag{OEDP-Limit} \label{eq:ContinuousOptDesignProblem}
    \max\limits_{\mathbf{w}} \, &\log(\phi_p(X(\mathbf{w}))) \\
    \text{s.t. }\, & \sum_{i=1}^m w_i = 1 \\
    & \mathbf{w} \in [0,1]^m
\end{align*}
\end{subequations} 
The variable vector $\mathbf{w}$ can be interpreted as a probability distribution. 
Under suitable assumptions, one can generate a finite experiment design from the solution of the limit problem, see \citet{ahipacsaouglu2015first}, \citet[Chapter~12]{pukelsheim2006optimal}.
One assumption is that the support of the solution of the Limit Problem \eqref{eq:ContinuousOptDesignProblem} is smaller than the allowed number of experiments $N$. This is often not the case \citep{pukelsheim2006optimal}. 
If an integer solution can be obtained, it is feasible but not necessarily optimal. 
Optimality can be achieved in some special cases, see \citet{pukelsheim2006optimal}. However, these have very specific requirements that are often not achievable.

\subsection{The A-, D- and GTI-optimal experiment design problem}\label{subsec:AandDProblems}

For this paper, we focus on two particular criteria arising from the matrix means. 
The two most commonly used criteria are the D-optimality and A-optimality criterion, $p=0$ and $p=-1$, respectively.

Many methods for solving MINLP expect minimization formulations. 
Hence, we will reformulate the problems as minimization problems by flipping the sign of the objective.

\paragraph{D-Criterion}
Choosing $p=0$ in \cref{def:matrixMean}, yields for the objective of \eqref{eq:GeneralOptDesignProblem}
\begin{equation*}
    \max\limits_{\vx\in\N_0^m} \, \log\det\left((X(\vx))^\frac{1}{n}\right) 
\end{equation*}
Note that $\log(\det(X)^\frac{1}{n})=\frac{1}{n}\log\det X$ and thus we can state the \emph{D-Optimal Experiment Design Problem} as:
\begin{align}
    \label{eq:DOptDesignProblem}
    \tag{D-Opt}
\begin{split}
    \min\limits_{\vx} \, & -\log(\det(X(\vx))) \\
    \text{s.t. } \, & \sum_{i=1}^m x_i = N \\
    & \vl \leq \vx \leq \vu \\
    & \vx \in \Z
\end{split}
\end{align}
Observe that it is equivalent to minimize $\log\det(X(\vx)^{-1}) = \log\det D(\vx)$, so the determinant of the dispersion matrix. This is also called \emph{the generalized variance} of the parameter $\boldsymbol{\theta}$ \citep{pukelsheim2006optimal}.   
Thus, a maximal value of $\det X$ corresponds to a maximal volume of standard ellipsoidal confidence region of $\boldsymbol{\theta}$ \citep{ponte2023branch}.
Additionally, the D-criterion is invariant under reparameterization \citep{pukelsheim2006optimal}.

\paragraph{A-Criterion}
For parameters with a physical interpretation, the A-optimality criterion is a good choice \citep{pukelsheim2006optimal} as it amounts to minimizing the average of the variances of $\boldsymbol{\theta}$. 
Geometrically, the A-criterion amounts to minimizing the diagonal of the bounding box of the ellipsoidal confidence region \citep{sagnol2010optimal}.
The resulting optimization problem can be stated in two different ways, keeping the $\log$ and using the $\log$ rules and the positive homogeneity of the information functions.
\begin{multicols}{2}
    \noindent 
    \begin{align}
        \label{eq:logAOptDesignProblem}
        \tag{logA-Opt}
    \begin{split}
        \min\limits_{\vx}\, & \log\left(\Tr\left((X(\vx))^{-1}\right)\right) \\
        \text{s.t.} \,& \sum_{i=1}^m x_i = N \\
        & \vl \leq \vx \leq \vu \\
        & \vx \in \N_0^m
    \end{split}
    \end{align}
    \begin{align}
        \label{eq:AOptDesignProblem}
        \tag{A-Opt}
    \begin{split}
        \min\limits_{\vx}\, & \Tr\left((X(\vx))^{-1}\right) \\
        \text{s.t.} \,& \sum_{i=1}^m x_i = N \\
        & \vl \leq \vx \leq \vu \\
        & \vx \in \N_0^m
    \end{split}
    \end{align}
\end{multicols}
We test both formulation within the newly proposed framework. 
Both formulations can be generalized for a real number $p>0$:
\begin{multicols}{2}
    \noindent 
    \begin{align}
        \label{eq:logGTIOptDesignProblem}
        \tag{logGTI-Opt}
    \begin{split}
        \min\limits_{\vx}\, & \log\left(\Tr\left((X(\vx))^{-p}\right)\right) \\
        \text{s.t.} \,& \sum_{i=1}^m x_i = N \\
        & \vl \leq \vx \leq \vu \\
        & \vx \in \N_0^m
    \end{split}
    \end{align}
    \begin{align}
        \label{eq:GTIOptDesignProblem}
        \tag{GTI-Opt}
    \begin{split}
        \min\limits_{\vx}\, & \Tr\left((X(\vx))^{-p}\right) \\
        \text{s.t.} \,& \sum_{i=1}^m x_i = N \\
        & \vl \leq \vx \leq \vu \\
        & \vx \in \N_0^m
    \end{split}
    \end{align}
\end{multicols}
We call the corresponding problem the \emph{Generalized-Trace-Inverse Optimal Experiment Design Problem}, short \emph{GTI-Optimal Problem}. 
\begin{remark}
We conjecture that OEDP under the GTI-criterion is $\mathcal{NP}$-hard for any $p>0$. 
\end{remark}

\paragraph{}
Note that the objectives of \eqref{eq:DOptDesignProblem}, \eqref{eq:AOptDesignProblem}, \eqref{eq:logAOptDesignProblem}, \eqref{eq:logGTIOptDesignProblem} and \eqref{eq:GTIOptDesignProblem} are only well defined if the information matrix $X(\vx)$ has full rank. 

\begin{remark}
    The Fusion Problems are created by replacing $X(\vx)$ with $X_C(\vx)$ in \eqref{eq:DOptDesignProblem}, \eqref{eq:AOptDesignProblem}, \eqref{eq:logAOptDesignProblem}, \eqref{eq:logGTIOptDesignProblem} and \eqref{eq:GTIOptDesignProblem}, respectively.
\end{remark}

For convenience in proofs in the next section, we define 
\begin{equation}
    \label{eq:FeasibleRegionWithoutInteger}
    \tag{Convex Hull}
    \mathcal{P} := [\vl,\vu] \cap \left\{\vx\in\R^m_{\geq0} \,\middle|\, \sum x_i = N\right\}
\end{equation}
as the feasible region of the OEDPs without the integer constraints, i.e.~the convex hull of all feasible integer points. Thus, the feasible region with integer constraints will be noted by $\mathcal{P} \cap \Z$.

We denote the domain of the objective functions by
\begin{equation}
    \label{eq:Domain}
    \tag{Domain}
    D := \left\{\vx\in\R^m_{\geq0} \,\middle|\, X(\vx) \in\pdCone^n\right\}
\end{equation}

Throughout the rest of this paper, we denote by
\begin{multicols}{3}
    \noindent
    \begin{equation*} f(X) = -\log\det(X)\end{equation*}
    \begin{equation*} g(X) = \Tr\left(X^{-p}\right) \end{equation*}
    \begin{equation*} k(X) = \log\left(\Tr\left(X^{-p}\right) \right) \end{equation*}
\end{multicols}
where $X\in\pdCone^n$ and $p>0$. 
With abusive of notation, we use the same identifiers for the objectives of the optimization problems.
\begin{multicols}{3}
    \noindent
    \begin{equation*} f(\vx) = -\log\det(X(\vx))\end{equation*}
    \begin{equation*} g(\vx) = \Tr\left(X(\vx)^{-p}\right) \end{equation*}
    \begin{equation*} k(\vx) = \log\left(\Tr\left(X(\vx)^{-p}\right) \right) \end{equation*}
\end{multicols}

\section{Properties of the problems and convergence guarantees} \label{sec:PropertiesAndConvergence}

To ensure that the newly proposed framework in \citet{hendrych2023convex} is applicable to the previously introduced problems,
we have to show that the Frank-Wolfe algorithms used as the subproblem solvers converge on the continuous subproblems.
Frank-Wolfe normally requires the objective function to be convex and $L$-smooth. 
The convexity part is trivial, for proofs see \cref{sec:Convexity}.
The $L$-smoothness property only holds for the Fusion Problems on the whole feasible region. 
For the Optimal Problems, we will show local $L$-smoothness in \cref{subsec:LSmoothness} and an alternative property guaranteeing convergence of Frank-Wolfe, namely (generalized) self-concordance, in \cref{subsec:GSC}.

Lastly, we establish \emph{sharpness} or the Hölder error bound condition of the objective functions, 
therefore yielding improved convergence rates for Frank-Wolfe methods compared to the standard $\mathcal{O}(1/\epsilon)$ rate where $\epsilon$ is the additive error in the primal gap, see \cref{subsec:StrongConvex}.

\subsection{Lipschitz smoothness} \label{subsec:LSmoothness}

The convergence analysis for Frank-Wolfe algorithms relies on the fact that the objective function values do not grow arbitrarily large. 
The property guaranteeing this is Lipschitz smoothness, further called $L$-smoothness.

\begin{definition}\label{def:LSmooth}
    A function $f:\R^n \rightarrow \R$ is called $L$-smooth if its gradient $\nabla f$ is Lipschitz continuous. That is, there exists a finite $L\in\R_{>0}$ with
    \[\norm{\nabla f(\vx) - \nabla f(\vy)}_2 \leq L \norm{\vx-\vy}_2\]
\end{definition}

Since the functions of interest are in $\mathcal{C}^3$, we can use an alternative condition on $L$-smoothness: 
\[LI\succcurlyeq \nabla_{xx}f(\vx) ,\]
i.e.~the Lipschitz constant $L$ is an upper bound on the maximum eigenvalue of the Hessian.

In \cref{sec:AdditionalSmoothness}, the reader can find $L$-smoothness proofs for $f(X)$, $g(X)$ and $k(X)$ over 
\[\{X\in\pdCone^n \,\mid\, \delta \geq \lambda_{\min}(X)\}.\] 
The lemmas are not immediately useful for us, in general we cannot bound the minimum eigenvalue of the information matrix.
In fact, we will see that we have $L$-smoothness on the whole convex feasible region $\mathcal{P}$ only for the Fusion Problems. 
Due to the positive definiteness of the matrix $C$, the information matrix in the Fusion case, $X_C(\vx)$, is always bounded away from the boundary of the PD cone and we do find a lower bound on the minimum eigenvalue of $X_C(\vx)$.
For Optimal Problems, $X(\vx)$ can be singular for some $\vx\in\mathcal{P}$, i.e.~the minimum eigenvalue of $X(\vx)$ cannot be bounded away from 0 on $\mathcal{P}$.
However, we can show $L$-smoothness in a local region for the Optimal Problems.

We start with the $L$-smoothness proofs for the objectives of the Fusion Problems.
\begin{theorem}\label{th:LSmoothLogDetandTrace}
    The functions $f_F(\vx) := -\log\det\left(X_C(\vx)\right)$ and $g_F(\vx) := \Tr\left(\left(X_C(\vx)\right)^{-p}\right)$ for $p\in\R_{>0}$  are $L$-smooth on $x\in\R_{\geq 0}$ with Lipschitz constants 
    \[L_{f_F}=\left(\max\limits_{i} (AA^\intercal)_{ii}\right) \frac{\|A\|_2^2}{\lambda_{\min}(C)^2}\] 
    and 
    \[L_{g_F} = p(p+1) \left(\max\limits_i (AA^\intercal)_{ii} \right)\frac{\norm{A}_2^2}{\lambda_{\min}(C)^{2+p}},\]
    respectively.
\end{theorem}
\begin{proof}
    For better readability and simplicity, we ignore the dependency of $X$ on $\vx$ for now. 
    As stated above, we can prove $L$-smoothness by showing that the largest eigenvalue of the Hessians is bounded for any $\vx$ in the domain. We have
    \begin{align}
        \nabla^2 f_F(\vx) &= (AA^\intercal)\circ(AX_C^{-2}A^\intercal) \nonumber\\
        \intertext{and}
        \nabla^2 g_F(\vx) &= p(p+1) \left(AA^\intercal\right) \circ \left(AX_C^{-p-2}A^\intercal\right). \nonumber \\
        \intertext{Observe that the Hessian expressions are very similar. 
        Hence, we will only present the proof for $g_F$. The proof of $f_F$ follows the same argumentation. 
        By \citet[page~95]{horn1990hadamard} it is known that the maximum eigenvalue of a Hadamard product between two positive semi-definite matrices $N$ and $M$, $\lambda_{\max}(N\circ M)$, is upper bounded by the maximum diagonal entry of $N$ and the largest eigenvalue of $M$, so}
        \lambda_{\max}(N\circ M) &\leq \max_{i} N_{ii} \lambda_{\max}(M) \label{eq:HornInequality}.\\
        \intertext{In our case, $N=AA^\intercal$ and $M= A X_C^{-p-2}A^\intercal$. 
        By \cref{cor:PowerPosDef} in \cref{subsec:MatrixExponent}, we have that $X_C^{-p-2}$ is positive definite. 
        Thus, both factors are semi-definite and we may use the provided bound.}
        \lambda_{\max}\left(\nabla^2 g_F(\vx)\right) &\leq \max\limits_{i} (AA^\intercal)_{ii} \lambda_{\max}(AX_C^{-p-2}A^\intercal) \nonumber\\
        \intertext{Note that the first factor is simply}
         \max\limits_i (AA^\intercal)_{ii} &= \max\limits_i \vvv_i^\intercal \vvv_i. \label{eq:MaxDiagEntryHessian} \\
         \intertext{For the second factor, we can exploit positive semi-definiteness with the spectral norm.}
        \lambda_{\max}\left(AX_C^{-p-2}A^\intercal\right) &= \norm{AX_C^{-p-2}A^\intercal}_2 \nonumber\\
        &\leq \norm{A}_2^2 \norm{X_C^{-p-2}}_2  \nonumber\\
        \intertext{By \cref{lm:EigValPowerMatrix}, we have $\lambda(A^p) = \lambda(A)^p$.}
        &= \norm{A}_2^2 \lambda_{\max}(X_C^{-1})^{2+p} \nonumber\\
        &= \frac{\norm{A}_2^2}{\lambda_{\min}(X_C)^{2+p}} \nonumber\\
        \intertext{By the definition of $X$ and by Weyl's inequality on the eigenvalues of sums of positive semi-definite matrices \citet{weyl1912asymptotische}, we have that}
        \lambda_{\min}(X_C) &\geq \lambda_{\min}(C) + \lambda_{\min}(A^\intercal \diag(x) A) \geq \lambda_{\min}(C). \nonumber \\
       \intertext{Hence,}
       \lambda_{\max}\left(AX_C^{-p-2}A^\intercal\right)&\leq \frac{\norm{A}_2^2}{\lambda_{\min}(C)^{2+p}}. \label{eq:EigMaxFactorHessian}\\
       \intertext{Combining \eqref{eq:HornInequality} with \eqref{eq:MaxDiagEntryHessian} and \eqref{eq:EigMaxFactorHessian} yields}
        \lambda_{\max}\left(\nabla^2 g_F(\vx)\right) &\leq p(p+1)\left(\max\limits_i v_i^\intercal v_i\right) \frac{\norm{A}_2^2}{\lambda_{\min}(C)^{2+p}}. 
        \end{align}
\end{proof}

For the $\log$ variant of the A-criterion, we can also bound the maximum eigenvalue of the hessian $\nabla^2 k_F$.
\begin{theorem}\label{th:LSmoothLogTrace}
    The function $k_F(\vx) := \log \Tr\left(\left(X_C(\vx)\right)^{-p}\right)$ is $L$-smooth on $\mathcal{P}$ with Lipschitz constant
    \[L_{k_F} = \frac{a_C^{p}}{n}p(p+1) \left(\max\limits_i (AA^\intercal)_{ii} \right)\frac{\norm{A}_2^2}{\lambda_{\min}(C)^{2+p}}\]
    where $a_C$ is the upper bound on the maximum eigenvalue of $X_C(\vx)$, that is
    \[a_C = \lambda_{\max}(C) + \max_{i\in[m]} u_i \max_{j\in[m]} \norm{v_j}_2^2 \] 
    with $u_i$ denoting the upper bound of variable $x_i$.
\end{theorem}
\begin{proof}
    The Hessian of $k_F$ is given by 
    \begin{align*}
        \nabla^2 k_F &= p \frac{(p+1) \left(AA^\intercal\right) \circ \left(AX_C^{-p-2}A^\intercal\right) \Tr\left(X_C^{-p}\right) - p \diag\left(A X_C^{-p-1} A^\intercal\right) \diag\left(A X_C^{-p-1} A^\intercal\right)^\intercal}{\Tr\left(X_C^{-p}\right)^2} .\\
        \intertext{As before, we omit the dependency on $\vx$ for better readability. Note that both parts of the nominator are symmetric matrices. Thus, we can use the upper bound on the eigenvalues of sums of symmetric matrices in \cref{lm:EigValSumSymMatrix}. Further, note that for any $\vz\in\R^n$ the matrix $\vz \vz^\intercal$ is PSD and its maximum eigenvalue is $\innp{\vz}{\vz}$.}
        \lambda_{\max}\left(\nabla^2 k_F\right) &\leq \frac{p(p+1)}{\Tr\left(X_C^{-p}\right)} \lambda_{\max}\left(\left(AA^\intercal\right) \circ \left(AX_C^{-p-2}A^\intercal\right)  \right)
    \end{align*}
    An upper bound on the maximum eigenvalue of $\left(AA^\intercal\right) \circ \left(AX_C^{-p-2}A^\intercal\right)$ is already provided by \cref{th:LSmoothLogDetandTrace}. Observe that 
    \[\Tr\left(X_C^{-p}\right) \geq n \lambda_{\min}\left(X_C^{-p}\right) = \frac{n}{\lambda_{\max}(X_C)^p} \geq \frac{n}{a_C}.\]
    Assembling everything yields the desired bound.
\end{proof}

As stated before, the proofs of \cref{th:LSmoothLogDetandTrace,th:LSmoothLogTrace} work because we can bound the minimum eigenvalue of the Fusion information matrix $X_C$ away from zero thanks to the regular matrix $C$.
This is not the case of the information matrices $X(\vx)$ of the Optimal Problems as they can be singular for some $x\in\mathcal{P}$. 
Therefore, the objectives of the Optimal Problems are not $L$-smooth over $\mathcal{P}$.
We can show, however, that the objectives are $L$-smooth on a restricted area of the feasible region.

Remember that $D$ denotes the domain of the objectives, see \eqref{eq:Domain}.
Since the Frank-Wolfe algorithms make monotonic progress, it is helpful to show that given an initial point $\vx_0$ in the domain $D$ the minimum eigenvalue is bounded in the region 
\[\mathcal{L}_0= \left\{\vx \in D \cap \sum\limits_{i=1}^m x_i = N \,\middle|\, (*)(\vx) \leq (*)(\vx_0)\right\}\]
where $(*)$ is a placeholder for $f(\vx)$, $g(\vx)$ and $k(\vx)$, respectively.
Using the bound provided by \citet[Theorem~1]{merikoski1997bounds} and the fact that the objective value is decreasing, we can bound the minimum eigenvalue by quantities depending on the experiment matrix $A$ and the start point $\vx_0$. Let further $X_0=X(\vx_0)$. 
\begin{theorem}\label{th:LocalSmoothness}
    The functions $f(\vx) := -\log\det\left(X(\vx)\right)$, $g(\vx) := \Tr\left(\left(X(\vx)\right)^{-p}\right)$ and $k(\vx) := \Tr\left(\left(X(\vx)\right)^{-p}\right)$ for $p>0$ are $L$-smooth locally on their respective $\mathcal{L}_0$ for a given initial point $x_0$ with
    \[ L_f = \left(\max\limits_{i} (AA^\intercal)_{ii}\right) \frac{\|A\|_2^2 \left(mN\max_i\norm{v_i}_2^2\right)^{n-1}}{\left(n-1\right)^{n-1} \det X_0^2}\]
    \[ L_g = p(p+1) \left(\max\limits_i (AA^\intercal)_{ii} \right)\norm{A}_2^2 \left(\Tr\left(X_0^{-1}\right)\right)^{2+p}\]
    and
    \[ L_k = \frac{a^p}{n}p(p+1) \left(\max\limits_i (AA^\intercal)_{ii} \right)\norm{A}_2^2 \left(\Tr\left(X_0^{-1}\right)\right)^{2+p}\]
    where $a= \max_{i\in[m]} u_i \max_{j\in[m]} \norm{v_j}_2^2$ is a upper bound on the maximum eigenvalue of $X(\vx)$.
\end{theorem}
\begin{proof}
    As seen in the previous proofs, we have
    \begin{align*}
        L_f &= \left(\max\limits_{i} (AA^\intercal)_{ii}\right) \frac{\|A\|_2^2}{\lambda_{\min}(X(\vx))^2} \\
        L_g &= p(p+1) \left(\max\limits_i (AA^\intercal)_{ii} \right)\frac{\norm{A}_2^2}{\lambda_{\min}(X(\vx))^{2+p}}. \\
        \intertext{and}
        L_k &= \frac{a^{p}}{n}p(p+1) \left(\max\limits_i (AA^\intercal)_{ii} \right)\frac{\norm{A}_2^2}{\lambda_{\min}(X(\vx))^{2+p}} \\
        \intertext{To bound the Lipschitz constant, we need lower bound the minimum eigenvalue of $X(\vx)$. 
        We start with $L_f$ and find by using the bound provided by \citet[Theorem~1]{merikoski1997bounds}}
        \lambda_{\min}(X) &\geq \left(\frac{n-1}{\Tr(X)}\right)^{n-1} \det X . \\
        \intertext{The trace can be upper bounded by $\Tr(X) \leq mN\max_i \norm{\vvv_i}_2^2$ and by knowing that $-\log\det(X(\vx)) \leq -\log\det(X_0)$, we find}
        \lambda_{\min}(X)&\geq \left(\frac{n-1}{mN\max_i\norm{\vvv_i}_2^2}\right)^{n-1} \det X_0 . \\
        \intertext{Thus,}
        L_f &= \left(\max\limits_{i} (AA^\intercal)_{ii}\right) \frac{\|A\|_2^2 \left(mN\max_i\norm{\vvv_i}_2^2\right)^{n-1}}{\left(n-1\right)^{n-1} \det X_0^2}. \\
        \intertext{To lowerbound the minimum eigenvalue of $X(\vx)$ w.r.t. the objectives $g(\vx)$ and $k(\vx)$, we first need}
        \Tr(X^{-p-2}) &= \sum\limits_{i=1}^n \frac{1}{\lambda_i(X)^{p+2}} = \sum\limits_{i=1}^n \frac{1}{\lambda_i(X)^2}\frac{1}{\lambda_i(X)^{p}} \\
        &\leq \frac{1}{\lambda_{\min}(X)^2} \Tr(X^{-p}) .\\
        \intertext{Note that the $\log$ is preserving ordering. Hence, we can use the fact that $\Tr(X^{-p}) \leq \Tr(X_0^{-p})$ and we find}
        &\leq \frac{1}{\lambda_{\min}(X)^2} \Tr(X_0^{-p}).\\
        \intertext{Exploiting positive definiteness and the spectral norm, we have }
        \lambda_{\min}(X)^{p+2} &= \lambda_{\min}\left(X^{\frac{p+2}{2}}\right)^2 = \frac{1}{\lambda_{\max}\left(X^{-\frac{p+2}{2}}\right)^2}\\
        &= \frac{1}{\norm{X^{-\frac{p+2}{2}}}_2^2} \geq \frac{1}{\norm{X^{-\frac{p+2}{2}}}_F^2} = \frac{1}{\Tr\left(X^{-p-2}\right)} \\
        & \\
        \lambda_{\min}(X)^{p+2} &\geq \frac{\lambda_{\min}(X)^2}{\Tr(X^{-p})} \geq \frac{\lambda_{\min}(X)^2}{\Tr\left(X_0^{-p}\right)} \geq \frac{\lambda_{\min}(X)^2}{\Tr\left(X_0^{-1}\right)^p}. \\
        \intertext{Ultimately, we find}
        \lambda_{\min}(X) &\geq \frac{1}{\Tr\left(X_0^{-1}\right)} \\
        \intertext{and the smooth constants are }
        L_g &= p(p+1) \left(\max\limits_i (AA^\intercal)_{ii} \right)\norm{A}_2^2 \left(\Tr\left(X_0^{-1}\right)\right)^{2+p} \\
        \intertext{and}
        L_k &= \frac{a^p}{n}p(p+1) \left(\max\limits_i (AA^\intercal)_{ii} \right)\norm{A}_2^2 \left(\Tr\left(X_0^{-1}\right)\right)^{2+p}.
\end{align*}
Thereby, we have local smoothness constants for the objectives $f$, $g$ and $k$.
\end{proof}

To summarize, we have shown $L$-smoothness of the objectives of the Fusion Problems which guarantees convergence of Frank-Wolfe. 
For the Optimal Problems, we have shown local $L$-smoothness on a restricted area of the feasible region.
This, however, does not suffice to guarantee convergence of Frank-Wolfe.
In the next section, we will show that the objectives satisfy an alternative condition that ensures convergence of Frank-Wolfe.

\subsection{Generalized self-concordance}\label{subsec:GSC}

By \citet{carderera2021simple}, the Frank-Wolfe algorithm also converges (with similar convergence rates to the $L$-smooth case) if the objective is generalized self-concordant. 

\begin{definition}[Generalized Self-Concordance \cite{sun2019generalized}]\label{def:GeneralizedSelfConcordance}
    A three-times differentiable, convex function $f:\R^n\rightarrow \R$ is $(M_f,\nu)$-generalized self-concordant with order $\nu >0$ and constant $M_f\geq 0$, if for all $x\in\text{dom}(f)$ and $\vu,\vvv\in\R^n$, we have
    \[\left|\innp{\nabla^3f(\vx)[\vu]\vvv}{\vvv}\right| \leq M_f \|\vu\|^2_x \; \|\vvv\|_x^{\nu-2} \;\|\vvv\|_2^{3-\nu}\]
    where
    \[\|\mathbf{w}\|_x = \innp{\nabla^2f(\vx)\mathbf{w}}{\mathbf{w}} \text{ and } \nabla^3f(\vx)[\vu] = \lim\limits_{\gamma\rightarrow 0}\gamma^{-1}\left(\nabla^2f(\vx+\gamma \vu) -\nabla^2f(\vx)\right) .\]
    Self-concordance is a special case of generalized self-concordance where $\nu =3$ and $\vu=\vvv$. 
    This yields the condition
    \[ \left|\innp{\nabla^3f(\vx)[\vu]\vu}{\vu}\right| \leq M_f \|\vu\|_x^3 . \]

    For a univariate, three times differentiable function $f: \R \rightarrow \R$, $(M_f,\nu)$-generalized self-concordance\footnote{For self-concordance, $\nu=3$.} condition is 
    \[|f'''(x)| \leq M_f \left(f''(x)\right)^{\frac{\nu}{2}} . \]
\end{definition}

By \citet[Proposition~2]{sun2019generalized}, the composition of a generalized self-concordant function with a linear map is still generalized self-concordant. 
Hence, it suffices that we show that the functions $f(X) = -\log\det(X)$ and $g(X)=\Tr\left(X^{-p}\right)$ and $k(X) = \log\left(\Tr\left(X^{-p}\right)\right)$, $p>0$, are generalized self-concordant for $X$ in the PD cone.

Notice that $f$ is the logarithmic barrier for the PSD cone which is known to be self-concordant \citep{nesterov1994interior}. Convergence of Frank-Wolfe for the \eqref{eq:DOptDesignProblem} Problem is therefore guaranteed. For convergence for the \eqref{eq:GTIOptDesignProblem} Problem and \eqref{eq:AOptDesignProblem} Problem, we can show that $g$ is self-concordant on a part of the PD cone.

\begin{theorem}
    \label{th:GSCTraceInverse}
    The function $g(X) = \Tr\left(X^{-p}\right)$, with $p > 0$, is $\left(3, \frac{(p+2)\sqrt[4]{\alpha^{2p}n}}{\sqrt{p(p+1)}}\right)$-generalized self concordant on $\mathcal{D}=\left\{X\in\pdCone^n\mid 0 \prec X \preccurlyeq \alpha I \right\}$.
\end{theorem}
\begin{proof}
    As with the convexity proofs in \cref{sec:Convexity}, we bring our function back to the univariate case.
    We define $M(t) = V + tU$ where $V\in D$, $U$ a symmetric matrix and $t\in\R$ such that $V+tU\in \mathcal{D}$. 
    By \cref{lm:ConvexityTraceInverse}, the second and third derivatives of $h(t)=\Tr\left(M(t)^{-p}\right)$ are given by
    \begin{align*}
        h''(t)_{\restriction t=0} &= p(p+1) \Tr\left(V^{-p}UV^{-1}UV^{-1}\right) \\
        h'''(t)_{\restriction t=0} &= -p(p+1)(p+2) \Tr\left(V^{-p}UV^{-1}UV^{-1}UV^{-1}\right).\\
        \intertext{Thus, the condition we want to satisfy is}
        \left| p(p+1)(p+2) \Tr\left(V^{-p}UV^{-1}UV^{-1}UV^{-1}\right) \right|&\leq M_f p(p+1) Tr\left(V^{-p}UV^{-1}UV^{-1}\right)^{\nu/2} .\\
        \intertext{If the LHS is equal to 0, we are done. Otherwise, we divide both sides by the LHS and first try to lower bound the following fraction.}
        &\frac{\Tr\left(V^{-\frac{p+1}{2}}UV^{-1}UV^{-\frac{p+1}{2}}\right)^{\nu/2}}{\left|\Tr\left(V^{-\frac{p+1}{2}}UV^{-1}UV^{-1}UV^{-\frac{p+1}{2}}\right) \right|} \\
        \intertext{For that we define $H = V^{-1/2}UV^{-1/2}$ and $L = V^{-p/2}H$.}
        \frac{\Tr\left(V^{-\frac{p+1}{2}}UV^{-1}UV^{-\frac{p+1}{2}}\right)^{\nu/2}}{\left|\Tr\left(V^{-\frac{p+1}{2}}UV^{-1}UV^{-1}UV^{-\frac{p+1}{2}}\right)\right|} &= \frac{\Tr\left(LL^\intercal\right)^{\nu/2}}{\left|\Tr\left(LHL^\intercal\right)\right|} \\
        \intertext{Using Cauchy-Schwartz, we find}
        \left|\Tr\left(LHL^\intercal\right)\right| &= \left|\innp{L^\intercal L}{H}\right| \\
        &\leq \Tr\left(LL^\intercal\right) \sqrt{\Tr\left(H^2\right)}. \\
        \frac{\Tr\left(LL^\intercal\right)^{\nu/2}}{\left|\Tr\left(LHL^\intercal\right)\right|}  &\geq \frac{\Tr\left(LL^\intercal\right)^{\nu/2}}{\Tr\left(LL^\intercal\right)\sqrt{\Tr\left(H^2\right)}} \\
        \Tr\left(HH\right) &= \Tr\left(V^{p/2}V^{-p/2}HHV^{-p/2}V^{p/2}\right) \\
        &= \Tr\left(V^{p/2}LL^\intercal V^{p/2}\right) \\
        &\leq \Tr\left(LL^\intercal\right) \sqrt{\Tr\left(V^{2p}\right)} \\
        \frac{\Tr\left(LL^\intercal\right)^{\nu/2}}{\Tr\left(LL^\intercal\right)\sqrt{\Tr\left(H^2\right)}} &\geq \frac{\Tr\left(LL^\intercal\right)^{\nu/2}}{\Tr\left(LL^\intercal\right)\sqrt{\Tr\left(LL^\intercal\right) \sqrt{\Tr\left(V^{2p}\right)}}} \\
        &= \Tr\left(LL^\intercal\right)^{\nu/2 - 3/2} \frac{1}{\sqrt{\norm{V^p}_F}} \\
        &\geq \Tr\left(LL^\intercal\right)^{\nu/2 - 3/2} \frac{1}{\sqrt{\norm{V^p}_2} \sqrt[4]{n}} \\
        \intertext{By the assumption that V is a feasible point, we have $\norm{V}_2=\lambda_{\max}(V)\leq \alpha$.}
        &\geq \Tr\left(LL^\intercal\right)^{\nu/2 - 3/2} \frac{1}{\sqrt{\alpha^p} \sqrt[4]{n}} \\
        \intertext{Then,}
        \frac{M_f (p(p+1))^{\nu/2}}{p(p+1)(p+2)} \frac{\Tr\left(V^{-p}UV^{-1}UV^{-1}\right)^{\nu-2}}{\Tr\left(V^{-p}UV^{-1}UV^{-1}UV^{-1}\right)} &\geq \frac{M_f (p(p+1))^{\nu/2}}{p(p+1)(p+2)} \Tr\left(LL^\intercal\right)^{\nu/2 - 3/2} \frac{1}{\sqrt{\alpha^p} \sqrt[4]{n}}. \\
        \intertext{We can select $\nu=3$ and re-express the GSC condition as}
        1 &\leq \frac{M_f (p(p+1))^{1/2}}{(p+2)} \frac{1}{\sqrt[4]{\alpha^{2p} n}} \\
        M_f &\geq \frac{(p+2)\sqrt[4]{\alpha^{2p}n}}{\sqrt{p(p+1)}}.
    \end{align*}
    Thus, $g$ is $\left(3, \frac{(p+2)\sqrt[4]{\alpha^{2p}n}}{\sqrt{p(p+1)}}\right)$-generalized self concordant.
\end{proof}

\begin{remark}
    We conjecture that there is another value of $\nu$ with a tighter value for the constant $M_f$ satisfying the condition in \cref{def:GeneralizedSelfConcordance}.
\end{remark}

\begin{corollary}
The function $g(X) = \Tr\left(X^{-1}\right)$ is $\left(3, \frac{3\sqrt[4]{\alpha^2n}}{\sqrt{2}}\right)$-generalized self concordant on $\mathcal{D}=\left\{X\mid 0 \prec X \preccurlyeq \alpha I \right\}$.
\end{corollary}

Hence, the objectives of the \ref{eq:AOptDesignProblem} and \ref{eq:GTIOptDesignProblem} Problems are generalized self-concordant if we can upperbound the maximum eigenvalue of the information matrix $X(\vx)$.

\begin{lemma}
    \label{lm:MaxEigInformationMatrix}
    Let $x\in\mathcal{P}$. Then,
    \[\lambda_{\max}(X(\vx)) \leq \max_i u_i \max_j \norm{\vvv_j}_2^2  \].
\end{lemma}
\begin{proof}
    \begin{align*}
        \lambda_{\max}(X(\vx)) &= \lambda_{\max} \left(\sum\limits_{i=1}^m x_i \vvv_i \vvv_i^\intercal\right) \\
        \intertext{By the Courant Fischer Min-Max Theorem, we have}
        &\leq \sum\limits_{i=1}^m \lambda_{\max}\left(x_i\vvv_i \vvv_i^\intercal\right). \\
        \intertext{Using $\vx \leq \vu$, yields}
        &\leq \max_{i} u_i \max_j \norm{\vvv_j}_2^2.
    \end{align*}
\end{proof}

The proof of \cref{th:GSCTraceInverse} does not translate to the log variant of the A-criterion.
Our computational experiments strongly suggest that the function $k(X) = \log\left(\Tr\left(X^{-p}\right)\right)$ is generalized self-concordant.

\begin{conjecture}
    The function $k(X) = \log\left(\Tr\left(X^{-p}\right)\right)$ is $\left(\nu, M\right)$-generalized self concordant for some $\nu \in [2,3]$ and $M > 0$ on some subset of PD cone.
\end{conjecture}

Thus, we have convergence of Frank-Wolfe for the \ref{eq:DOptDesignProblem}, \ref{eq:AOptDesignProblem} and \ref{eq:GTIOptDesignProblem} Problems.
For the \ref{eq:logAOptDesignProblem} and \ref{eq:logGTIOptDesignProblem} Problems, we have strong empirical evidence that convergence holds but the theoretical proof is still open.
Lastly, we aim to show that we have improved convergence rates of Frank-Wolfe for our problems of interest. 

\subsection{Sharpness}\label{subsec:StrongConvex}

The classical convergence speed of Frank-Wolfe methods given a convex objective is $\Theta(1/\epsilon)$ where $\epsilon$ is the additive primal error to the true optimum $x^*$, i.e.~$f(\vx) - f(\vx^*) \leq \epsilon$. 
For the objectives of the problems defined in \cref{subsec:AandDProblems}, we can prove a stronger property than convexity, namely sharpness on (a subset of) the feasible region $\mathcal{P}$.
This improves the convergence rates of Frank-Wolfe, for the proof see \cref{sec:AppendixLinConvergenceBPCG} and \cite{zhao2025new}.

\begin{definition}[Sharpness {\cite[Section~3.1.5]{braun2022conditional}}]\label{def:Sharpness}
    Let $\mathcal{X}$ be a compact convex set. A convex function $f$ is $(c, \theta)$-sharp (over $\mathcal{X}$) for $0 < M < \infty$, $0 < \theta < 1$ if for all $x\in\mathcal{X}$ and $\vx^* \in \Omega^*_{\mathcal{X}}$, we have 
    \[
        M\left( f(\vx) - f(\vx^*)\right)^\theta \geq \min_{\vy\in\Omega^*_{\mathcal{X}}} \norm{\vx-\vy}
        \]
    where $\Omega^*_{\mathcal{X}}$ denotes the set of minimizers $\min_{\vx\in\mathcal{X}} f(\vx)$.
\end{definition}
\begin{theorem}\label{th:SharpnessObjectives}
The objective functions associated with the Optimal Problem and Fusion Problem
\begin{multicols}{3}
    \noindent
    \begin{align*} f(\vx) &= -\log\det(X(\vx)) \\ f_F(\vx) &= -\log\det(X_C(\vx))\end{align*}
    \begin{align*} g(\vx) &= \Tr\left(X(\vx)^{-p}\right) \\ g_F(\vx) &= \Tr\left(X_C(\vx)^{-p}\right)  \end{align*}
    \begin{align*} k(\vx) &= \log\left(\Tr\left(X(\vx)^{-p}\right) \right) \\ k_F(\vx) &= \log\left(\Tr\left(X_C(\vx)^{-p}\right) \right) \end{align*}
\end{multicols}
with $p > 0$ are $(c, 1/2)$-sharp for some $c > 0$ on $\mathcal{W} := \conv\left\{ x \in \mathcal{P} \cap D \cap \N_0 ^m \right\}$.
\end{theorem} 

To prove \cref{th:SharpnessObjectives}, we first prove that the functions $f(X)=-\log\det(X)$, $g(X)=\Tr(X^{-p})$ and $k(X)=\log\left(\Tr\left(X^{-p}\right)\right)$ with $p>0$, are strongly convex on a part of the PD cone.
\begin{definition}[Strong Convexity]\label{def:StrongConvexity}
A differentiable function $f:D\rightarrow \R$ with $D\subseteq \R^n$ is called $\mu$-strongly convex iff:
\begin{align*}
\innp{\nabla f(\vx) - \nabla f(\vy)}{\vx-\vy} \geq \mu \norm{\vx-\vy}_2^2 \;\; \forall \vx, \vy \in D,
\end{align*}
with $\mu > 0$. If $f$ is twice differentiable, an equivalent requirement is:
\begin{align*}
\nabla^2 f(\vx) \succcurlyeq \mu I \;\; \forall \vx \in D.    
\end{align*}
That is, $\mu$ is a lower bound on the minimum eigenvalue of the Hessian of $f$.
\end{definition}

\begin{lemma}\label{th:StrongConvexityLogDet}
The function $f(X) := -\log\det(X)$ is $\mu$-strongly convex on $\mathcal{D} := \{X\in\pdCone^n \, \mid\, \lambda_{\max}(X) \leq \alpha\}$ with $\mu \leq \frac{1}{a^2}$.
\end{lemma}
\begin{proof}
    Let $A\in\mathcal{D}$, $B\in\mathbb{S}^n$ and $t\in\R$ such that $A+tB\in\mathcal{D}$. 
    Then we can define $h(t)=-\log\det\left(A+tB\right)$. By \cref{lm:ConvexityLogDet}, $f(X)$ is convex and the second derivative of $h$ is given by
    \begin{align*}
        h''(t) &= \sum\limits_{i=1}^n \frac{\lambda_i\left(A^{-1/2}BA^{-1/2}\right)^2}{1+t\lambda_i\left(A^{-1/2}BA^{-1/2}\right)} \\
        \intertext{It is sufficient to show that $h''(t)\restriction_{t=0} \geq \mu$ to argue strong convexity for $f$.}
        h''(t)\restriction_{t=0} &= \sum\limits_{i=1}^n \lambda_i\left(A^{-1/2}BA^{-1/2}\right)^2 \\
        \intertext{By using a similar idea to \cref{lm:EigValProdPosDefMatrix}, we find}
        &\geq \sum\limits_{i=1}^n \lambda_{\min}\left(A^{-1}\right)^2 \lambda_i(B)^2 \\
        &= \lambda_{\min}\left(A^{-1}\right)^2 \Tr\left(B^2\right)\\
        &= \lambda_{\min}\left(A^{-1}\right)^2 \norm{B}_F^2 \\
        \intertext{W.l.o.g. let $\norm{B}_F=1$.}
        &= \frac{1}{\lambda_{\max}(A)^2} \\
        &\geq \frac{1}{\alpha^2} \\
        &= \mu
    \end{align*}
    Thus, the function is strongly convex at any point $A\in D$.
\end{proof}

\begin{lemma}\label{th:StrongConvexityTrPower}
    For $p \in \R_{>0}$, the function $g(X) := \Tr\left(X^{-p}\right)$ is $\mu$-strongly convex on the domain $\mathcal{D} := \{X\in\pdCone^n \, \mid\, \lambda_{\max}(X) \leq \alpha\}$ with 
    \[\mu \leq \frac{p(p+1)}{a^{p+2}}.\]
\end{lemma}
\begin{proof}
    Let $A \in \mathcal{D}$ and $B\in\mathbb{S}^n$ and $t\in\R$ such that $A+tB\in\mathcal{D}$.  
    Define $h(t) = \Tr((A+tB)^{-p})$. 
    For the computation of $h''(t)$, see \cref{lm:ConvexityTraceInverse}.
    \begin{align*}
        h''(t)_{\restriction t=0} &= p(p+1) \Tr\left(A^{-p} BA^{-1}BA^{-1} \right). \\
        \intertext{Since $A^{-p}$ is positive definite by \cref{cor:PowerPosDef} and the trace is cyclic, we have}
        &= p(p+1) \Tr\left(A^{-\frac{p+2}{2}}BBA^{-\frac{p+2}{2}}\right) \\
        \intertext{Note that $BB$ is postive semi-definite. Thus, we can use the bound in \cite{fang1994inequalities}.}
        &\geq p(p+1) \lambda_{\min}\left(A^{-(p+2)}\right) \Tr(BB) \\
        &= p(p+1) \lambda_{\min}\left(A^{-(p+2)}\right) \norm{B}_F^2 \\
        \intertext{W.l.o.g. let $\norm{B}_F=1$.}
        &= \frac{p(p+1)}{\lambda_{\max}\left(A^{p+2}\right)} \\
        &\geq \frac{p(p+1)}{\alpha^{p+2}}
    \end{align*}
    Thus, $h(t)$ is strongly convex with $\mu = \frac{p(p+1)}{\alpha^{p+2}}$. 
\end{proof}

\begin{lemma}\label{lm:StrongConvexityLogofTrace}
    For $p \in \R_{>0}$, , the function $k(X) = \log\left(\Tr(X^{-p})\right)$ is strongly convex on $\mathcal{D}:= \{X\in\pdCone^n \, \mid\, \lambda_{\max}(X) \leq \alpha, \kappa(X) \leq \kappa \}$. 
\end{lemma}
\begin{proof}
    Let $A \in \mathcal{D}$, $B \in \mathbb{S}^n$ and $t\in\R$ s.t. $A+tB\in\mathcal{D}$. We define $h(t) = k(A+tB) = \log\left(\Tr((A+tB)^{-p})\right)$.
    By \cref{lm:ConvexityLogofTrace}, we have
    \begin{align*}
        h''(t)_{\restriction t=0} &= p \frac{(p+1) \Tr\left(A^{-p-2}BB\right)\Tr\left(A^{-p}\right) - p \Tr\left(A^{-p-1}B\right)^2}{\Tr\left(A^{-p}\right)^2} .\\
        \intertext{Using the Cauchy Schwartz inequality $\Tr\left(XY^\intercal\right)^2 \leq \Tr(XX^\intercal)\Tr(YY^\intercal)$ and setting $X = A^{-p/2}$ and $Y=A^{-\frac{p+2}{2}}B$, we have}
        &\geq p \frac{(p + 1 -p) \Tr\left(A^{-p}\right)\Tr\left(A^{-p-2}BB\right)}{\Tr(A^{-p})^2}  \\
        &= p \frac{\sum\limits_{i=1}^n \lambda_i\left(A^{-p-2}BB\right)}{\Tr\left(A^{-p}\right)} . \\
        \intertext{Note that the matrix $BB$ is PSD. The proof of \cref{lm:EigValProdPosDefMatrix} can be adapted for one of matrices being symmetric and the other positive definite.}
        &\geq p \frac{\lambda_{\min}\left(A^{-p-2}\right)\Tr(BB)}{\Tr\left(A^{-p}\right)} \\
        \intertext{Observe that $Tr(BB) = \norm{B}_F^2$. Without loss of generality, we can assume $\norm{B}_F=1$.}
        &\geq p \frac{\lambda_{\min}\left(A^{-p-2}\right)}{n \lambda_{\max}(A^{-p})} \\
        &= \frac{p}{n \kappa(A)^p \lambda_{\max}(A)^2} \\
        &\geq \frac{p}{n \kappa^p a^2} = \mu
    \end{align*}
    Thus, $k$ is strongly convex for any $A\in\mathcal{D}$.
\end{proof}

By \cref{lm:MaxEigInformationMatrix}, we know that the largest eigenvalue of the information matrix can be bounded. As for the minimum eigenvalue, it can be lower bounded on 
\[\mathcal{W} := \conv\left\{ x \in \mathcal{P} \cap D \cap \N_0 ^m \right\}\]
so the convex hull of all domain-feasible integer points, see \cref{fig:DomainFeasibleIntegerPoints}. 
Any point $Y \in\mathcal{W}$ can be represented as a convex combination of the integer vertices $\{V_i\}\in\mathcal{W}$.
By \cref{lm:EigValSumSymMatrix}, we have 
\begin{align*}
    \lambda_{\min}(Y) &= \lambda_{\min}\left(\sum\limits_{i=1}^n \gamma_i V_i\right) \\
    &\geq \sum\limits_{i=1}^n \gamma_i \lambda_{\min}(V_i) \\
    &\geq \min_{V \in \mathcal{W}} \lambda_{\min}(V)
\end{align*}

\begin{figure}[h]
    \centering
    \begin{tikzpicture}
    \coordinate (O) at (0,0);
    \coordinate (A) at (1,0);
    \coordinate (B) at (0,2);
    \coordinate (C) at (1,4);
    \coordinate (D) at (4,5);
    \coordinate (E) at (6,3);
    \coordinate (F) at (6,1);
    \coordinate (G) at (3,0);
    \coordinate (H) at (2,4);
    \coordinate (I) at (2,2);
    \coordinate (J) at (3,1);

    \draw[thick, fill=blue!40, opacity=0.5] 
        (A) -- (B) -- (C) -- (D) -- (E) -- (F) -- (G) -- cycle;

    \draw[thick, fill=red!20, opacity=0.5]
        (4.7,3) ellipse [x radius=3.5, y radius=2.5];

    \draw[thick, fill=orange!70, opacity=0.95] 
        (H) -- (I) -- (J) -- (F) -- (E) -- (D) -- cycle;

    \node[below] at (4,2.8) {$\mathcal{W}$};
    \node[right] at (1.2,1) {$\mathcal{P}$};
    \node[left] at (6.8,4) {$D$};

    \draw[thick] (A) -- (B) -- (C) -- (D) -- (E) -- (F) -- (G) -- cycle; 
    \draw[thick] (4.7,3) ellipse [x radius=3.5, y radius=2.5]; 

    \foreach \x in {0,1,...,8} 
        \foreach \y in {0,1,...,5} 
            \fill (\x,\y) circle (1.5pt); 
    
    \end{tikzpicture}
    \caption{A schematic representation of the feasible region $\mathcal{P}$, the domain of the objective $\mathcal{D}$ and the convex hull of vertices that are both feasible and in the domain denoted as $\mathcal{W}$.}
    \label{fig:DomainFeasibleIntegerPoints}
\end{figure} 

\begin{proof}[Proof of \cref{th:SharpnessObjectives}]
    The functions of interest are compositions of strongly convex functions, see \cref{th:StrongConvexityLogDet} and \cref{th:StrongConvexityTrPower}, and linear maps. 
    Observe that the linear maps $X(\vx)$ and $X_C(\vx)$ are, in general, not injective and hence strong convexity is not preserved.

    We can, however, follow the argumentation in \citet[Example~3.28]{braun2022conditional}. 
    All of the functions of interest have the form $k(\vx)=r(B\vx)$ where $r$ is a strongly convex function. 
    Notice that $B=A^\intercal * A^\intercal$ where $*$ denotes the column-wise Kroenecker product. 
    Observe that
    \[\mathcal{W} := \conv\left(D\cap \mathcal{P} \cap \Z^m\right)\]
    is the convex hull of the feasible integer points that are also in the domain.
    This is a polytope and we may use the bound in \citet[Lemma~3.27]{braun2022conditional}. 
    Notice that the gradient of $k$ is given by $\nabla k(\vx) = B^\intercal \nabla r(B\vx)$. 
    As in \cref{def:Sharpness} let $\Omega^*_{\mathcal{W}}$ denote the set of minimizers of $\min_{\vx\in\mathcal{W}}k(\vx)$ and let $\vx*\in\Omega^*_{\mathcal{W}}$. 
    Then, we have $\innp{\nabla k(\vx*)}{\vx-\vx*} \geq 0$ for $\vx\in\mathcal{W}$.
    \begin{align*}
        k(\vx) - k(\vx^*) &= r(B\vx) - r(B\vx^*) \\
        \intertext{Since $r$ is $\mu$-strongly convex  we have}
        &\geq \innp{\nabla r(B\vx^*)}{B\vx - B\vx^*} + \mu \frac{\norm{B\vx -B\vx^*}^2}{2} \\
        &= \innp{\nabla k(\vx)}{\vx - \vx^*} + \mu \frac{\norm{B\vx -B\vx^*}_2}{2} \\
        &= \mu \frac{\norm{B\vx -B\vx^*}^2}{2} \\
        \intertext{Since our region of interest $\mathcal{W}$ is a polytope \citet[Lemma~3.27]{braun2022conditional} holds and there exists a constant $c > 0$ such that the distance $\norm{B\vx-B\vy}$ between any two points $\vx,\vy\in\mathcal{W}$ is bounded by}
        \norm{B\vx-B\vy} &\geq c \,\text{dist}\left(\vx, \{\mathbf{w}\in\mathcal{W}\mid B\mathbf{w}=B\vy\}\right).  \\
        \intertext{In conclusion, we find}
        k(\vx) - k(\vx^*) &\geq \mu \frac{\norm{B\vx -B\vx^*}^2}{2} \geq \frac{\mu}{2} c^2 \min_{\vz\in\Omega^*_{\mathcal{W}}} \norm{\vx-\vz}^2
    \end{align*}
    Thus, functions are $\left(\frac{\sqrt{2}}{c\sqrt{\mu}}, \frac{1}{2}\right)$-sharp.
\end{proof}

Regarding the constant $c$ in the previous proof, we find
\begin{align*}
\|B\vx-B\vy\|_F &= \|A^\intercal \diag(\vx) A - A^\intercal \diag(y) A\|_F \\
&= \|A^\intercal \diag(\vx-\vy) A\|_F. \\
\intertext{Let $C=A^\intercal$ and $B=\diag(\vx-\vy)A$ and let us use the bound in \cref{lm:BoundsFrobeniusNorm}.}
&\geq \sigma_{\min}(A) \|\diag(\vx-\vy)A\|_F \\
&\geq \sigma_{\min}(A)^2 \|\diag(\vx-\vy)\|_F \\
&= \sigma_{\min}(A)^2 \|\vx-\vy\|_2
\end{align*}
Note that $\|\vx-\vy\|_2 \geq \text{dist}\left(\vx, \{\mathbf{w}\in\mathcal{W}\mid B\mathbf{w}=B\vy\}\right)$ and thus, we have found $c=\sigma_{min}(A)^2$. Note that since $A$ has full column rank, the minimum singular value is strictly greater than 0.

We have thus established sharpness for all objective functions on $\mathcal{W}$. 
Observe that $\mathcal{W}=\mathcal{P}$ for the Fusion Problems, thus, linear convergence of Frank-Wolfe is guaranteed on them.
In general, however, we have $\mathcal{P}\backslash\mathcal{W} \cap D \neq \emptyset$, see also \cref{fig:DomainFeasibleIntegerPoints}. 
That is there exists domain feasible points $\vy\in\mathcal{P}$ which cannot be expressed as the convex combination of domain feasible integer points in $\mathcal{P}$.
Hence, we have only linear convergence of Frank-Wolfe on the continuous subproblems if the corresponding continuous solution is in $\mathcal{W}$.
In this case, we may use \cref{th:LinConvBPCG} in \cref{sec:AppendixLinConvergenceBPCG} to argue linear convergence.
For the \cref{eq:DOptDesignProblem} Problem, we could adapt the proof in \cite{zhao2025new} to the Blended Pairwise Conditional Gradient (BPCG) 
to argue linear convergence on all the subproblems, see \cref{prop:linconvdopt} in \cref{sec:AppendixLinConvergenceBPCG}. 
For the A-criterion and its log variant, we have only linear convergence on nodes where the continuous solution is in $\mathcal{W}$.
We cannot guarantee that this is the case for all subproblems at the nodes. 

\begin{figure}
    \centering
    \begin{subfigure}{0.3\textwidth}
        \centering
        \includegraphics[width=\textwidth]{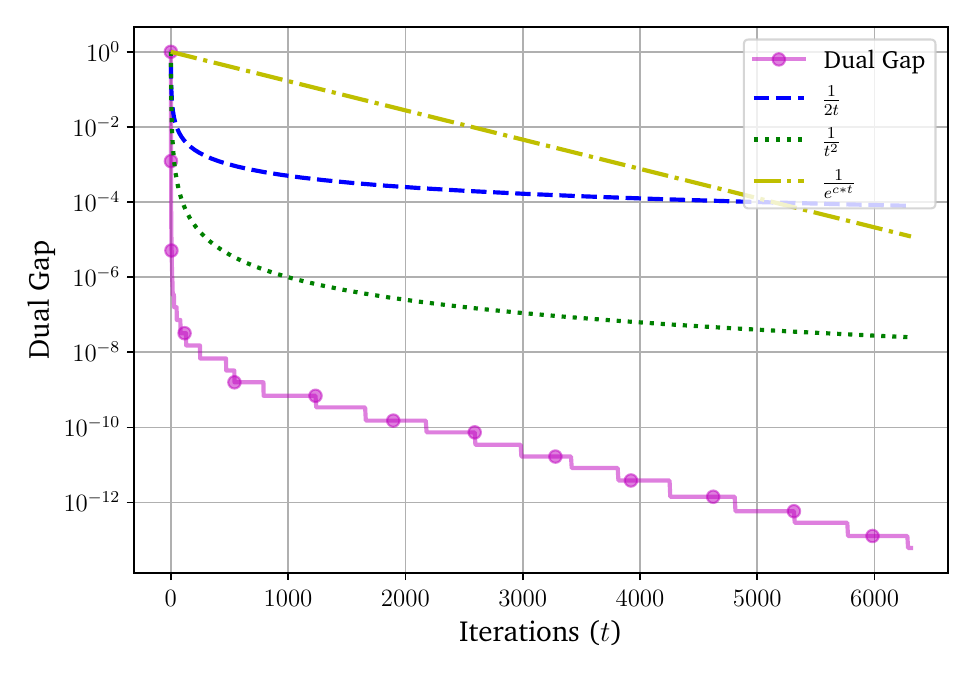}
        \caption{$A$-criterion}
    \end{subfigure}
    \begin{subfigure}{0.3\textwidth}
        \centering
        \includegraphics[width=\textwidth]{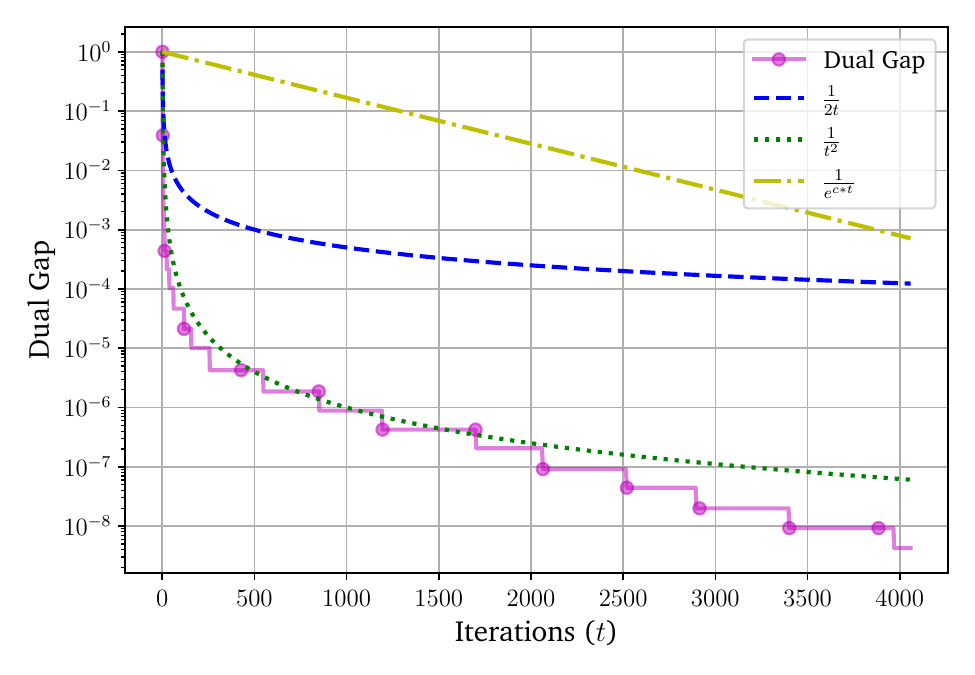}
        \caption{$D$-criterion}
    \end{subfigure}
    \begin{subfigure}{0.3\textwidth}
        \centering
        \includegraphics[width=\textwidth]{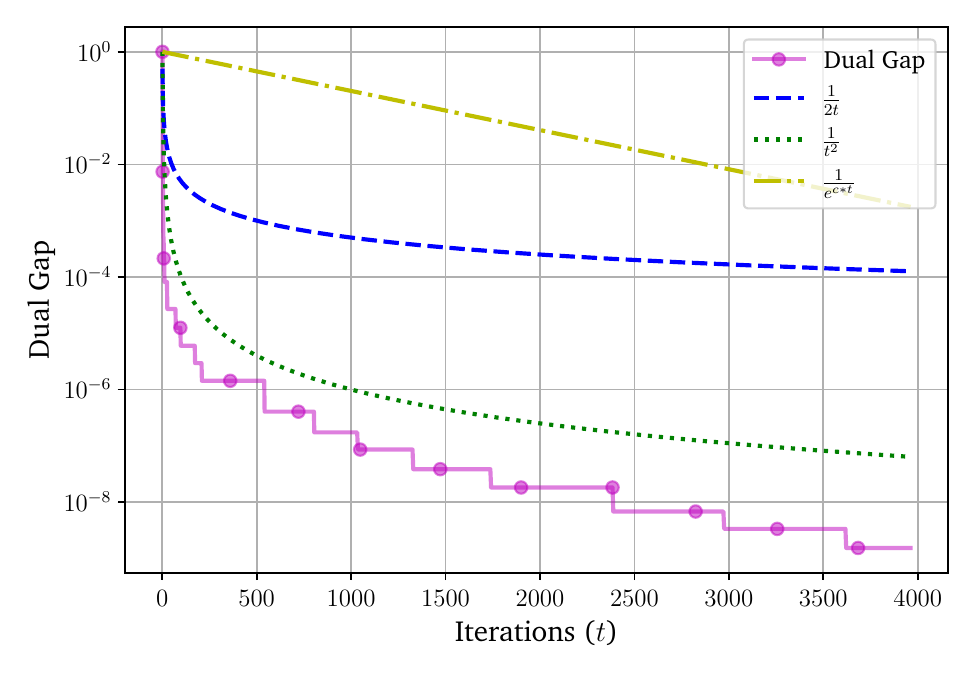}
        \caption{$\log A$-criterion}
    \end{subfigure}
    \caption{Trajectory of the dual gap at the root node compared to $1/2t$, $1/t^2$ and $1/e^{ct}$.}
    \label{fig:DualGapLogAandLogGTI}
\end{figure}

From our experiments, we have empirical evidence that we have linear convergence of the dual gap for all Optimal Problems.
In \cref{fig:DualGapLogAandLogGTI}, we show the trajectory of the dual gap at the root node for the $A$-criterion, $D$-criterion and $\log A$-criterion.
After a burn-in phase, we see that the trajectory becomes parallel to $1/e^{ct}$ for some constant $c>0$ which is a representation of the linear convergence.

In summary, we have definitely accelerated convergence rates for the Fusion Problems. 
In case of the Optimal Problems, it will depend on the position of the continuous solution, so the input experiment data.
In our own experiments, we could observe accelerated rates.

\begin{table}[!ht]
    \centering
    \label{tab:convergence_properties}
    \begin{tabular}{l | cc }
        \toprule
        \thead{Objective} & \thead{Convergence} & \thead{Linear\\Convergence} \\
        \midrule
        $f_F(\vx)=\log\det(X_C(\vx))$ & $\checkmark$ & $\checkmark$ \\
        $g_F(\vx)=\Tr(X_C(\vx)^{-p})$ & $\checkmark$  & $\checkmark$  \\
        $k_F(\vx)=\log\left(\Tr(X_C(\vx)^{-p})\right)$ & $\checkmark$  & $\checkmark$ \\
        \midrule
        $f(\vx)=-\log\det(X(\vx))$ & $\checkmark$ & $\checkmark$ \\
        $g(\vx)=\Tr(X(\vx)^{-p})$ & $\checkmark$  & Not guaranteed \\
        $k(\vx)=\log\left(\Tr(X(\vx)^{-p})\right)$ & Conjecture & Not guaranteed \\
        \bottomrule
    \end{tabular}
    \caption{Convergence of Frank-Wolfe on different objective functions.}
\end{table}

\section{Solution methods}\label{sec:SolutionMethods}

The main goal of this paper is to propose a new solution method for the Optimal and Fusion Problems under the A-criterion and D-criterion
based on the novel framework \package{} and assess its performance compared to several other convex MINLP approaches.
In the following, we introduce the chosen MINLP solvers and state the necessary conditions and possible reformulations that are needed. 

\paragraph{Branch-and-bound with Frank-Wolfe algorithms (Boscia).} 
The new framework introduced in \citet{hendrych2023convex} is implemented in the \julia{} package \package. 
It is a Branch-and-Bound (BnB) framework that utilizes Frank-Wolfe methods to solve the relaxations at the node level.
The Frank-Wolfe algorithm \citep{frank1956algorithm,braun2022conditional}, also called Conditional Gradient algorithm \citep{levitin1966constrained}, and its variants are first-order methods solving problems of the type:
\begin{align*}
    \min_{\vx \in \mathcal{X}} \; & f(\vx)
\end{align*}
where $f$ is a convex, Lipschitz-smooth function and $\mathcal{X}$ is a compact convex set.
These methods are especially useful if linear minimization problems over $\mathcal{X}$ can be solved efficiently.
The Frank-Wolfe methods used in \package{} are implemented in the \julia{} package \frankwolfe, see \citet{besanccon2022frankwolfe,besanccon2025improved}.

At each iteration $t$, Frank-Wolfe solves the linear minimization problem over $\mathcal{X}$ taking the current gradient as the linear objective, resulting in a vertex $v_t$ of $\mathcal{X}$. The next iterate $x_{t+1}$ is computed as a convex combination of the current iterate $x_t$ and the vertex $v_t$. Many Frank-Wolfe variants explicitly store the vertex decomposition of the iterate, henceforth called the \emph{active set}. We utilize the active set representation to facilitate warm starts in Boscia by splitting the active set when branching.

One novel aspect of \package{} is its use of a \emph{Bounded Mixed-Integer Linear Minimization Oracle (BLMO)} as the \emph{Linear Minimization Oracle (LMO)} in Frank-Wolfe. 
Typically, the BLMO is a MIP solver but it can also be a combinatorial solver. 
This leads to more expensive node evaluations but has the benefit that feasible integer points are found from the root node and the feasible region is much tighter than the continuous relaxation for many problems. 
In addition, Frank-Wolfe methods can be lazified, i.e.~calling the LMO at each iteration in the node evaluation can be avoided, see \citep{braun2017lazifying}.

In the case of OEDP, strong lazification is not necessary since the corresponding BLMO is very simple. The feasible region $\mathcal{P}$ is just the scaled probability simplex intersected with integer bounds, which is amenable to efficient linear optimization.
Given a linear objective $d$, we first assign $\vx=\vl$ to ensure that the lower bound constraints are met.
Next, we traverse the objective entries in increasing order, adding to the corresponding variable the value of $\max\{u_i-l_i, N - \sum(\vx)\}$.
This way, we ensure that both the upper-bound constraints and the knapsack constraint are satisfied.
The LMO over the feasible set can also be cast as a simple network flow problem with $m$ input nodes connected to a single output node, which must receive a flow of $N$ while the edges respect the lower and upper bounds.
The scaled probability simplex also allows for a simple heuristic taking a fractional point and outputting an integer feasible point on the probability simplex.

Due to the convexity of the objective, the difference $\innp{\nabla f}{x_t-v_t}$ is upper bounding the primal gap $f(x_t)-f(x^*)$ at each iteration. We call the quantity the dual gap (or the Frank-Wolfe gap). The dual gap can therefore be used as a stopping criterion. Frank-Wolfe's error adaptiveness can be exploited to a) solve nodes with smaller depth with a coarser precision and b) dynamically stop a node evaluation if the lower bound on the node solution exceeds the incumbent.

Observe that in contrast to the epigraph-based formulation approaches that generate many hyperplanes, this method works with the equivalent of a single supporting hyperplane given by the current gradient and moves this hyperplane until it achieves optimality.
It is known that once the optimal solution is found, a single supporting hyperplane can be sufficient to prove optimality (described e.g.~for generalized Benders in \cite{sahinidis1991convergence}). 
Finding this final hyperplane, however, may require adding many hyperplanes beforehand at suboptimal iterations.
In the case of the problems discussed in this paper, the constraint polytope is uni-modular.
Adding hyperplanes created from the gradient will not maintain this structure and in consequence, yields a numerically more challenging MIP.
Our approach, on the other hand, keeps the polytope and thereby its uni-modularity intact.

A new development during this work was the implementation of a domain oracle in \package.
As stated earlier, not all points in the feasible region are domain feasible for the objective functions.
Thus, we cannot start Boscia (and Frank-Wolfe) at a random point.
We therefore added the option of a \textit{domain oracle} and a \textit{domain point function}.
The former checks whether a given point is domain feasible, in our case by checking if the associated information matrix is regular.
The latter provides a domain feasible point given the feasible region and additional bound constraints from the node.
The returned point has to respect the global constraints as well as the node specific constraints.
The domain point function is called during branching if the splitting of the active set results in an invalid point.
A small projection problem is then solved, also using Frank-Wolfe, to find a new point within the intersection of the feasible region and the domain.
Note that we stop the Frank-Wolfe method once we are in the intersection.

For the line search problem within Frank-Wolfe, we utilize the Secant line search which has shown a superior performance for generalized self concordant functions, see \citet{hendrych2025secant}.

\paragraph{Outer approximation (SCIP OA).} Outer Approximation schemes are a popular and well-established way of solving MINLPs \citep{kronqvist2019review}. 
This approach requires an epigraph formulation of \eqref{eq:GeneralOptDesignProblem}:
\begin{align} \label{eq:EpiGraphFormulation}
    \tag{E-OEDP}
    \begin{split}
        \min_{t,\vx} \; & t \\
        \text{s.t. } & t \geq \log\left(\phi_p(\vx)\right) \\
        & \vx \in \mathcal{P} \\
        & t \in \R, \vx \in \Z^m.
\end{split}
\end{align}
This approach approximates the feasible region of \eqref{eq:EpiGraphFormulation} with linear cuts derived from the gradient of the non-linear constraints, in our case $\nabla f$. 
Note that this requires the information matrix $X(\vx)$ at the current iterate $\vx$ to be positive definite, otherwise an evaluation of the gradient is not possible or rather it will evaluate to $\infty$.
The implementation is done with the \julia{} wrapper of \scip{}, \citep{bestuzheva2021scip,bestuzheva2023enabling}. 
Note that generating cuts that prohibit points leading to singular $X(\vx)$, we will refer to them as \emph{domain cuts}, is non-trivial. 
Thus, we are applying this approach to the Fusion Problems only since their information matrix is always positive definite.

\paragraph{LP/NLP branch-and-bound with direct conic formulation (Direct Conic).}
Another Outer Approximation approach, as implemented in \pajarito{} \citep{CoeyLubinVielma2020}, represents the non-linearities as conic constraints.
This is particularly convenient in combination with the conic interior point solver \hypatia{} \citep{coey2022solving} as it implements the $\log\det$ cone (the epigraph of the perspective function of $\log\det$) directly.
\[
    \mathcal{K}_{\log\det} :=  \text{cl}\left\{ (u,v,W) \in \R\times\R_{>0}\times\pdCone^n \mid u \leq v\log\det(W/v)\right\} 
    \]
The formulation of \eqref{eq:DOptDesignProblem} then becomes 
\begin{align*} 
    \begin{split}
        \max_{t,\vx} \; & t \\
        \text{s.t. } & (t,1,X(\vx)) \in \mathcal{K}_{\log\det} \\
        & \vx \in \mathcal{P} \\
        & t \in \R, \vx \in \Z^m.
\end{split}
\end{align*}
For the representation of the trace inverse, we utilize the dual of the separable spectral function cone \citep[Section~6]{coey2022performance}:
\begin{equation*}
    \mathcal{K}_{\text{sepspec}} := \text{cl}\left\{(u,v,w)\in\R\times\R_{>0}\times \text{int}(\mathcal{Q}) \mid u \geq v \varphi(w/v) \right\}
\end{equation*}
For our purposes, $\mathcal{Q}$ is the PSD cone and the spectral function $\varphi$ is the negative square root whose convex conjugate is precisely the trace inverse, see \citep[Table~1]{coey2023conic} and \citep[Collary 1]{pilanci2015sparse}\footnote{For further details, see this discussion \cite{formulationforum}.}.
\begin{equation*}
    \mathcal{K}^*_{\text{sepspec}} := \text{cl}\left\{(u,v,w)\in\R\times\R_{>0}\times \text{int}(\sdCone^n) \mid v \geq u/4 \Tr((w/u)^{-1}) \right\}
\end{equation*}
The conic formulation of the \eqref{eq:AOptDesignProblem} is therefore
\begin{align*} 
    \begin{split}
        \min_{t,\vx} \; & 4t \\
        \text{s.t. } & (1,t,X(x)) \in \mathcal{K}^*_{\text{sepspec}} \\
        & \vx \in \mathcal{P} \\
        & t \in \R, \vx \in \Z^m.
\end{split}
\end{align*}

Additionally, the conic formulation allows the computation of domain cuts for $\vx$. Hence, this solver can be used on all problems. Note that we use \texttt{HiGHS}~\citep{huangfu2018parallelizing} as a MIP solver within \pajarito.

\paragraph{LP/NLP branch-and-bound with second-order cone formulation (SOCP).}
The conic solver from the previous approach was developed fairly recently. 
A previous conic approach was introduced first in \cite{sagnol2011computing} for the continuous case and later extended to the mixed-integer case in \cite{sagnol2015computing}.
They showed OEDP under the A-criterion and D-criterion can be formulated as a second-order cone program (SOCP).
The SOCP model for \eqref{eq:DOptDesignProblem} is given by
\begin{align*}
    \begin{split}
        \max_{\vw,Z, T, J} \; & \prod\limits_{j=1}^n \left(J_{jj}\right)^{\frac{1}{n}} \\
        \text{s.t. } & \sum\limits_{i=1}^m A_i Z_i = J \\
        & J \text{ is a lower triangle matrix}\\
        & \norm{Z_i \ve_j }^2 \leq T_{ij} \vw_i \quad \forall i \in [m], j \in [n]\\
        & \sum\limits_{i=1}^m T_{ij} \leq J{jj} \quad \forall j \in [n]\\
        & T_{ij} \geq 0 \quad \forall i \in [m], j \in [n]\\
        & \sum\limits_{i=1}^m \vw_i = N, \; \vw \geq 0
 \end{split}
\end{align*}
where $J$ is a $n\times n$ matrix, $Z$ and $T$ are $m\times n$ matrices.
Observe that the problem size has increased significantly with both the number of variables and number of contraints being $\geq m(2n+1)$.

The SOCP formulation for \eqref{eq:AOptDesignProblem} is 
\begin{align*}
    \begin{split}
        \min_{\vw,\boldsymbol{\mu}, Y} \; & \sum\limits_{i=1}^m \mu_i \\
        \text{s.t. } & \sum\limits_{i=1}^m A_iY_i = \sum\limits_{i=1}^m \boldsymbol{\mu}_i \\
        &  \norm{Y_j}^2\leq \boldsymbol{\mu}_j \vw_i \quad \forall i \in [m]\\
        & \boldsymbol{\mu} \geq 0,  \; \vw \geq 0\\
        & \sum\limits_{i=1}^m \vw_i = N
\end{split}
\end{align*}
where $Y$ is a $m\times n$ matrix and $\boldsymbol{\mu}$ is a $m\times 1$ vector.
The problem size does not increase as much, with $m(2+n)$ variables and $3m+2$ constraints.

To solve the SOCP formulation, we again utilize \pajarito{} with \texttt{HiGHS} as a MIP solver and \hypatia{} as a conic solver. 
For the terms involving the norm, we make use of \hypatia{}'s \texttt{EpiNormEuclCone} and \texttt{EpiPerSquareCone} cones which do not require us to deconstruct the norm terms into second order cones.

\paragraph{A custom branch-and-bound for \eqref{eq:GeneralOptDesignProblem} (Co-BnB).}
The most general solver strategy for \eqref{eq:GeneralOptDesignProblem} with matrix means criteria was introduced in \cite{ahipacsaouglu2021branch}. 
Like \package, it is a Branch-and-Bound-based approach with a first-order method to solve the node problems. The first-order method in question is a coordinate-descent-like algorithm. 
As the termination criterion, this method exploits that the objective function is a matrix mean and shows the connection of the resulting optimization problem
\begin{align}
    \tag{M-OEDP}
    \label{eq:MatrixMeansOptimalDesign}
    \begin{split}
        \max_{\vw} \, & \, \log\left(\phi_p(X(\vw))\right) \\
        \text{s.t. } \, & \sum\limits_{i=1}^m w_i = 1  \\
            & \vw \geq 0 \\
            & w_iN \in \Z \, \forall i\in[m] 
    \end{split}
\end{align}
to the generalization of the Minimum Volume Enclosing Ellipsoid Problem (MVEP) \citep{ahipacsaouglu2021branch}. 
The variables $w$ can be interpreted as a probability distribution and the number of times the experiments are to be run is $wN$. Concerning \eqref{eq:GeneralOptDesignProblem}, one can say $x=wN$.

Note first that we have reimplemented the method in \julia{} and modified it so that we solve the formulation as depicted in \eqref{eq:GeneralOptDesignProblem}. 
Secondly, we have improved and adapted the step size rules within the first-order method, see \cref{sec:AppendixCustomBnB}.
Note further that the solver was developed for instances with a plethora of experiments and very few parameters. 
The solver employs the simplest Branch-and-Bound tree, i.e.~with the most fractional branching rule and traverses the tree with respect to the minimum lower bound.  
In the next section, we will see that the method works well in cases where $n$ is small but struggles if the value of $n$ increases.

\section{Computational experiments}\label{sec:ComputationalExperiments}

In this section, we present the computational experiments for the Optimal Problem and Fusion Problem, both under the A- and D-criterion, respectively. The resulting problems will be referred to as the \emph{A-Fusion Problem (AF)}, \emph{D-Fusion Problem (DF)}, \emph{A-Optimal Problem (AO)}, and \emph{D-Optimal Problem (DO)}. 

Furthermore, we run the Optimal and Fusion Problems under the GTI-criterion for different values of $p$ using permissible solvers. Lastly, we showcase the performance of \package{} and Co-BnB on problems of large dimension, i.e.~number of variables between 300 and 500.

\paragraph{Experimental setup.}
For the instance generation, we choose the number of experiments $m \in \{50,60,80,\linebreak 100,120\}$, the number of parameters $n\in \{\lfloor m/4\rfloor, \lfloor m/10 \rfloor\}$, and the number of allowed experiments $N = \lfloor 1.5 n\rfloor$ for the Optimal Problems and 
$N \in [m/20, m/3]$ for the Fusion Problems.
The lower bounds are zero. Note that for the Fusion Problems, the fixed experiments are encoded in a separate matrix. The upper bounds are randomly sampled between 1 and $N/3$ for \eqref{eq:AOptDesignProblem} and \eqref{eq:DOptDesignProblem}. In the Fusion case, they are sampled between 1 and $m/10$. 
We generate both independent and correlated experiment data. Also, note that the matrices generated are dense. 
Seeds used are $[1,2,3,4,5]$. In total, there are 50 instances for each combination of problem and data set.

Experiments were run on a cluster equipped with Intel Xeon Gold 6338 CPUs running at 2 GHz and a one-hour time limit. 
The Julia version used is \textit{1.10.2}, \texttt{Boscia.jl} \textit{v0.1.35}, \texttt{FrankWolfe.jl} \textit{v0.4.13}, \texttt{HiGHS.jl} \textit{v1.15.0}, 
\texttt{Hypatia.jl} \textit{v0.7.4}, \texttt{Pajarito.jl} \textit{v0.8.2}, \texttt{SCIP.jl} \textit{v0.12.3}.
Note that we use the unregistered package \texttt{PajaritoExtras.jl} to model the dual of the separable spectral function cone.
The source code is hosted on GitHub\footnote{\url{https://github.com/ZIB-IOL/OptimalDesignWithBoscia}}.

\paragraph{Start solution.}
Note that both the objectives \eqref{eq:DOptDesignProblem} and \eqref{eq:AOptDesignProblem} are only well defined if the information matrix $X(\vx)$ has full rank. This is the case for the Fusion Problem, not necessarily for the Optimal Problem. 
Both \package{} and Co-BnB require a feasible starting point $\vz_0$. 
For its construction, we find a set $S\subset[m]$ of $n$ linearly independent experiments, i.e.~$n$ linearly independent rows of $A$.
Assign those experiments their upper bound. If the sum $\sum \vz_0$ exceeds $N$, remove 1 from the experiment with the largest entry. 
If the sum is less than $N$, pick an experiment in $[m]\backslash S$ at random and assign it as many runs as possible. Repeat until the sum is equal to $N$.
Note that due to the monotonic progress of both first-order methods, the current iterate will never become domain infeasible, i.e.~singular. 

\paragraph{Results.}
An overview of the results of the computational experiments is given in \cref{tab:SummaryByDifficulty} and extensive results can be found in \cref{sec:AppendixComEx}. 
The new framework \package{} shows a superior performance compared to the other methods both in terms of time and number of instances solved to optimality.
In comparison to the Outer Approximation methods, for example, it solves nearly twice as many instances. 

\texttt{Co-BnB} is the only competitive method. 
It mostly solves fewer instances than \package{}, the exception being the A-Fusion Problem with both data sets. 
With independent data, Co-BnB is faster solving the same number of instances.
In case of the correlated data set, Co-BnB solves roughly twice as many instances as \package.
In \cref{tab:SummaryByDifficultyExt}, we see that \package{} and \texttt{Co-BnB} have similar magnitudes in number of nodes for both A-Fusion Problems 
whereas for all other problems \package{} has significantly fewer nodes than \texttt{Co-BnB}.
As stated previously, node evaluation in \package{} is more computationally expensive.
The fact that the magnitudes of nodes is similar points to the optimal solution being strictly in the interior. 
Note that in general, \texttt{Co-BnB} fares well for the instances where $n=\lfloor m/10\rfloor$ as it was designed with such problems. 
It struggles for the instances where $n=\lfloor m/4\rfloor$.
In \cref{fig:DualGapComparison}, we see that the absolute gap of \texttt{Co-BnB} increasing drastically as the fraction between number of parameters
and number experiments grow larger. 
While \package{} also shows an increase, it is much more moderate.
\texttt{Co-BnB} implements no advanced Branch-and-Bound tree specifications like a better traverse strategy or branching strategy.
A greater value of $N$ naturally increases the size of the tree and the number of nodes to be processed.
\package{} has the advantage here since it finds many integer feasible points while solving the relaxations which have the potential to improve the incumbent. 
A better incumbent, in turn, lets us prune non-improving nodes early on. Overall, fewer nodes have to be investigated by \package. 
\begin{figure}[hb]
    \centering
    \begin{subfigure}{.4\textwidth}
        \centering
        \includegraphics[width=1\linewidth]{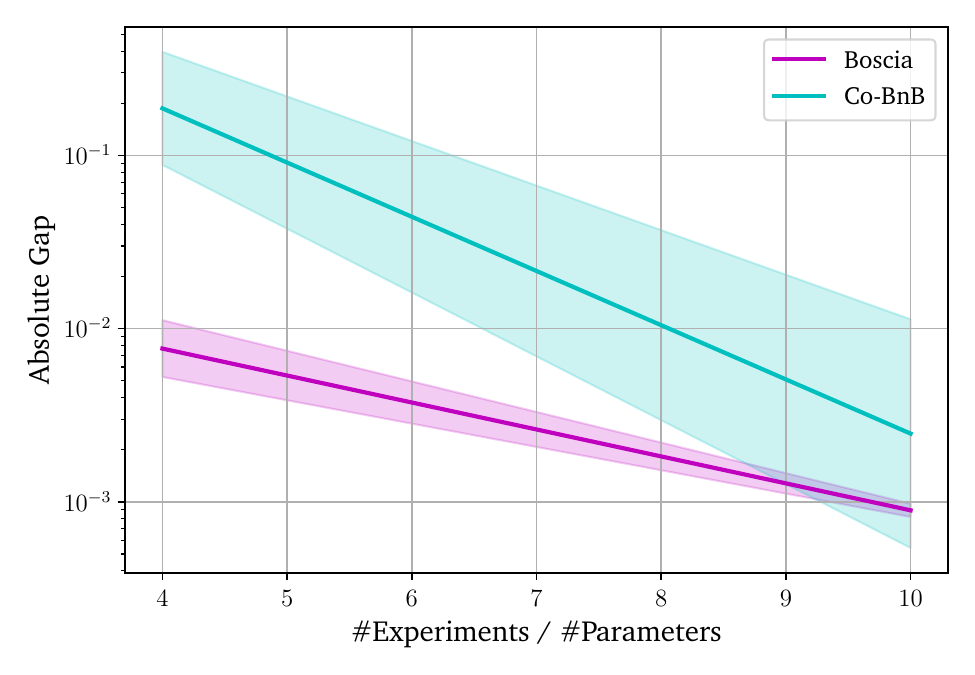}
        \caption{A-Optimal Problem with independent data}
        \label{fig:AINDDualGap}
    \end{subfigure}%
    \begin{subfigure}{.4\textwidth}
        \centering      
        \includegraphics[width=1\linewidth]{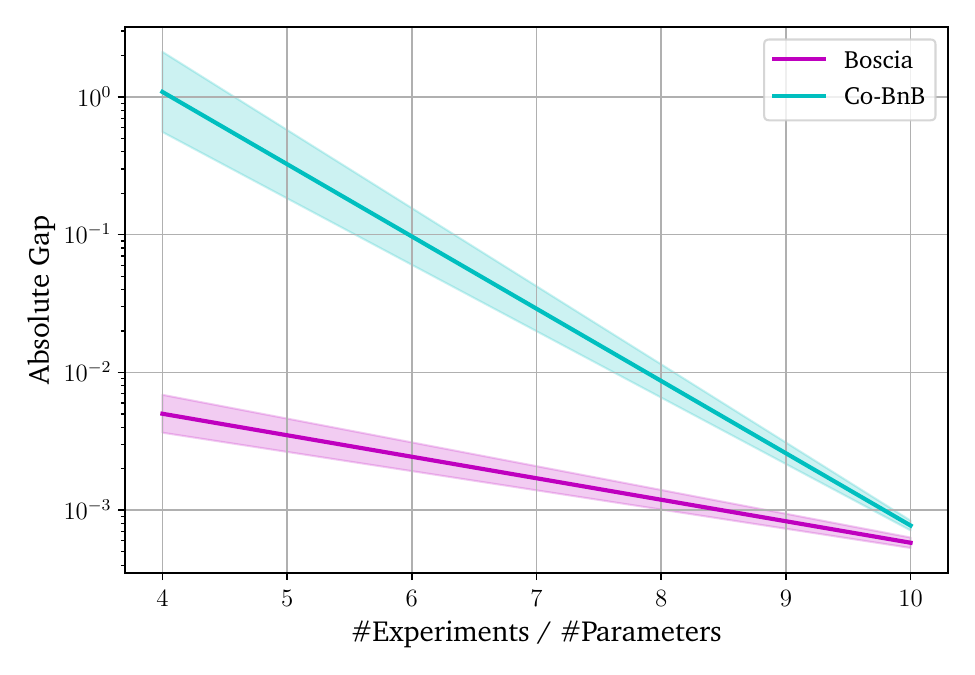}
        \caption{D-Optimal Problem with independent data}
        \label{fig:DINDDualGap}
    \end{subfigure}
    \begin{subfigure}{.4\textwidth}
        \centering
        \includegraphics[width=1\linewidth]{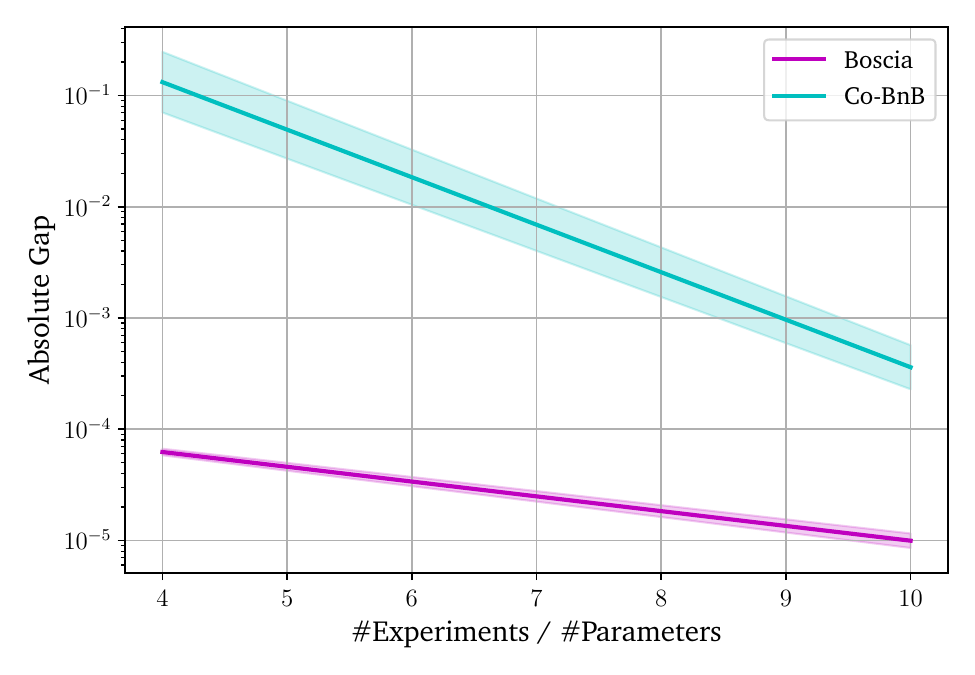}
        \caption{A-Optimal Problem with correlated data}
        \label{fig:ACORRDualGap}
    \end{subfigure}%
    \begin{subfigure}{.4\textwidth}
        \centering
        \includegraphics[width=1\linewidth]{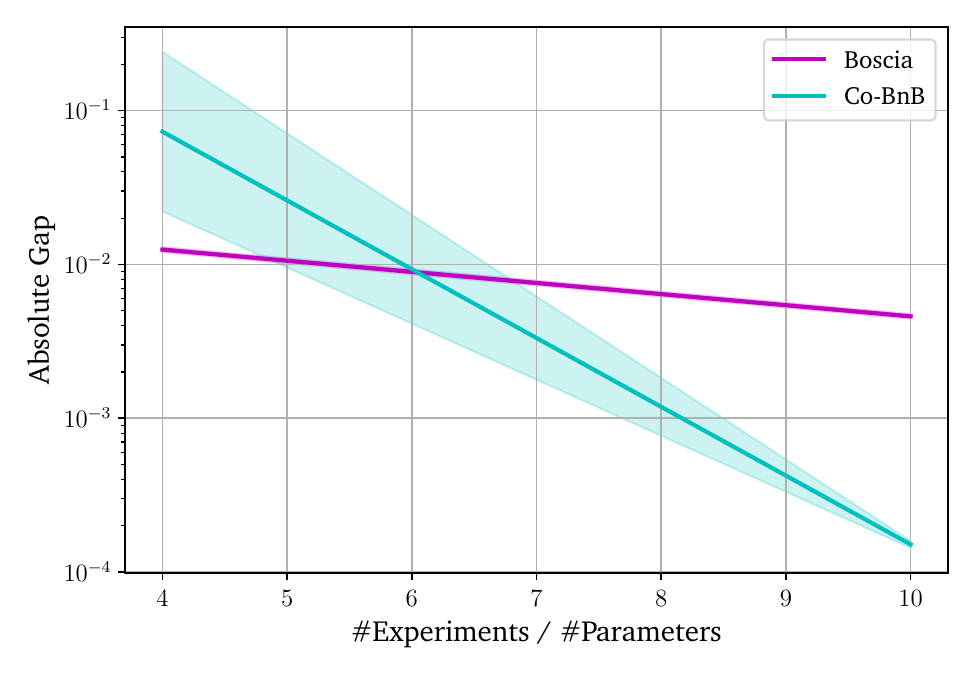}
        \caption{D-Optimal Problem with correlated data}
        \label{fig:DCORRDualGap}
    \end{subfigure}
    \caption{Geometric mean and geometric standard deviation of the absolute gap between the incumbent and lower bound over the conditioning of the problem.}
    \label{fig:DualGapComparison}
\end{figure}

In terms of time, \package{} also shows promising results, especially for instances of larger scale. 
For small-scale instances, the Outer Approximation approaches are fast, in some cases faster than \package.
This can be observed in \cref{fig:TerminationOverTimeOptimal,fig:TerminationOverTimeFusion}.
The plots also highlight that the Outer Approximation Schemes solve fewer instances than the two Branch-and-Bound Approaches.
It should be noted that \texttt{SCIP + OA} simply timed out often. 
On the other hand, the MIP solvers for the conic formulations with \pajarito{} often struggled to solve the resulting MIP problems. 
\texttt{SOCP}, for example, only solved about a quarter of the instances of the D-Fusion Problem and D-Optimal Problem.
In addition, the many solutions provided by \texttt{SOCP} in case of a timeout were not always domain feasible, and thus, the original objective function could not be evaluated at the provided solutions.
The superiority of the Branch-and-Bound approaches stems from the fact that the relaxations of both of them keep the simple feasible region intact.
Whereas the Outer Approximation schemes add many additional constraints, i.e.~cuts, resulting in larger LPs to be solved, see \cref{tab:SummaryByDifficultyExt} for the average number of cuts.
Furthermore, these cuts are dense, i.e.~have many nonzero entries, leading to an increased difficulty for the LP solvers.

\begin{figure}
    \centering
\begin{subfigure}{.5\textwidth}
  \centering
  \includegraphics[width=1.0\linewidth]{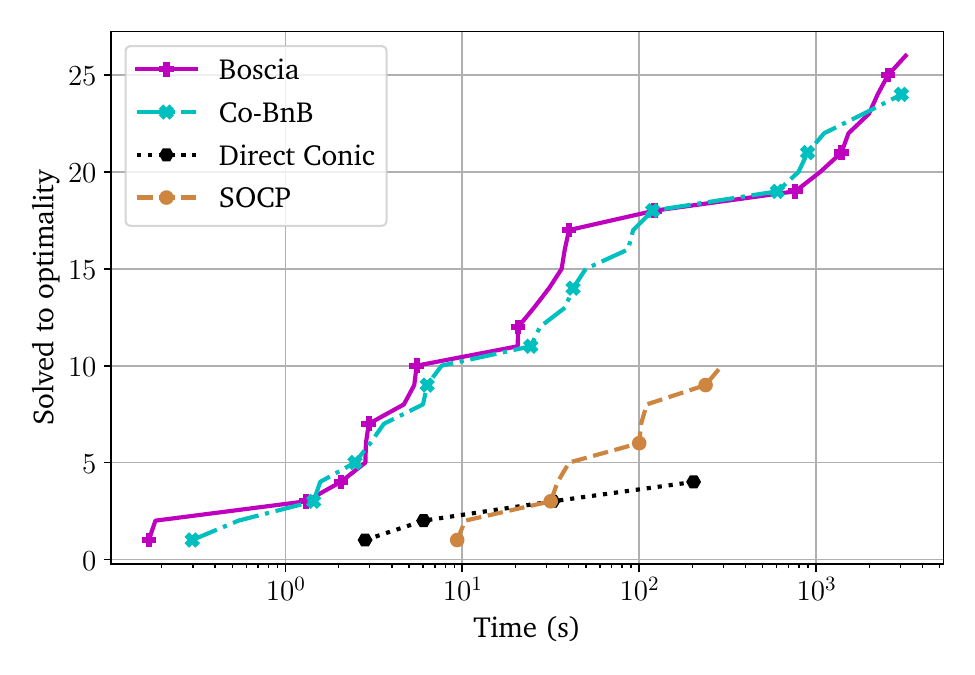}
  \caption{A-Optimal Problem with independent data}
  \label{fig:AINDTermination}
\end{subfigure}%
\begin{subfigure}{.5\textwidth}
    \centering
    \includegraphics[width=1.0\linewidth]{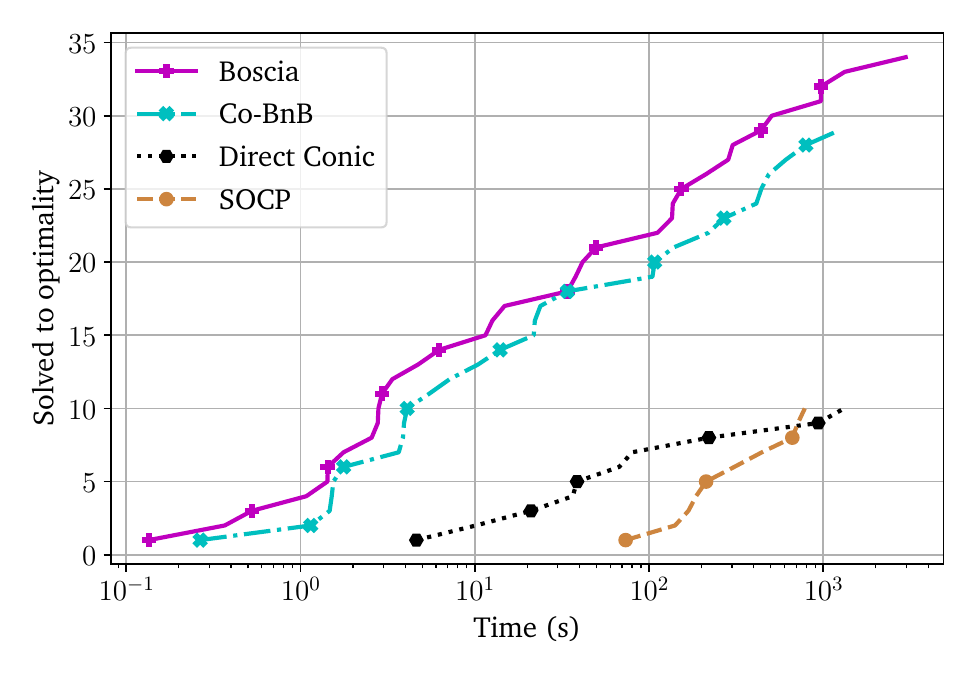}
    \caption{D-Optimal Problem with independent data}
    \label{fig:DINDTermination}
  \end{subfigure}
\begin{subfigure}{.5\textwidth}
  \centering
  \includegraphics[width=1.0\linewidth]{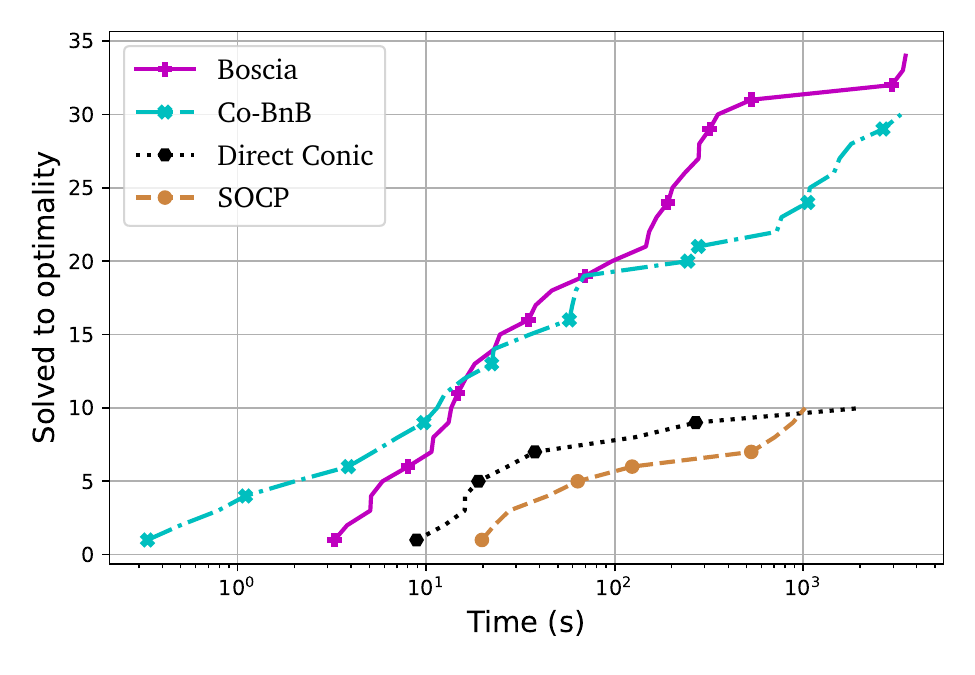}
  \caption{A-Optimal Problem with correlated data}
  \label{fig:ACORRTermination}
\end{subfigure}%
\begin{subfigure}{.5\textwidth}
  \centering
  \includegraphics[width=1.0\linewidth]{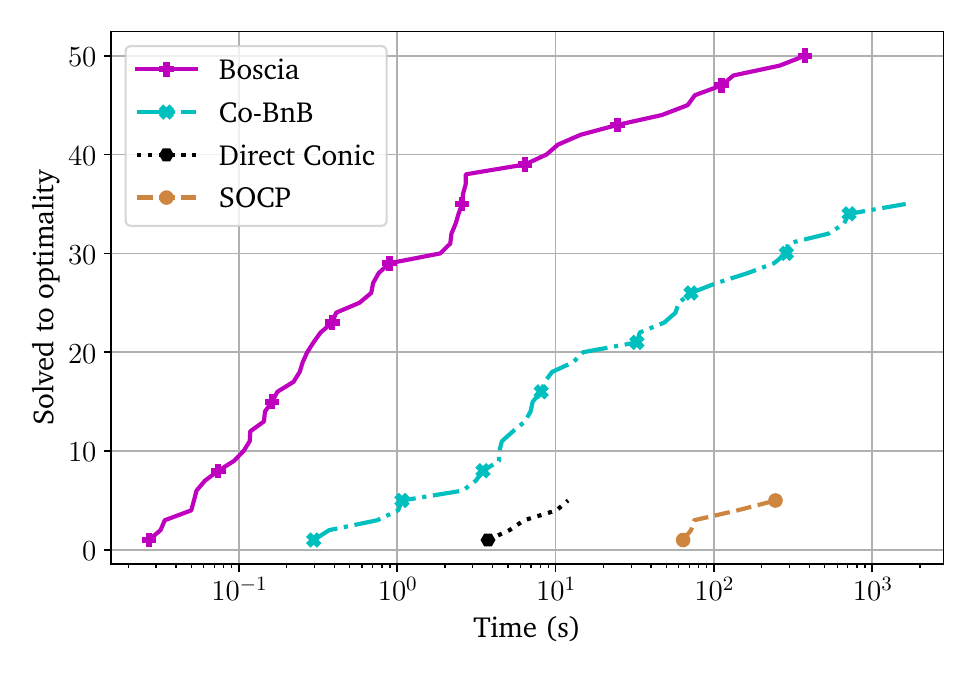}
  \caption{D-Optimal Problem with correlated data}
  \label{fig:DCORR  Termination}
\end{subfigure}
\caption{The number of instances solved to optimality over time for the A-Optimal Problem and D-Optimal Problem with both data sets. The upper plots are with the inpdendent data set, the lower ones with correlated data. Note that \texttt{SCIP+OA} is not applicable on the Optimal Problems.}
\label{fig:TerminationOverTimeOptimal}
\end{figure}

\begin{figure}
    \centering
\begin{subfigure}{.5\textwidth}
  \centering
  \includegraphics[width=1.0\linewidth]{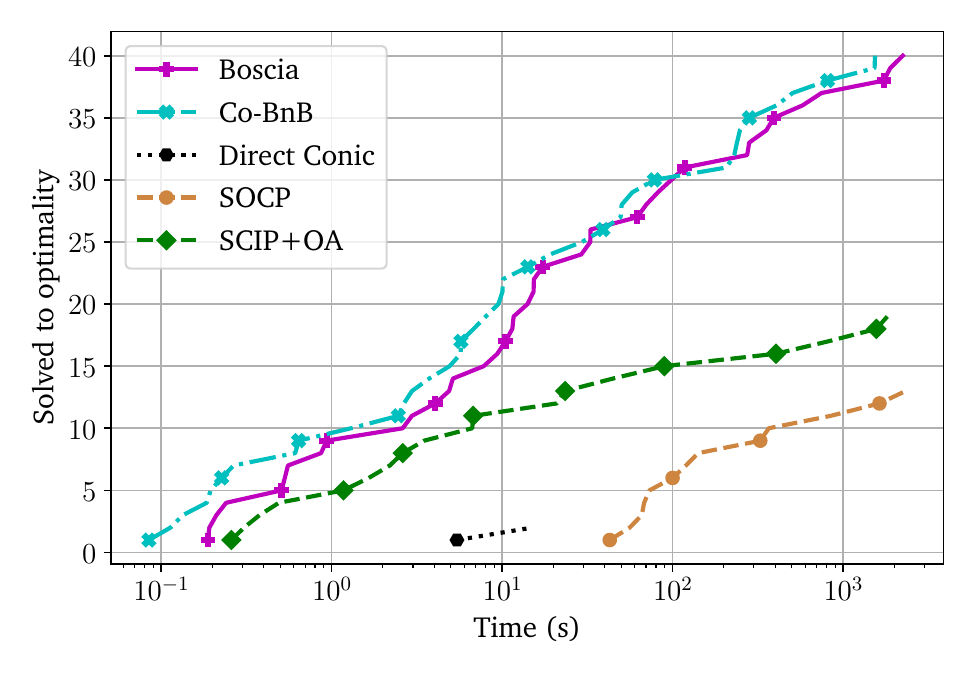}
  \caption{A-Fusion Problem with independent data}
  \label{fig:AFINDTermination}
\end{subfigure}%
\begin{subfigure}{.5\textwidth}
    \centering
    \includegraphics[width=1.0\linewidth]{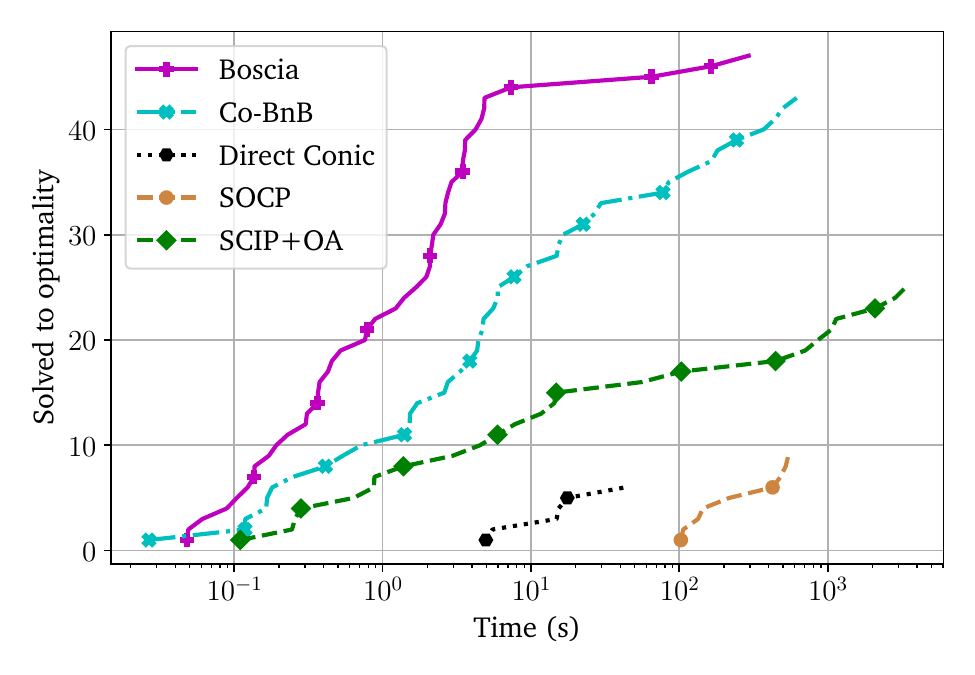}
    \caption{D-Fusion Problem with independent data}
    \label{fig:DFINDTermination}
  \end{subfigure}
\begin{subfigure}{.5\textwidth}
  \centering
  \includegraphics[width=1.0\linewidth]{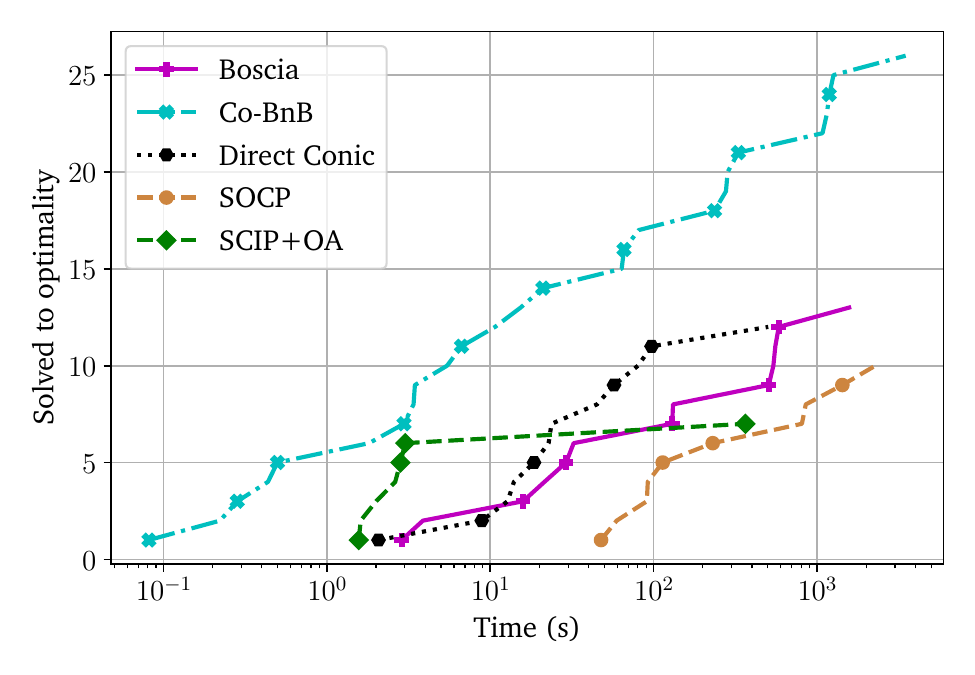}
  \caption{A-Fusion Problem with correlated data}
  \label{fig:AFCORRTermination}
\end{subfigure}%
\begin{subfigure}{.5\textwidth}
  \centering
  \includegraphics[width=1.0\linewidth]{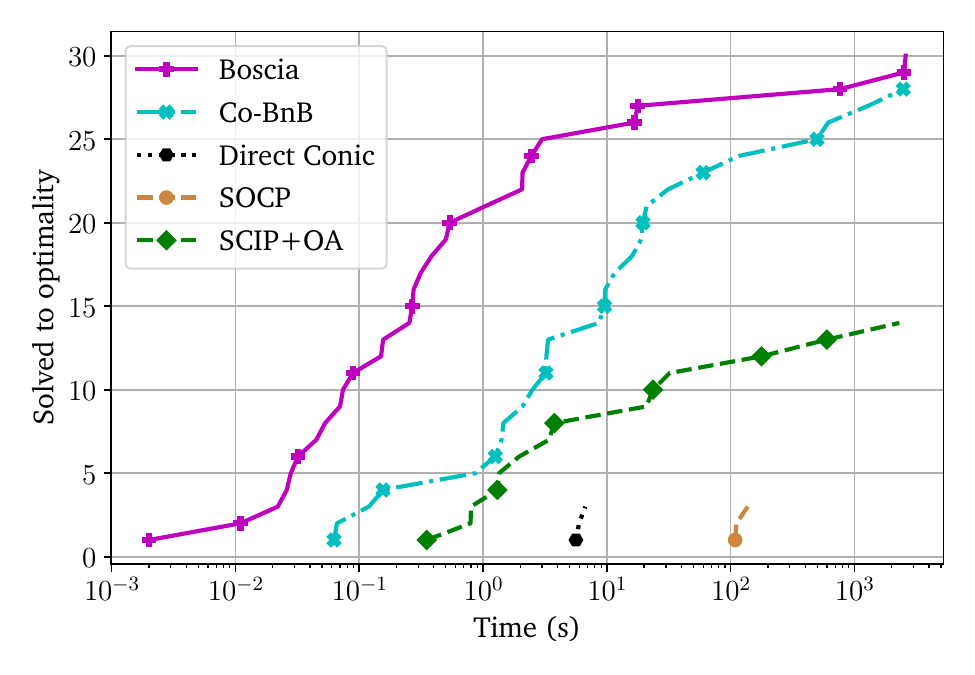}
  \caption{D-Fusion Problem with correlated data}
  \label{fig:DFCORRTermination}
\end{subfigure}
\caption{The number of instances solved to optimality over time for the A-Fusion Problem and D-Fusion Problem with both data sets. The upper plots are with the independent data set, the lower ones with correlated data.}
\label{fig:TerminationOverTimeFusion}
\end{figure}

For the \eqref{eq:GTIOptDesignProblem} Problem, the results are given in \cref{tab:SummaryGTI}.
Note that we used \package{} on both the \eqref{eq:GTIOptDesignProblem} and the \eqref{eq:logGTIOptDesignProblem}, the latter titled as "Boscia Log" in the table.
\package{} also shows a better performance here than \texttt{Co-BnB} and \texttt{SCIP + OA}. 
Interestingly, the \eqref{eq:logGTIOptDesignProblem} objective seems to lead to a better conditioning of the problem. 
This further motivates an investigation into generalized self-concordance of those objectives.

In \cref{tab:SummaryLongRuns}, we see the relative and absolute gap of \package{} and \texttt{Co-BnB} for instances with sizes from 300 and 500 variables.
\package{} showcases better gap values especially for the problems under the D-criterion.
Note that almost none of the instances were solved to optimality.

We performed a hyperparameter search for \package{} the results of which can be found in \cref{tab:SummaryByDifficultySettings1,tab:SummaryByDifficultySettings2,tab:SummaryByDifficultyLinesearch}.

Aside from the performance comparison of the solvers, we investigate how the problems themselves compare to each other and if the difficulty of the instances is solver-dependent. 

It can be observed in \cref{tab:SummaryByDifficulty} that most solvers solve fewer instances under the A-criterion.
The notable exception is \texttt{SOCP} on the A-Fusion Problem where it solves more instances to optimality compared to the D-criterion. 

Taking a look at some example contour plots shown in \cref{fig:ContourPlots}, we observe that the contour lines for the A-criterion are steeper than those of the D-criterion for both the Optimal Problem and the Fusion Problem, respectively. 
This points to the conditioning increasing. 

\begin{figure}[!ht]
    \centering
\begin{subfigure}{.4\textwidth}
  \centering
  \includegraphics[width=.7\linewidth]{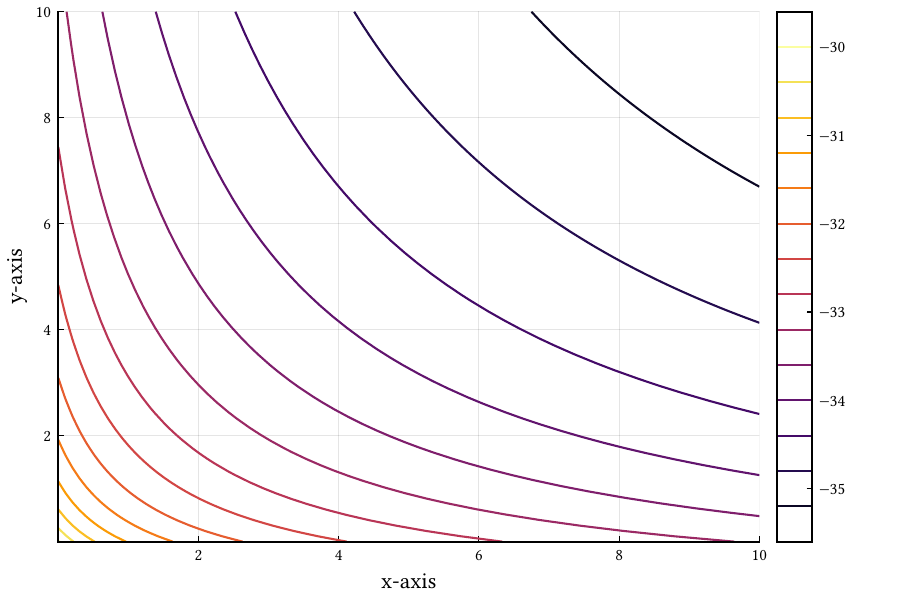}
  \caption{D-Optimal Correlated}
  \label{fig:ContourLogDetOptimalCORR}
\end{subfigure}%
\begin{subfigure}{.4\textwidth}
  \centering
  \includegraphics[width=.7\linewidth]{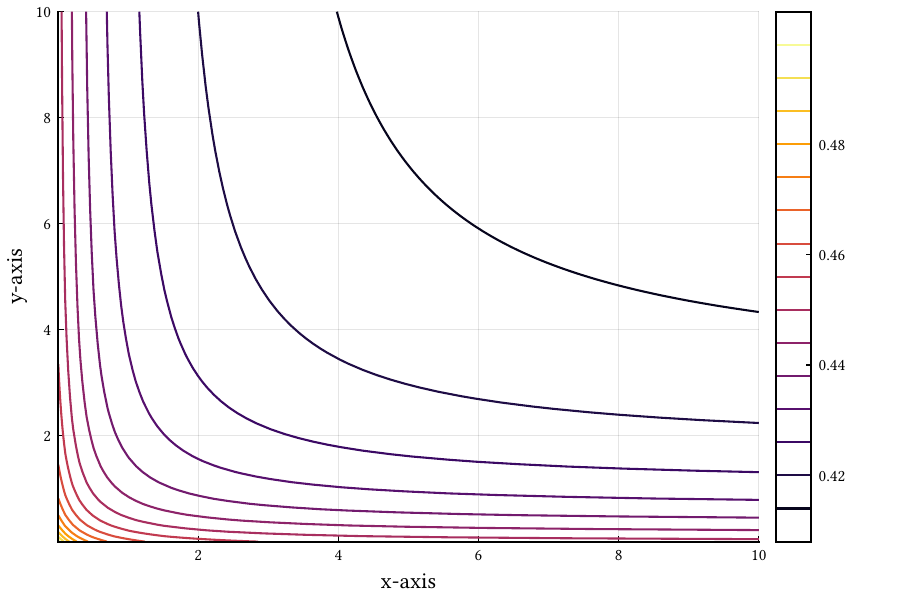}
  \caption{A-Optimal Correlated}
  \label{fig:ContourTrInvOptimalCORR}
\end{subfigure}

\begin{subfigure}{.4\textwidth}
  \centering
  \includegraphics[width=.7\linewidth]{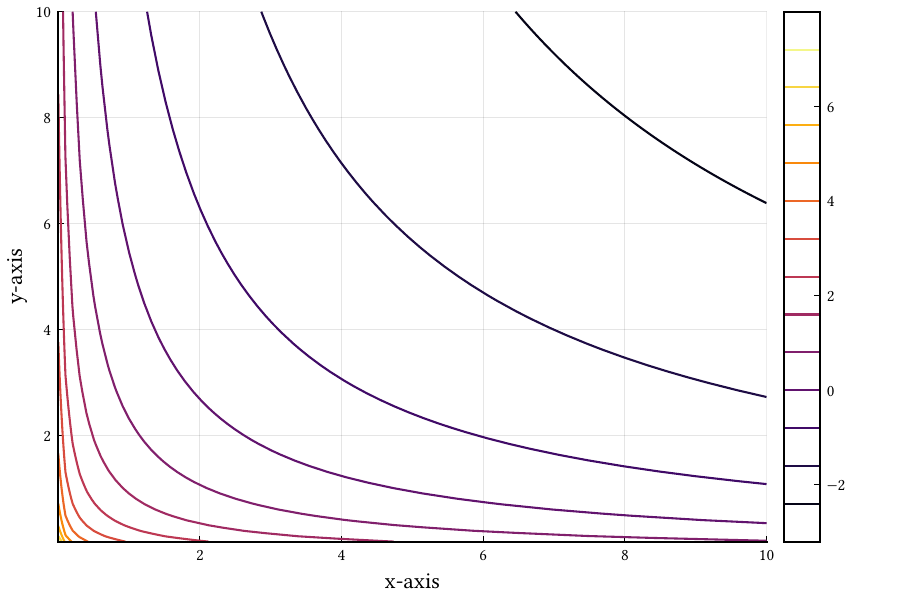}
  \caption{D-Fusion Independent}
  \label{fig:ContourLogDetFusionIND}
\end{subfigure}%
\begin{subfigure}{.4\textwidth}
  \centering
  \includegraphics[width=.7\linewidth]{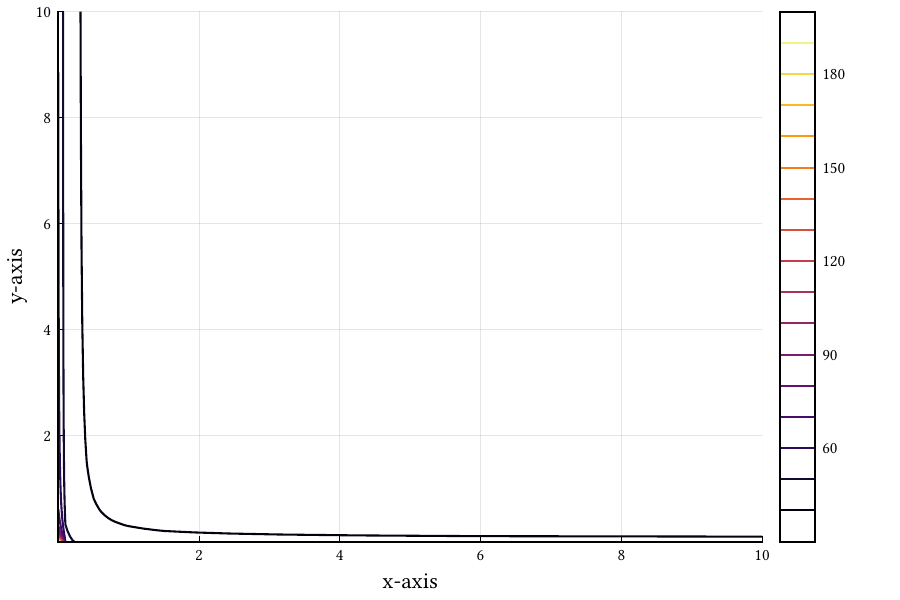}
  \caption{A-Fusion Independent}
  \label{fig:ContourTrInvFusionIND}
\end{subfigure}
\caption{Example contour plots in two dimensions for objectives of both Optimal Problems with correlated data and both Fusion Problems with independent.}
\label{fig:ContourPlots}
\end{figure}

In terms of the data, one could assume that all problems would be easier to solve with independent data. 
Noticeably, this is not the case, rather, it depends on the problem type. 
The Fusion Problems are easier to solve with independent data, the Optimal Problems are more often solved with the correlated data.
\Crefrange{fig:ProgressPlotsDOIND}{fig:ProgressPlotsAFCORR} depict the progress of the incumbent, lower bound, and dual gap within \package{} for selected instances of each combination of problem and data type. 

Interestingly, the independent data leads to proof of optimality, i.e.~the optimal solution is found early on and the lower bound has to close the gap, regardless of the problem, see \cref{fig:ProgressPlotsDOIND,fig:ProgressPlotsAFIND}. 
The difference in problems has, however, an impact on how fast the lower bound can catch up with the incumbent. 
In the case of the Optimal Problem there is likely a larger region around the optimal solution where the corresponding points/designs $\vx$ provide roughly the same information. 
These other candidates have to be checked to ensure the optimality of the incumbent and thus the solving process slows down.
Note that \package{} can utilize strong convexity and sharpness of the objective to improve the lower bound.
However, the improvement via sharpness is currently only applicable if the optimum is unique which is not the case.

On the other hand, the correlated data leads to solution processes that are very incumbent-driven, i.e.~most improvement on the dual gap stems from the improving incumbent, not from the lower bound, as seen in \cref{fig:ProgressPlotsDOCORR,fig:ProgressPlotsAFCORR}.
Incumbent-driven solution processes can be identified by the dual gap making sudden jumps and the absence of (a lot of) progress between these jumps.
As before, the solution process speed depends on the Problem. 
In \cref{fig:ProgressPlotsDOCORR}, the dual gap makes big jumps throughout most of the solving process, in contrast to the dual gap \cref{fig:ProgressPlotsAFCORR}.
This indicates that the optimal solution of the Fusion Problem is strictly in the interior and thus, will not be found early as a vertex of a relaxation.
The key ingredient for improvement in this case will be the incorporation of more sophisticated primal heuristics in \package{}.

Overall, \package{} shows a very promising performance and further research and development will be performed to improve the performance even further.

\begin{figure}[h]
    \centering
    \begin{subfigure}[b]{0.9\textwidth}
       \includegraphics[width=1\linewidth]{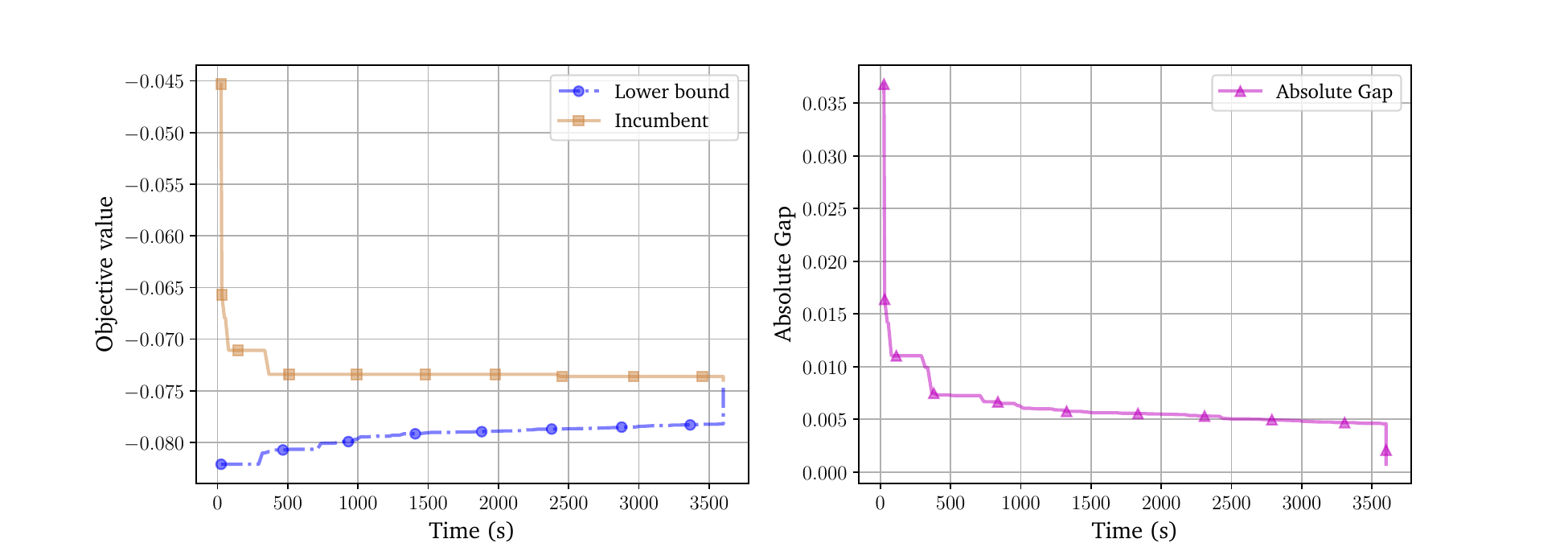}
       \caption{D-Optimal Problem with independent data and $n=10$}
       \label{fig:ProgressPlotsDOIND} 
    \end{subfigure}
    
    \begin{subfigure}[b]{0.9\textwidth}
       \includegraphics[width=1\linewidth]{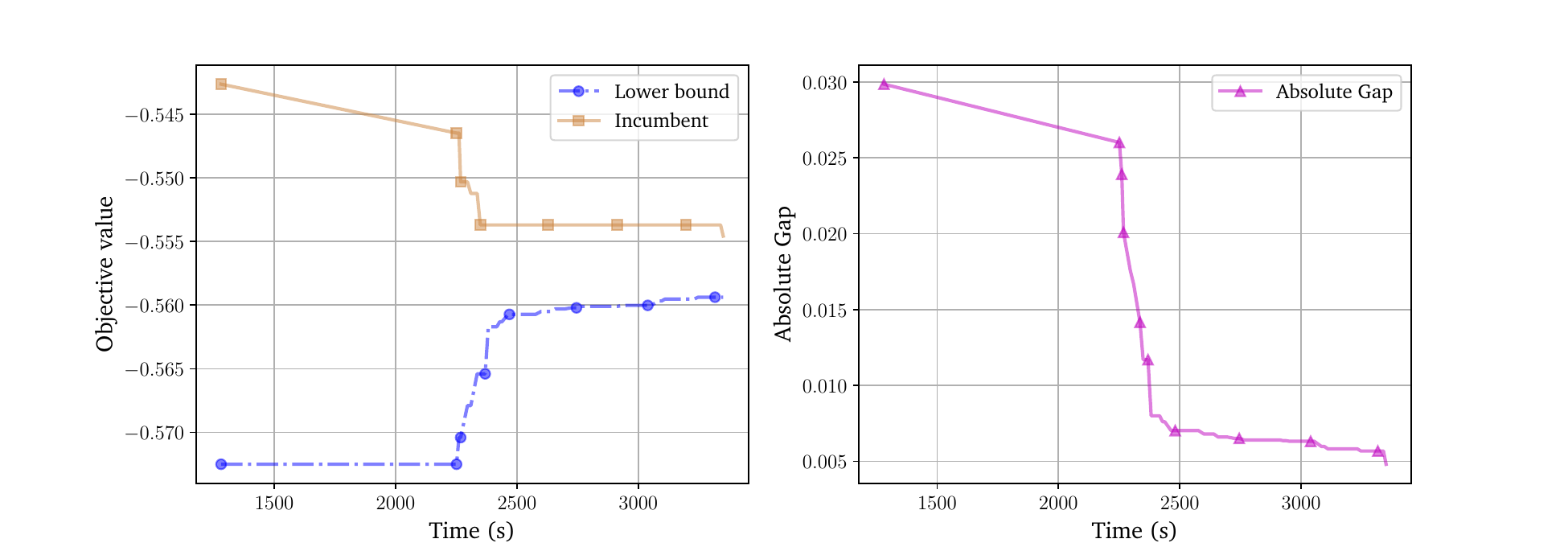}
       \caption{D-Optimal Problem with correlated data and $n=10$}
       \label{fig:ProgressPlotsDOCORR}
    \end{subfigure}
    \centering
    \begin{subfigure}[b]{0.9\textwidth}
       \includegraphics[width=1\linewidth]{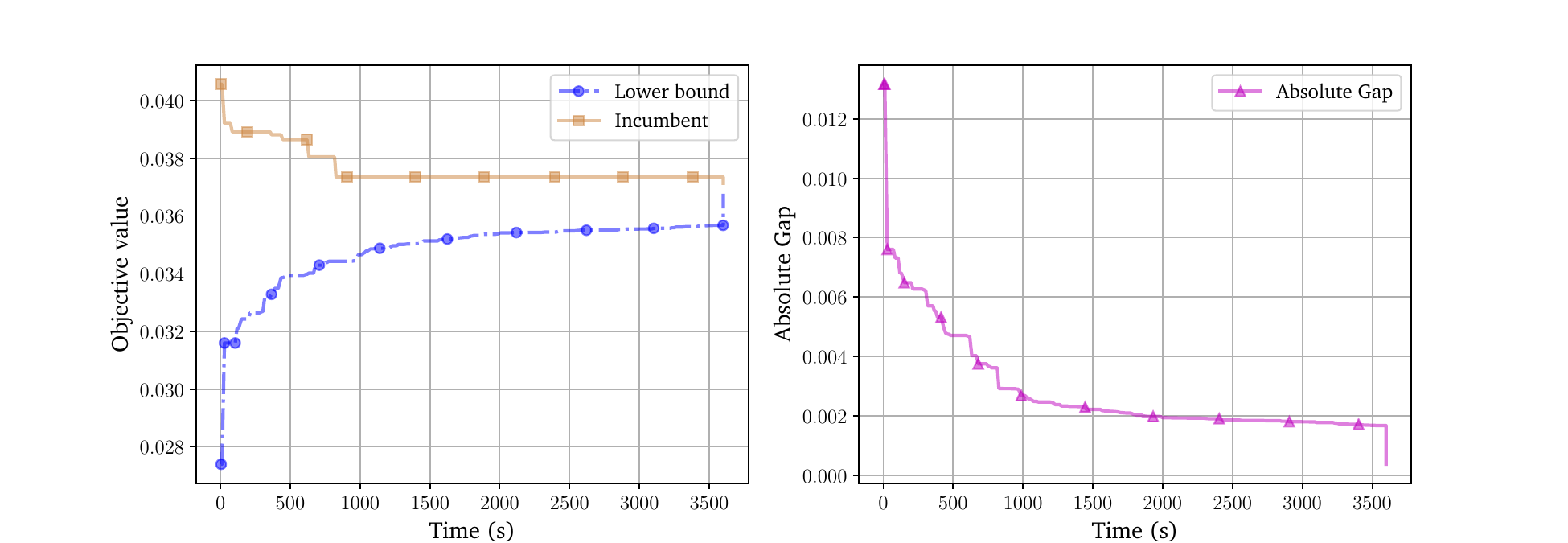}
       \caption{A-Fusion Problem with independent data and $n=10$}
       \label{fig:ProgressPlotsAFIND} 
    \end{subfigure}
    
    \begin{subfigure}[b]{0.9\textwidth}
       \includegraphics[width=1\linewidth]{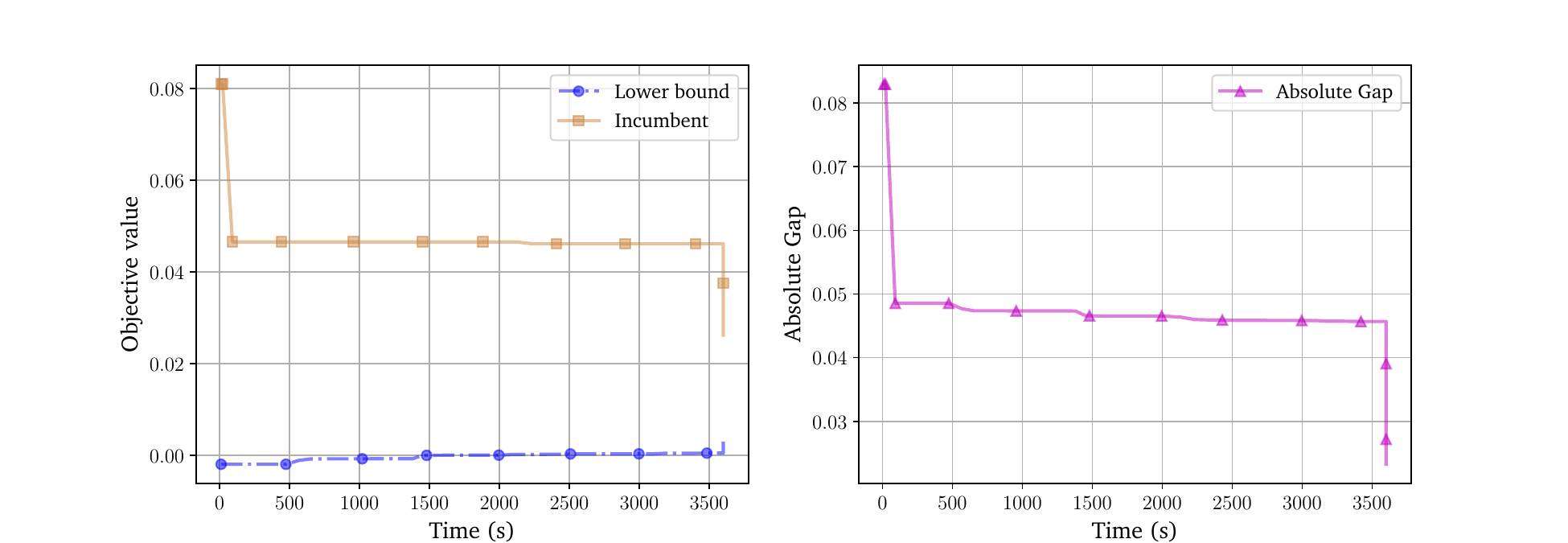}
       \caption{A-Fusion Problem with correlated data and $n=25$}
       \label{fig:ProgressPlotsAFCORR}
    \end{subfigure}
    \label{fig:ProgressPlots}
    \caption{Progress of the incumbent and the lower bound on the left and progress of the absolute gap on the right. All instances have 100 variables.}
\end{figure}

\begin{table}[hbt!] \scriptsize
    \centering
    \caption{\footnotesize Comparing the performance of \package{}, \pajarito, \scip, \texttt{Co-BnB} and \texttt{SOCP} on the different problems and the different data sets, i.e. A-Fusion (AF), D-Fusion (DF), A-Optimal (AO) and D-Optimal (DO). 
    One data set contains independent data, the other has correlated data. 
    Shown are the percentage of solved instances and the geometric mean of the solve time (shifted by 1 second) over all instances. Note that there are 50 instances for each problem and data set. \\
    For more extensive results like number of cuts/nodes, see \cref{tab:SummaryByDifficultyExt} in \cref{sec:AppendixComEx}.} 
    \label{tab:SummaryByDifficulty}
    \vspace{0.1cm}
    \begin{tabular}{llHH HrrHH HrrHH HrrHH HrrHH HrrHH} 
        \toprule
        \multicolumn{4}{l}{} & \multicolumn{5}{c}{Boscia} & \multicolumn{5}{c}{Co-BnB} & \multicolumn{5}{c}{\thead{Direct Conic}} & \multicolumn{5}{c}{SOCP} & \multicolumn{5}{c}{SCIP OA}\tabularnewline 
        
        \cmidrule(lr){5-9}
        \cmidrule(lr){10-14}
        \cmidrule(lr){15-19}
        \cmidrule(lr){20-24}
        \cmidrule(lr){25-29}

        \thead{Type} & \thead{Corr.} & \thead{Solved \\ after \\ (s)} & \thead{\#\\ inst.} & \thead{\# \\solved} & \thead{\% \\solved} & \thead{Time (s)} & \thead{relative \\ gap} & \thead{\# nodes} & \thead{\# \\solved} & \thead{\% \\solved} & \thead{Time (s)} & \thead{relative \\ gap} & \thead{\# nodes} & \thead{\# \\solved} & \thead{\% \\solved} & \thead{Time (s)} & \thead{relative \\ gap} & \thead{\# cuts} & \thead{\# \\solved} & \thead{\% \\solved} & \thead{Time (s)} & \thead{relative \\ gap} & \thead{\# cuts} & \thead{\# \\solved} & \thead{\% \\solved} & \thead{Time (s)} & \thead{relative \\ gap} & \thead{\# cuts} \\
        \midrule
        AO & no & 0 & 50 & 26 & \textbf{52 \%} & \textbf{374.74} & \textbf{0.0851} & \textbf{7340} & 24 & 48 \% & 401.67 & 51.4694 & 22211 & 4 & 8 \% & 2382.52 & 2.4436 & 112 & 10 & 20 \% & 1578.03 & Inf & 2842 & - & - & - & - & - \\
       AO & yes & 0 & 50 & 34 & \textbf{68 \%} & \textbf{227.69} & \textbf{0.0336} & \textbf{7814} & 30 & 60 \% & 301.8 & 2058.4159 & 22379 & 10 & 20 \% & 1515.23 & 30.3669 & 188 & 10 & 20 \% & 1861.35 & Inf & 3623 & - & - & - & - & - \\
       \midrule
       AF & no & 0 & 50 & 40 & \textbf{80 \%} & 62.83 & \textbf{0.0168} & \textbf{11634} & 40 & \textbf{80 \%} & \textbf{50.92} & 30.35 & 10648 & 2 & 4 \% & 2846.5 & 0.0371 & 58 & 13 & 26 \% & 1686.2 & Inf & 5634 & 19 & 38 \% & 463.65 & 0.0302 & 801 \\  
       AF & yes & 0 & 50 & 13 & 26 \% & 1368.83 & \textbf{2.1573} & 28578 & 26 & \textbf{52 \%} & \textbf{311.36} & 1407.4062 & 30449 & 12 & 24 \% & 1148.87 & 9.0208 & 318 & 10 & 20 \% & 2131.81 & Inf & 4539 & 7 & 14 \% & 1480.73 & 3.4254 & 552 \\
       \midrule
       DO & no & 0 & 50 & 34 & \textbf{68 \%} & \textbf{136.92} & \textbf{0.0208} & \textbf{3578} & 29 & 58 \% & 216.31 & 2.7678 & 9559 & 10 & 20 \% & 1241.37 & 0.676 & 1077 & 10 & 20 \% & 2169.4 & Inf & 4514 & - & - & - & - & - \\
       DO & yes & 0 & 50 & 50 & \textbf{100 \%} & \textbf{2.56} & \textbf{0.0094} & \textbf{108} & 35 & 70 \% & 101.89 & 0.0809 & 7973 & 5 & 10 \% & 1398.01 & 0.0206 & 309 & 5 & 10 \% & 2526.74 & Inf & 2871 & - & - & - & - & - \\
       \midrule 
       DF & no & 0 & 50 & 47 & \textbf{94 \%} & \textbf{3.49} & \textbf{0.0102} & \textbf{301} & 43 & 86 \% & 22.79 & 0.0382 & 6782 & 6 & 12 \% & 1906.61 & 0.0126 & 433 & 9 & 18 \% & 2213.03 & Inf & 4787 & 25 & 50 \% & 324.11 & 0.0133 & 1155 \\
       DF & yes & 0 & 50 & 30 & \textbf{60 \%} & \textbf{54.01} & \textbf{0.0842} & \textbf{7635} & 28 & 56 \% & 153.26 & 0.4948 & 22485 & 3 & 6 \% & 2284.06 & 0.1556 & 319 & 3 & 6 \% & 2934.61 & Inf & 3117 & 14 & 28 \% & 752.15 & 0.092 & 786 \\
        \bottomrule
    \end{tabular}
\end{table}

\begin{table}[hbt!] \scriptsize
    \centering
    \caption{\footnotesize Comparing the performance of \package{},\scip and \texttt{Co-BnB} on the General-Trace-Inverse (GTI) Optimal Problem and GTI-Fusion (GTIF) Problem for different values of $p$ and under independent and correlated data. 
    Note the the objective function of Boscia has the form $\phi_p(X(\vx)) = \Tr\left(X(\vx)^{-p}\right)$. The column titled "Boscia Log" shows the runs using $\log\phi_p$ as the objective. 
    Shown are the percentage of solved instances and the geometric mean of the solve time (shifted by 1 second) over all instances. Note that there are 30 instances for each problem and data set. \\
    For extensive results, see \cref{tab:SummaryByDifficultyGTI,tab:SummaryByDifficultyGTIF} in \cref{sec:AppendixComEx}.} 
    \label{tab:SummaryGTI}
    \vspace{0.1cm}
    \begin{tabular}{llHH HrrHH HrrHH HrrHH HrrHH}
        \toprule
        \multicolumn{4}{l}{} & \multicolumn{5}{c}{Boscia}  & \multicolumn{5}{c}{Boscia Log} & \multicolumn{5}{c}{Co-BnB} & \multicolumn{5}{c}{SCIP+OA} \tabularnewline
        
        \cmidrule(lr){5-9}
        \cmidrule(lr){10-14}
        \cmidrule(lr){15-19}
        \cmidrule(lr){20-24}

        \thead{p} & \thead{Corr.} & \thead{Solved \\ after \\ (s)} & \thead{\#\\ inst.} & \thead{\# \\solved} & \thead{\% \\solved} & \thead{Time (s)} & \thead{relative \\ gap} & \thead{\# nodes} & \thead{\# \\solved} & \thead{\% \\solved} & \thead{Time (s)} & \thead{relative \\ gap} & \thead{\# nodes} & \thead{\# \\solved} & \thead{\% \\solved} & \thead{Time (s)} & \thead{relative \\ gap} & \thead{\# nodes} & \thead{\# \\solved} & \thead{\% \\solved} & \thead{Time (s)} & \thead{relative \\ gap} & \thead{\# nodes} \\ 
        \midrule
        Optimal Problem & & & & &  & & & & &  & & & & &  & & & & &  & & & \\
        \midrule
        0.25 & no & 0 & 30 & 20 & 67 \% & 162.41 & 0.0229 & 340 & 26 & \textbf{87 \%} & \textbf{27.13} & \textbf{0.0106} & 39 & 15 & 50 \% & 359.92 & 0.6731 & 8495 & - & - & - & - & - \\
        0.25 & yes & 0 & 30 & 20 & 67 \% & 110.25 & 0.0205 & 130 & 23 & \textbf{77 \%} & \textbf{85.83} & \textbf{0.0126} & 164 & 18 & 60 \% & 264.22 & 0.987 & 5463 & - & - & - & - & - \\
       \midrule
        0.5 & no & 0 & 30 & 16 & 53 \% & 246.61 & 0.0625 & 1786 & 21 & \textbf{70 \%} & \textbf{85.87} & \textbf{0.0225} & 1225 & 15 & 50 \% & 430.08 & 4.7973 & 12614 & - & - & - & - & - \\
        0.5 & yes & 0 & 30 & 18 & \textbf{60 \%} & \textbf{129.86} & \textbf{0.0393} & 1246 & 17 & 57 \% & 245.36 & 0.0836 & 2971 & 17 & 57 \% & 313.7 & 32.2241 & 6271 & - & - & - & - & - \\
        \midrule
        0.75 & no & 0 & 30 & 11 & 37 \% & 799.73 & 0.1606 & 349 & 13 & 43 \% & 505.86 & \textbf{0.0541} & 218 & 15 & \textbf{50 \%} & \textbf{471.95} & 35.492 & 14963 & - & - & - & - & - \\
        0.75 & yes & 0 & 30 & 14 & 47 \% & 541.67 & \textbf{0.081} & 676 & 13 & 43 \% & 557.15 & 0.1052 & 506 & 15 & \textbf{50 \%} & \textbf{357.77} & 1690.0144 & 5438 & - & - & - & - & - \\
        \midrule
        1.5 & no & 0 & 30 & 13 & 43 \% & 659.25 & 1.2623 & 2344 & 12 & 40 \% & \textbf{576.51} & \textbf{0.2686} & 1308 & 15 & \textbf{50 \%} & 619.95 & 96345.6059 & 14548 & - & - & - & - & - \\
        1.5 & yes & 0 & 30 & 14 & 47 \% & 418.43 & 0.2003 & 4555 & 16 & \textbf{53 \%} & \textbf{306.7} & \textbf{0.0489} & 1474 & 14 & 47 \% & 506.3 & 7.351349315159e8 & 3818 & - & - & - & - & - \\
        \midrule
        2.0 & no & 0 & 30 & 13 & 43 \% & 661.29 & 3.8118 & 3386 & 15 & \textbf{50 \%} & \textbf{562.4} & \textbf{0.5984} & 7311 & 15 & \textbf{50 \%} & 630.18 & 2.37363945824e7 & 15719 & - & - & - & - & - \\
        2.0 & yes & 0 & 30 & 13 & 43 \% & 644.53 & 1.1 & 25219 & 16 & \textbf{53 \%} & \textbf{297.98} & \textbf{0.0373} & 4067 & 14 & 47 \% & 611.8 & 3.672262322023188e12 & 6479 & - & - & - & - & - \\
        \midrule
        Fusion Problem & & & & &  & & & & &  & & & & &  & & & & &  & & & \\
        \midrule
        0.25 & no & 0 & 30 & 27 & 90 \% & 6.57 & 0.0087 & 13 & 30 & \textbf{100 \%} & \textbf{1.71} & \textbf{0.0066} & 6 & 15 & 50 \% & 725.52 & 0.1107 & 309 & 14 & 47 \% & 376.08 & 0.0085 & 671 \\
        0.25 & yes & 0 & 30 & 16 & 53 \% & \textbf{138.78} & 0.2337 & 231 & 17 & \textbf{57 \%} & 184.83 & \textbf{0.1003} & 455 & 8 & 27 \% & 1449.73 & 1.1622 & 408 & 8 & 27 \% & 927.92 & 0.2099 & 1090 \\
        \midrule
        0.5 & no & 0 & 30 & 26 & 87 \% & 9.1 & 0.014 & 140 & 27 & \textbf{90 \%} & \textbf{3.63} & \textbf{0.0096} & 27 & 14 & 47 \% & 787.45 & 0.2604 & 450 & 13 & 43 \% & 349.14 & 0.0164 & 807 \\
        0.5 & yes & 0 & 30 & 17 & \textbf{57 \%} & \textbf{175.85} & 0.586 & 2649 & 16 & 53 \% & 271.12 & \textbf{0.5412} & 3554 & 9 & 30 \% & 1547.49 & 4.087 & 559 & 6 & 20 \% & 850.41 & 0.583 & 298 \\
        \midrule
        0.75 & no & 0 & 30 & 21 & 70 \% & 85.39 & 0.0234 & 123 & 26 & \textbf{87 \%} & \textbf{42.47} & \textbf{0.0141} & 120 & 17 & 57 \% & 353.3 & 0.5793 & 191 & 13 & 43 \% & 415.34 & 0.0285 & 725 \\
        0.75 & yes & 0 & 30 & 11 & 37 \% & 741.57 & 1.8886 & 867 & 13 & \textbf{43 \%} & \textbf{434.78} & 1.7079 & 353 & 8 & 27 \% & 1639.63 & 14.5788 & 516 & 7 & 23 \% & 1053.8 & 1.6689 & 877 \\
        \midrule
        1.5 & no & 0 & 30 & 23 & \textbf{77 \%} & \textbf{91.83} & 0.0579 & 1772 & 20 & 67 \% & 129.57 & \textbf{0.056} & 1555 & 14 & 47 \% & 780.56 & 2.0633 & 336 & 12 & 40 \% & 406.94 & 0.0836 & 718 \\
        1.5 & yes & 0 & 30 & 11 & 37 \% & 715.74 & 13.4006 & 8008 & 15 & \textbf{50 \%} & \textbf{214.97} & 7.6622 & 2014 & 9 & 30 \% & 1415.97 & 790.4184 & 444 & 7 & 23 \% & 885.83 & 12.5165 & 764 \\
       \midrule
        2.0 & no & 0 & 30 & 22 & \textbf{73 \%} & \textbf{96.28} & \textbf{0.0877} & 2668 & 21 & 70 \% & 117.1 & 0.323 & 2985 & 17 & 57 \% & 797.56 & 4.3773 & 741 & 12 & 40 \% & 402.28 & 0.1261 & 641 \\
        2.0 & yes & 0 & 30 & 10 & 33 \% & 905.59 & 51.5626 & 19493 & 16 & \textbf{53 \%} & \textbf{202.0} & \textbf{1.4881} & 5179 & 10 & 33 \% & 1305.28 & 11737.6455 & 644 & 6 & 20 \% & 923.45 & 110.7213 & 377 \\
        \bottomrule
    \end{tabular}
\end{table}

\begin{table}[hbt!] \scriptsize
    \centering
    \caption{\footnotesize Comparing Boscia and Co-BnB on large instances.
    The dimensions are $m=300,400,500$. We showcase only the absolute and relative gap. Note that the majority of the instances were solved to optimality.
    Boscia solved 3 instances of the D-Fusion Problem with independent data. 
    The relative gap and the absolute gap are computed for the instances on which at least one method did not terminate within the time limit. That means it excludes the instances on which all methods terminated.
    The "log" prefix indicates that Boscia was used with the $\log\Tr$ objective. } 
    \label{tab:SummaryLongRuns}
    \vspace{0.1cm}
    \begin{tabular}{llr HHHrHr HHHrHr} 
        \toprule
        \multicolumn{3}{l}{} & \multicolumn{6}{c}{Boscia} & \multicolumn{6}{c}{Co-BnB} \tabularnewline 
        
        \cmidrule(lr){4-9}
        \cmidrule(lr){10-15}

        \thead{Problem} & \thead{Corr.} & \thead{\#\\ inst.} & \thead{\# \\solved} & \thead{\% \\solved} & \thead{Time (s)} & \thead{relative \\ gap} & \thead{\# nodes}  & \thead{absolute \\ gap} & \thead{\# \\solved} & \thead{\% \\solved} & \thead{Time (s)} & \thead{relative \\ gap} & \thead{\# nodes} & \thead{absolute \\ gap} \\
        \midrule
        A & no & 12 & 0 & 0 \% & 14416.84 & 188.2483 & NaN & \textbf{3.8117} & 0 & 0 \% & 15705.1 & \textbf{2.4036} & NaN & 5.2537 \\
        A & yes & 12 & 0 & 0 \% & 14606.85 & Inf & NaN & \textbf{0.0024} & 0 & 0 \% & 15756.78 & Inf & NaN & 5.8384 \\
        AF & no & 12 & 0 & 0 \% & 14415.49 & \textbf{0.2243} & NaN & \textbf{0.0019} & 0 & 0 \% & 15619.76 & 0.9264 & NaN & 1.0258 \\
        AF & yes & 12 & 0 & 0 \% & 14414.79 & Inf & NaN & \textbf{0.0049} & 0 & 0 \% & 16174.42 & Inf & NaN & 5.6742 \\
        \midrule
        log A & no & 12 & 0 & 0 \% & 14681.67 & 2.4047 & NaN & 5.2608 & 0 & 0 \% & 15705.1 & 2.4036 & NaN & \textbf{5.2537} \\
        log A & yes & 12 & 0 & 0 \% & 14680.83 & Inf & NaN & \textbf{5.6258} & 0 & 0 \% & 15756.78 & Inf & NaN & 5.8384 \\
        log AF & no & 12 & 0 & 0 \% & 14644.6 & 1.0491 & NaN & 2.6105 & 0 & 0 \% & 15619.76 & \textbf{0.9264} & NaN & \textbf{1.0258} \\
        log AF & yes & 12 & 0 & 0 \% & 14647.7 & \textbf{0.0955} & NaN & 6.9323 & 0 & 0 \% & 16174.42 & Inf & NaN & \textbf{5.6742} \\
        \midrule
        DF & no & 12 & 3 & 25 \% & 8666.65 & \textbf{0.0169} & 124 & \textbf{0.01} & 0 & 0 \% & 15623.99 & Inf & NaN & 38.0805 \\
        DF & yes & 12 & 0 & 0 \% & 14516.23 & \textbf{0.9819} & NaN & \textbf{0.4054} & 0 & 0 \% & 16301.39 & Inf & NaN & 343.1067 \\
        D & no & 12 & 0 & 0 \% & 14512.85 & \textbf{0.2295} & NaN & \textbf{0.1577} & 0 & 0 \% & 16321.7 & Inf & NaN & 69.0189 \\
        D & yes & 12 & 0 & 0 \% & 14499.83 & \textbf{0.4642} & NaN & \textbf{0.091} & 0 & 0 \% & 16084.57 & 673.4528 & NaN & 73.5618 \\
        \bottomrule
    \end{tabular}
\end{table}

\section{Conclusion}

We proposed a new approach for the Optimal Experiment Design Problem based on the \package{} framework and proved convergence of the method on the problems.
Our approach exhibits a superior performance compared to other MINLP approaches.
In addition, it also outperforms the approach specifically developed for the OEDP, in particular for large-scale instances and a larger number of parameters.
This superiority can be explained by the fact that \package{} keeps the structure of the problem intact and that it utilizes a combinatorial solver to find integer feasible points at each node.

\section*{Acknowledgments}

Research reported in this paper was partially supported through the Research Campus Modal funded by the German Federal Ministry of Education and Research (fund numbers 05M14ZAM,05M20ZBM) and the Deutsche Forschungsgemeinschaft (DFG) through the DFG Cluster of Excellence MATH+.

\clearpage
\bibliographystyle{icml2021}
\bibliography{paper}

\ifarxiv
\newpage
\appendix
\section{Additional proofs and calculations}\label{sec:AppendixProofs}

Note that the following lemmas and the proofs are not new but they are stated here for completeness.

\subsection{General power of a matrix}\label{subsec:MatrixExponent}
To start, we present some smaller statements about the power of a matrix.
These are used in the consequent proofs.
\begin{lemma}\label{lm:SymOfLogX}
    Let $X\in\R^{n\times n}$ be an invertible, symmetric matrix. Then, the logarithm $\log(X)$ is symmetric.
\end{lemma}
\begin{proof}
    Let $V$ denote the matrix consisting of eigenvectors of $X$. Observe that $V$ is an orthogonal matrix and w.l.o.g.~we assume $V$ is scaled such that it is orthonormal. Next, let $D$ denote the diagonal matrix whose entries are the eigenvalues of $X$. The logarithm of $X$ is then defined as 
    \begin{align*}
        \log(X) &= V \log(D) V^{-1} \\
        \intertext{where $\log(D)$ is the diagonal matrix with logarithms of the eigenvalues. In our case, $V$ is orthonomal, so $V^{-1}=V^\intercal$. Hence,}
        \log(X)^\intercal &= \left(V\log(D)V^\intercal\right)^\intercal =V \log(D) V^\intercal = \log(X)
    \end{align*}
\end{proof}

\begin{lemma}\label{lm:EigValPowerMatrix}
    Let $X\in\R^{n\times n}$ by an invertible matrix and $r\in\R$. Then, the eigenvalues of the power matrix are
    \[
        \lambda_i(X^r) = \lambda_i(X)^r \quad \forall i=[n]
    \]
    and the eigenvectors of the power matrix are the same as those of $X$.
\end{lemma}
\begin{proof}
    Let $\vvv$ be an eigenvector of $X$ and let $\lambda_{\vvv}$ denote the corresponding eigenvalue.
    \begin{align*}
        X^r \vvv &= \left(\sum\limits_{k=0}^\infty \frac{r^k}{k!} \log(X)^k\right) \vvv = \sum\limits_{k=0}^\infty \frac{r^k}{k!} \log(X)^k\vvv \\
        \intertext{For any square matrix $A$ and a natural number $m\in\N$, we have that $A^m \vvv = \lambda_{\vvv}(A)^m \vvv$. Also, the logarithm of a matrix $A$ has the same eigenvectors as the original and its eigenvalues are $\log(\lambda_i(A))$.}
        &= \left(\sum\limits_{k=0}^\infty \frac{r^k\log(\lambda_{\vvv}(X))^k}{k!}\right)\vvv = \exp(r\log(\lambda_{\vvv}(X)))\vvv \\ 
        &= \lambda_{\vvv}(X)^r \vvv
    \end{align*}
    This concludes the proof.
\end{proof}

Combining \cref{lm:SymOfLogX} and \cref{lm:EigValPowerMatrix}, we get the following corollary.
\begin{corollary}\label{cor:PowerPosDef}
    For $r\in\R$, if $X\in\pdCone^n$, then $X^r\in\pdCone^n$.
\end{corollary}

\begin{lemma}\label{lm:DerivativeMatrixPower}
    Let $r\in\R$ and let the matrix $X$ be diagonalizable. Let $f(X) = \Tr(X^r)$. The gradient of $f$ is then
    \[\nabla f(x) = rX^{r-1}\]
\end{lemma}
\begin{proof}
    We use the definition of $X^r$ to prove the result.
    \begin{align*}
        \Tr(X^r) &= \Tr\left(\sum\limits_{k=0}^\infty \frac{r^k}{k!} \log(X)^k\right) = \sum\limits_{k=0}^\infty \frac{r^k}{k!} \Tr\left(\log(X)^k\right) \\
        \intertext{Thus, using the sum rule, $(\Tr(X^n))' =nX^{n-1}$ for any positive integer $n$ and $\nabla\Tr(\log(X)) = X^{-1}$, we have}
        \nabla \Tr(X^r) &= \sum\limits_{k=1}^\infty \frac{r^k}{k!} k {\log(X)}^{k-1} X^{-1} \\ 
        &= r X^{-1} \sum\limits_{k=0}^\infty \frac{r^k}{k!} \log(X)^k\\
        &= r X^{-1} X^{r} = rX^{r-1}
    \end{align*}
    This concludes the proof.
\end{proof}

\begin{lemma}
    For $X\in\pdCone^n$, we have $X^{-p} = \left(X^{-1}\right)^p$.
\end{lemma}
\begin{proof}
    \begin{align*}
        \left(X^{-1}\right)^p &= \sum\limits_{k=0}^\infty \frac{p^k}{k!} \log\left(X^{-1}\right)^k \\
        \log\left(X^{-1}\right) &= V \log(D^{-1}) V^\intercal = - V \log(D) V^\intercal = -\log(X) \\
        \left(X^{-1}\right)^p &= \sum\limits_{k=0}^\infty \frac{p^k}{k!} \left(-\log(X)\right)^k  \\
        &= \sum\limits_{k=0}^\infty \frac{(-p)^k}{k!} \log(X)^k \\
        &= X^{-p}
    \end{align*}
\end{proof}

\begin{remark}
On the contrary, $\left(X^p\right)^{-1} \neq X^{-p}$ in general.
\end{remark}

\subsection{Convexity of the objectives}\label{sec:Convexity}
In this subsection, we present the convexity proofs for all the considered objectives. 
Note that concavity of the $\log\det$ is well known and is stated here for completeness.
\begin{lemma}\label{lm:ConvexityLogDet}
    The function $f(X)=-\log\det(X)$ is convex for $X\in\pdCone^n$.
\end{lemma}
\begin{proof}
Let $A\in\pdCone^n$, $B\in\mathbb{S}$ and $t\in\R$ such that $A+tB\in\pdCone^n$. We define $h(t)=-\log\det\left(A+tB\right)$.
\begin{align*}
        -\log\det(A+tB) &= -\log\det\left(\mRoot{A}\mRoot{A} + t \mRoot{A}\invMRoot{A} B \invMRoot{A}\mRoot{A}\right) \\
        &= -\log\det\left(\mRoot{A} \left(I + t \invMRoot{A} B \invMRoot{A}\right) \mRoot{A}\right) \\
        &= -\log\det(A) - \log\det\left(I + t \invMRoot{A} B \invMRoot{A}\right) \\
        \intertext{Observe that $I + t \invMRoot{A} B \invMRoot{A}$ is a positive definite matrix and its eigenvalues are $1+t\lambda_i$ with $\lambda_i$ being the $i$-th eigenvalue of $\invMRoot{A} B \invMRoot{A}$. Note that this also means $1+t\lambda_i > 0$.}
        &= -\log\det(A) - \log\left(\prod\limits_{i=1}^n(1+t\lambda_i)\right) \\
        &= -\log\det(A) - \sum\limits_{i=1}^n \log\left(1+t\lambda_i\right) \\
        \intertext{The first and second derivatives of $g$ are}
        h'(t) &= -\sum\limits_{i=1}^n \frac{1}{1+t\lambda_i} \lambda_i \\
        h''(t) &= \sum\limits_{i=1}^n\frac{\lambda_i^2}{\left(1+t\lambda_i\right)}\\
        &> 0 \quad \forall t
    \end{align*}
    Thus, $f(X)$ is convex.
\end{proof}

\begin{lemma}\label{lm:ConvexityTraceInverse}
    The function $g(X) = \Tr(X^{-p})$ is convex for  $X \in\pdCone^n$ and $p>0$. 
\end{lemma}
\begin{proof}
    Let $A\in\pdCone^n$, $B\in\mathbb{S}$ and $t\in\R$ such that $A+tB\in\pdCone^n$. We define $h(t)=\Tr\left(\left(A+tB\right)^{-p}\right)$.
    \begin{align*}
        h(t) &= \Tr((A+tB)^{-p}) \\
        &= \Tr\left(\left((A+tB)^{-1}\right)^p\right) \\
        \intertext{Using \cref{lm:DerivativeMatrixPower} for the derivative yields}
        h'(t) &= \Tr\left(p \left((A+tB)^{-1}\right)^{p-1} (-1) (A+tB)^{-1}B(A+tB)^{-1}\right) \\
        &= -p\Tr\left(\left((A+tB)^{-1}\right)^{p}B(A+tB)^{-1}\right) \\
        h''(t) &= -p\Tr\left(p \left((A+tB)^{-1}\right)^{p-1} (-1) (A+tB)^{-1}B(A+tB)^{-1}B(A+tB)^{-1}\right. \\
        & \left. \quad + \left((A+tB)^{-1}\right)^p B (-1)(A+tB)^{-1}B(A+tB)^{-1} \vphantom{\left((A+tB)^{-1}\right)^{p-1}} \right) \\
        &= p \Tr\left((p+1)\left((A+tB)^{-1}\right)^p B(A+tB)^{-1}B(A+tB)^{-1} \right) \\
        \intertext{Restricting $h''$ to $t=0$ yields}
        h''(t)_{\restriction t=0} &= p(p+1) \Tr\left(A^{-p} BA^{-1}BA^{-1} \right). \\
        \intertext{For the third derivative, we have}
        h'''(t) &= p(p+1) \Tr\left(p \left((A+tB)^{-1}\right)^{p-1} (-1)(A+tB)^{-1} B(A+tB)^{-1}B(A+tB)^{-1}B(A+tB)^{-1} \right. \\
        &\left. \quad +2\left((A+tB)^{-1}\right)^{p} (-1)B(A+tB)^{-1}B(A+tB)^{-1}B(A+tB)^{-1} \vphantom{\left((A+tB)^{-1}\right)^{p-1}}\right) \\
        h'''(t) &= -p(p+1)(p+2) \Tr\left((A+tB)^{-p}B(A+tB)^{-1}B(A+tB)^{-1}B(A+tB)^{-1}\right). \\
        h'''(t)_{\restriction t=0} &= -p(p+1)(p+2)\Tr\left(A^{-p}BA^{-1}BA^{-1}BA^{-1}\right) \\
        \intertext{By \cref{cor:PowerPosDef}, $A^{-p}$ is positive definite and thus, has a unique root $A^{-p/2}$. Using the traces cyclic property, we find}
        h''(t)_{\restriction t=0} &= p(p+1) \Tr\left(A^{-\frac{p+1}{2}}BA^{-1}BA^{-\frac{p+1}{2}}\right) \\
        &= p(p+1) \sum\limits_{i=1}^n \lambda_i\left(A^{-\frac{p+1}{2}}BA^{-1}BA^{-\frac{p+1}{2}}\right) \\
        &\geq 0 
    \end{align*}
    The last inequality holds because the matrix $A^{-\frac{p+1}{2}}BA^{-1}BA^{-\frac{p+1}{2}}$ is positive definite by positive definiteness of $A$, $B$ and $A^{-p}$. Thus, $g(X)$ is convex.
\end{proof}

\begin{lemma}\label{lm:ConvexityLogofTrace}
    The function $k(X) = \log\left(\Tr(X^{-p})\right)$ is convex on $\pdCone^n$ for $p > 0$. 
\end{lemma}
\begin{proof}
    Let $A \in \pdCone^n$, $B \in \mathbb{S}^n$ and $t\in\R$ s.t. $A+tB\in\pdCone^n$. We define $h(t) = k(A+tB) = \log\left(\Tr((A+tB)^{-p})\right)$.
    \begin{align*}
       h(t) &= \log\left(\Tr\left((A+tB)^{-p}\right)\right) \\
       h'(t) &= -\frac{p}{\Tr((A+tB)^{-p})} \Tr((A+tB)^{-p-1}B) \\
       h''(t) &= p \frac{(p+1) \Tr\left((A+tB)^{-p-2}BB\right)\Tr\left((A+tB)^{-p}\right) - p \Tr\left((A+tB)^{-p-1}B\right)^2}{\Tr\left((A+tB)^{-p}\right)^2} \\
       \intertext{Restricting us to zero, we get}
       h''(t)_{\restriction t=0} &= p \frac{(p+1) \Tr\left(A^{-p-2}BB\right)\Tr\left(A^{-p}\right) - p \Tr\left(A^{-p-1}B\right)^2}{\Tr\left(A^{-p}\right)^2} .\\
       \intertext{By assumption, $p$ is positive as is the denominator. For the nominator, we use the Cauchy Schwartz inequality $\Tr\left(XY^\intercal\right)^2 \leq \Tr(XX^\intercal)\Tr(YY^\intercal)$ and set $X = A^{-p/2}$ and $Y=A^{-\frac{p+2}{2}}B$}
       h''(t)_{\restriction t=0} &\geq 0
    \end{align*}
    Thus, $k(X)$ is convex.
\end{proof}

\subsection{Additional smoothness proofs}\label{sec:AdditionalSmoothness}
In following, we show that all the considered functions are $L$-smooth if we can bound the minimum
eigenvalue of the input matrix from below.
\begin{lemma}\label{th:LSmoothLogDet}
The function $f(X) := -\log\det(X)$ is $L$-smooth on $D := \{X\in\pdCone^n \, \mid\, \delta \leq \lambda_{\min}(X)\}$ with $L \geq \frac{1}{\delta^2}$.
\end{lemma}
\begin{proof}
    Let $A\pdCone^n$, $B\in\mathbb{S}^n$ and $t\in[0,1]$. 
    Then we can define $h(t)=-\log\det\left(A+tB\right)$. By \cref{lm:ConvexityLogDet}, $f(X)$ is convex and the second derivative of $h$ is given by
    \begin{align*}
        h''(t) &= \sum\limits_{i=1}^n \frac{\lambda_i\left(A^{-1/2}BA^{-1/2}\right)^2}{1+t\lambda_i\left(A^{-1/2}BA^{-1/2}\right)} \\
        \intertext{It is sufficient to show that $h''(t)\restriction_{t=0} \leq L$ to argue smoothness for $f$.}
        h''(t)\restriction_{t=0} &= \sum\limits_{i=1}^n \lambda_i\left(A^{-1/2}BA^{-1/2}\right)^2 \\
        \intertext{By using a similar idea to \cref{lm:EigValProdPosDefMatrix}, we find}
        &\leq \sum\limits_{i=1}^n \lambda_{\max}\left(A^{-1}\right)^2 \lambda_i(B)^2 \\
        &= \lambda_{\max}\left(A^{-1}\right)^2 \Tr\left(B^2\right)\\
        &= \lambda_{\max}\left(A^{-1}\right)^2 \norm{B}_F^2 \\
        \intertext{W.l.o.g. let $\norm{B}_F=1$.}
        &= \frac{1}{\lambda_{\min}(A)^2} \\
        &\geq \frac{1}{\delta^2} \\
        &= L
    \end{align*}
    Thus, the function $f$ is smooth at any point $A\in D$.
\end{proof}
    
\begin{lemma}\label{th:LSmoothTrPower}
    For $p \in \R_{>0}$, the function $g(X) := \Tr\left(X^{-p}\right)$ is $L$-smooth on the domain $D := \{X\in\pdCone^n \, \mid\, \delta \leq \lambda_{\min}(X) \}$ with 
    \[L \geq \frac{p(p+1)}{\delta^{p+2}}.\]
\end{lemma}
\begin{proof}
    Let $A \in D$ and $B\in\mathbb{S}^n$ and $t\in[0,1]$. Define $h(t) = \Tr((A+tB)^{-p})$. In \cref{lm:ConvexityTraceInverse} the proof for convexity of $g(X)$ is presented. There one can also find the calculation for $h''(t)_{\restriction t=0}$.
    \begin{align*}
        h''(t)_{\restriction t=0} &= p(p+1) \Tr\left(A^{-p} BA^{-1}BA^{-1} \right). \\
        \intertext{Since $A^{-p}$ is positive definite by \cref{cor:PowerPosDef} and the trace is cyclic, we have}
        &= p(p+1) \Tr\left(A^{-\frac{p+2}{2}}BBA^{-\frac{p+2}{2}}\right) \\
        \intertext{Note that $BB$ is postive semi-definite. Thus, we can use the bound in \cite{fang1994inequalities}.}
        &\leq p(p+1) \lambda_{\max}\left(A^{-(p+2)}\right) \Tr(BB) \\
        &= p(p+1) \lambda_{\max}\left(A^{-(p+2)}\right) \norm{B}_F^2 \\
        \intertext{W.l.o.g. let $\norm{B}_F=1$.}
        &= \frac{p(p+1)}{\lambda_{\min}\left(A^{p+2}\right)} \\
        &\geq \frac{p(p+1)}{\delta^{p+2}}
    \end{align*}
    Thus, $g$ is smooth with $L = \frac{p(p+1)}{\delta^{p+2}}$. 
\end{proof}

\begin{lemma}\label{lm:LSmoothLogofTrace}
    The function $k(X) = \log\left(\Tr(X^{-p})\right)$ is smooth on $D:= \{X\in\pdCone^n \, \mid\, \alpha I \succcurlyeq X \succcurlyeq \delta I \}$. 
\end{lemma}
\begin{proof}
    Let $A \in \pdCone^n$, $B \in \mathbb{S}^n$ and $t\in\R$ s.t. $A+tB\in\pdCone^n$. We define $h(t) = k(A+tB) = \log\left(\Tr((A+tB)^{-p})\right)$.
    By \cref{lm:ConvexityLogofTrace}, we have 
    \begin{align*}
        h''(t)_{\restriction t=0} &= p \frac{(p+1) \Tr\left(A^{-p-2}BB\right)\Tr\left(A^{-p}\right) - p \Tr\left(A^{-p-1}B\right)^2}{\Tr\left(A^{-p}\right)^2}\\
        &\leq p \frac{(p + 1) \Tr\left(A^{-p}\right)\Tr\left(A^{-p-2}BB\right)}{\Tr(A^{-p})^2}  \\
        &= p(p+1) \frac{\sum\limits_{i=1}^n \lambda_i\left(A^{-p-2}BB\right)}{\Tr\left(A^{-p}\right)} \\
        \intertext{Note that $BB$ is PSD. The proof of \cref{lm:EigValProdPosDefMatrix} can be adapted for one of matrices being symmetric and the other positive definite.}
        &\leq p(p+1) \frac{\lambda_{\max}\left(A^{-p-2}\right)\Tr(BB)}{\Tr\left(A^{-p}\right)} \\
        \intertext{Observe that $Tr(BB) = \norm{B}_F^2$. Without loss of generality, we can assume $\norm{B}_F=1$.}
        &\leq p(p+1) \frac{\lambda_{\max}\left(A^{-p-2}\right)}{n \lambda_{\min}(A^{-p})} \\
        &= \frac{p(p+1)}{n} \frac{\lambda_{\max}(A^p)}{\lambda_{\min}(A^{p+2})} \\
        &\leq \frac{p(p+1)}{n} \frac{\alpha^p}{\delta ^{p+2}} \\
        &= L
    \end{align*}
\end{proof}

\begin{remark}
    In our case, the minimum eigenvalue is bounded around the optimum. 
    However, it is not bounded in general for our feasible region unless we can introduce domain cuts.
    Hence, \cref{lm:LSmoothLogofTrace,th:LSmoothLogDet,th:LSmoothLogDet} are not straight away useful to us.
\end{remark}

\subsection{Results on positive semi-definite matrices}
We present two known results on the eigenvalues of the sum and product of positive definite matrices for completeness.
\begin{lemma}[Courant-Fischer min-max Theorem]\label{lm:EigValSumSymMatrix}
    Let $A, B \in \R^{n\times n}$ symmetric matrices and let $\lambda_1(A)\leq \dots\leq \lambda_n(A)$ and $\lambda_1(B)\leq\dots\leq\lambda_n(B)$ denote their respective eigenvalues in increasing order. Then, we have the following inequalities concerning the eigenvalues of their sum $A+B$:
    \begin{equation}\label{eq:BoundSumEigValSum}
        \lambda_k(A) + \lambda_1(B) \leq \lambda_k(A+B) \leq \lambda_k(A) + \lambda_n(B)
    \end{equation}
\end{lemma}
\begin{proof}
    Firstly, we know that
    \begin{align*}
    \lambda_1(A) &\leq \frac{\innp{A}{\vx}}{\innp{\vx}{\vx}} \leq \lambda_n(A) \quad \forall \vx\in\R^n\backslash\{0\}. \\
    \intertext{By Conrad-Fisher Min-Max Theorem, we have}
    \lambda_k(A+B) &= \min_{\substack{F\subset \R^n \\ \dim F = k}} \left(\max_{\vx\in F\backslash \{0\}} \frac{\innp{(A+B)\vx}{\vx}}{\innp{\vx}{\vx}}\right) \\
    &= \min_{\substack{F\subset \R^n \\ \dim F = k}} \left(\max_{\vx\in F\backslash \{0\}} \frac{\innp{A\vx}{\vx}}{\innp{\vx}{\vx}} + \frac{\innp{B\vx}{\vx}}{\innp{\vx}{\vx}}\right). \\
    \intertext{Most likely, there will not be an $\vx$ attaining the maximum simultaneously for summands. Using the above inequality for the eigenvalues, we get}
    &\geq \lambda_1(B) + \min_{\substack{F\subset \R^n \\ \dim F = k}} \left(\max_{\vx\in F\backslash \{0\}} \frac{\innp{A\vx}{\vx}}{\innp{\vx}{\vx}}\right) \\
    &= \lambda_1(B) + \lambda_k(A).
    \end{align*}
    The proof in the other direction follows analogously.
\end{proof}

\begin{lemma}\label{lm:EigValProdPosDefMatrix}
    Let $A,B \in\pdCone^n$ and let $\lambda_1(A)\leq \dots\leq \lambda_n(A)$ and $\lambda_1(B)\leq\dots\leq\lambda_n(B)$ denote their respective eigenvalues in increasing order. Then, we have the following inequalities concerning the eigenvalues of $AB$:
    \begin{equation}\label{eq:BoundProductEigValPosDef}
        \lambda_k(A)\lambda_1(B) \leq \lambda_k(AB) \leq \lambda_k(A)\lambda_n(B) \quad  \forall k\in[n].
    \end{equation} 
\end{lemma}
\begin{proof}
    Since $A$ and $B$ are both positive definite, they both have an invertible root, i.e.~there exist matrices $\sqrt{A}$ and $\sqrt{B}$ with $A=\sqrt{A}\sqrt{A}$ and $B=\sqrt{B}\sqrt{B}$. Observe that $\lambda_k(AB)=\lambda_k(\sqrt{B}A\sqrt{B})$ and by definition of eigenvalues, we have
    \[ \lambda_1(B) \leq \frac{\innp{B\vx}{\vx}}{\innp{\vx}{\vx}} \leq \lambda_n(B) \quad  \forall \vx\in\R^n\backslash \{0\}\].
    Thus, we find
    \begin{align*}
        \lambda_k(AB) &= \lambda_k(\sqrt{B}A\sqrt{B}) \\
        &= \min_{\substack{F\subset \R^n \\ \dim F= k}} \left(\max_{\vx\in F\backslash \{0\}} \frac{\innp{\sqrt{B}A\sqrt{B}\vx}{\vx}}{\innp{\vx}{\vx}}\right) \\
        &= \min_{\substack{F\subset \R^n \\ \dim F= k}} \left(\max_{\vx\in F\backslash \{0\}} \frac{\innp{A\sqrt{B}\vx}{\sqrt{B}\vx}}{\innp{\sqrt{B}\vx}{\sqrt{B}\vx}} \frac{\innp{B\vx}{\vx}}{\innp{\vx}{\vx}}\right) \\
        \intertext{Using the above inequality for the eigenvalues of $B$ and the fact that the root of $B$ is invertible, we get}
        &\geq  \lambda_1(B) \min_{\substack{F\subset \R^n \\ \dim F= k}} \left(\max_{\vx\in F\backslash \{0\}} \frac{\innp{A\vx}{\vx}}{\innp{\vx}{\vx}} \right) \\
        &= \lambda_1(B) \lambda_k(A)
    \end{align*}
    The other side of the statement follows analogously. Note that we technically only use that $B$ is regular. The matrix $A$ can be positive semi-definite.
\end{proof}






\subsection{Notes on the Frobenius norm}
\begin{lemma}
    \label{lm:BoundsFrobeniusNorm}
    Let $A\in\R^{m\times n}$ and $B\in\R^{n \times p}$. Then,
    \begin{equation}
        \sigma_{\min}(A) \|B\|_F \leq \|AB\|_F \leq \sigma_{\max} \|B\|_F
    \end{equation}
\end{lemma}
\begin{proof}
    By the definition of the Frobenius norm, we have
    \begin{align*}
    \|AB\|_F^2 &= \Tr\left((AB)^\intercal AB\right) \\
    &= \Tr(A^\intercal A B^\intercal B). \\
    \intertext{Note that both $A^\intercal A$ and $B^\intercal B$ are positve semi-definite. Thus, we can use the result in \cite{fang1994inequalities}:}
    & \geq \lambda_{\min}(A^\intercal A) \Tr(B^\intercal B) \\
    &= \sigma_{\min}(A) \|B\|_F^2.\\
    \intertext{The norm is nonnegative in any case and by definition singular values are also nonnegative. Thus, we can take the root and get}
    \sigma_{\min}(A) \|B\|_F \leq \|AB\|_F. 
    \end{align*}
    The other direction follows analogously.
\end{proof}

\section{Custom branch-and-bound with coordinate descent for OEDP}\label{sec:AppendixCustomBnB}

In this section, we showcase the changes we made to the solver set-up in \cite{ahipacsaouglu2021branch}.
Firstly, we state and solve the problems of interest as minimization problems. 
Secondly, we forgo the $\vw N \in \Z \,\forall i\in[m]$ in favor of directly having the variables $x$ to encode if and how often the experiments should be run.
The D-Optimal Problem is then stated as
\begin{align*}
        \min_{\vx} \, & \, -\log\det\left((X(\vx))\right) \\
        \text{s.t. } \, & \sum\limits_{i=1}^m x_i = N  \\
        & \vx \geq 0 \\
        & x_i \in \Z \, \forall i\in[m].
\end{align*}
The A-Optimal Problem is formulated as
\begin{align*}
        \min_{\vx} \, & \, \log\left(\Tr(X^{-1})\right) \\
        \text{s.t. } \, & \sum\limits_{i=1}^m x_i = N  \\
        & \vx \geq 0 \\
        & x_i \in \Z \, \forall i\in[m].
\end{align*}

Additionally, we have adapted the step size for the A-Optimal Problem. For the update in the objective, we get
\begin{align*}
    -\log(\Tr(X(\vx_+))) &= -\log(\Tr(X(\vx)^{-1}) - \frac{\theta A + \theta^2B}{1 + \theta C - \theta^2D}) \\
    \intertext{compare \citet[Section~4.2]{ahipacsaouglu2021branch}. The constant $A$, $B$, $C$ and $D$ are defined as}
    A &= \zeta_j -\zeta_k \\
    B &= 2\omega_{jk}\zeta_{jk} - \omega_j\zeta_k - \omega_k\zeta_j \\
    C &= \omega_j - \omega_k \\
    D &= \omega_j\omega_k - \omega_{jk}^2 \\
    \intertext{where $\omega_j = \vvv_j^\intercal X(x)^{-1} \vvv_j$ and $\zeta_j = \vvv_j^\intercal X(x)^{-2}\vvv_j$, and $\omega_{jk} = \vvv_j^\intercal X(x)^{-1} \vvv_k$ and $\zeta_{jk} = \vvv_j^\intercal X^{-2}\vvv_k$. The objective value should be improving.}
    -\log\Tr(X(\vx)^{-1}) &\geq -\log\Tr(X(\vx_+)^{-1}) \\
    \Tr(X(\vx)^{-1}) &\leq \Tr(X(\vx_+)^{-1}) \\
    &= \Tr(X(\vx)^{-1}) - \frac{\theta A + \theta^2B}{1 + \theta C - \theta^2D} \\
    \intertext{Hence, we want to maximise the function $q(\theta) = \frac{\theta A + \theta^2B}{1 + \theta C - \theta^2D}$. Let $\Delta = AD + BC$. By computing the first and second derivatives, it is easy to verify that we have}
    \theta^* &= -\frac{B + \sqrt{B^2 - A \Delta}}{\Delta} \\
    \intertext{if $\Delta \neq 0$. If $\Delta=0$ but $B$ is not, then we have}
    \theta^* &= - \frac{A}{2B}. \\
    \intertext{If both $B$ and $\Delta$ are zero, then we choose $\theta$ as big as possible given the bound constraints.}
    \theta &= \begin{cases}\min\left\{ -\frac{B+\sqrt{B^2-A\Delta}}{\Delta}, u_j-x_j, x_k-l_k\right\} & \text{if } \Delta \neq 0 \\
        \min\left\{-\frac{A}{2B}, u_j-x_j, x_k-l_k\right\} & \text{if } \Delta = 0 \text{ and } B \neq 0 \\
        \min\{u_j- x_j, x_k - l_k\} & \text{if } \Delta=B=0 \text{ and } A > 0\end{cases}
\end{align*}

\section{BPCG convergence under sharpness conditions}\label{sec:AppendixLinConvergenceBPCG}

In this appendix, we prove linear convergence of BPCG on polytopes when the objective function is \emph{sharp} or presents a Hölderian error bound.
For the definition of sharpness, see \cref{def:Sharpness}.
The following proofs rely on a geometric constant of a polytope $\mathcal{P}$.

\begin{definition}[Pyramidal Witdh \cite{braun2022conditional}]
The \textit{pyramidal width} of a polytope $\mathcal{P}\subset \R^n$ is the largest number $\delta$ such that for any set $\mathcal{S}$ of some vertices of $\mathcal{P}$, we have 
\begin{align*}
    \innp{\vw}{\vvv^A - \vvv^{FW}} \geq \delta \frac{\innp{\vw}{\vx-\vy}}{\norm{\vx-\vy}} \quad \forall \vx\in\conv\left(\mathcal{S}\right) \; \forall \vy \in \mathcal{P} \; \forall \vw \in \R^n
\end{align*}
where $\vvv^A := \argmax_{\vvv\in\mathcal{S}} \innp{\vw}{\vvv}$ and $\vvv^FW := \argmin_{\vvv\in\mathcal{S}} \innp{\vw}{\vvv}$.
\end{definition}

\begin{theorem}[Linear convergence of BPCG]
    \label{th:LinConvBPCG}
Let $P$ be a polytope of pyramidal width $\delta$ and diameter $D$, and let $f$ be an $L$-smooth and $(M,\theta)$-sharp function over $P$.
Noting $f^* = \min_{\vx\in P} f(\vx)$ and $\{\vx_t\}_{i=0}^T$ the iterations of the BPCG algorithm, then it holds that:
\begin{align*}
    f(\vx_t) - f^* \leq \begin{cases}
        c_0 (1-c_1)^{2t-1} & \text{ if } 1 \leq t \leq \frac{t_0}{2} \\
        \frac{(c_1 / c_2)^{1/\alpha}}{ \left(1 + c_1 \alpha (2t - t_0)\right)^{1/\alpha}} & \text{ if } t \geq \frac{t_0}{2}.
    \end{cases}
\end{align*}
with parameters:
\begin{align*}
    & c_1 = \frac12 \\
    & \alpha = 2-2\theta \\
    & c_2 = \frac{\delta^2}{8LD^2 M^2}\\
    & t_0 = \max \left\{\left\lfloor \mathrm{log}_{1-c_1} \left( \frac{(c_1/c_2)^{1/\alpha}}{c_0} \right)\right\rfloor + 2, 1 \right\}.
\end{align*}
    
\end{theorem}
\begin{proof}
We will denote $h_t = f(\vx_t) - f^*$. The proof adapts that of \citet[Theorem 3.2]{tsuji2021sparser} to the case of a sharp objective.
\citet[Lemma 3.5]{tsuji2021sparser} establishes that at any step $t$, the following holds:
\begin{align}\label{eq:progressineq}
    2\innp{\nabla f(\vx_t)}{\vd_t} \geq \innp{\nabla f(\vx_t)}{\vvv_t^{A} - \vvv_t^{FW}}.
\end{align}
\noindent
Let us first assume that $t$ is not a drop step.

\noindent
If $t$ follows Case~\ref{case:a} of \cref{lemma:minimumdecrease}, then
\begin{align*}
    h_t - h_{t+1} & \geq \frac{\innp{\nabla f(\vx_t)}{\vd_t}^2}{2LD^2} && \text{by }\eqref{eq:mindecreaseA}\\
                  & \geq \frac{\innp{\nabla f(\vx_t)}{\vvv_t^A - \vvv^{FW}_t}^2}{8LD^2} && \text{by }\eqref{eq:progressineq} \\
                  & \geq h_t^{2-2\theta} \frac{\delta^2}{8LD^2 M^2} && \text{by } \eqref{eq:geometricsharpness}.
\end{align*}
If $t$ follows Case~\ref{case:b} of \cref{lemma:minimumdecrease}, then:
\begin{align*}
    h_t - h_{t+1} & \geq \frac{1}{2} \innp{\nabla f(\vx_t)}{\vd_t} \geq \frac{1}{2} h_t,
\end{align*}
where the second inequality comes from $d_t$ being a FW step and from the FW gap being an upper bound on the primal gap.\\

Denoting with $\mathcal{T}_{\mathrm{drop}} \subseteq [T]$ the set of drop steps, we can already observe that the sequence $\{h_{t}\}_{t \in [T] \backslash \mathcal{T}_{\mathrm{drop}}}$
follows the hypotheses of \cref{lemma:powersequence}, with parameters $c_1, c_2, \alpha$ defined in the statement of the theorem.\\

We can then observe that the number of drop steps is bounded by the number of FW steps on the sequence. This implies in the worst case that in $2t, t > 1$ iterations,
we can reach the primal gap of $t$ FW and pairwise steps combined. Let us denote $\tau$ the number of steps that were not drop steps up to iteration $t$, then:
\begin{align*}
    h_t &\leq \begin{cases}
    c_0 (1-c_1)^{\tau-1} & 1 \leq \tau \leq t_0 \\
    \frac{(c_1 / c_2)^{1/\alpha}}{ \left(1 + c_1 \alpha (\tau - t_0)\right)^{1/\alpha}} & \tau \geq t_0,
\end{cases}\\
    & \leq \begin{cases}
        c_0 (1-c_1)^{2t-1} & 1 \leq t \leq \frac{t_0}{2} \\
        \frac{(c_1 / c_2)^{1/\alpha}}{ \left(1 + c_1 \alpha (2t - t_0)\right)^{1/\alpha}} & t \geq \frac{t_0}{2}.
    \end{cases}
\end{align*}
with $t_0$ as defined in \eqref{eq:t0def}.

\end{proof}

\begin{lemma}[Geometric sharpness \citet{braun2022conditional}, Lemma 3.32]
Let $P$ be a polytope with pyramidal width $\delta$ and let $f$ be $(M,\theta)$-sharp over $P$.
Given $S \subseteq \mathrm{Vert}(P)$ such that $\vx\in \mathrm{conv}(S)$ and:
\begin{align*}
    & \vvv^{FW} \in \argmin_{\vvv\in P} \innp{\nabla f(\vx)}{\vvv}\\
    & \vvv^{A} \in \argmax_{\vvv\in P} \innp{\nabla f(\vx)}{\vvv}.
\end{align*}
Then:
\begin{align}\label{eq:geometricsharpness}
    \frac{\delta}{M} (f(\vx) - f^*)^{(1-\theta)} \leq \innp{\nabla f(\vx)}{\vvv^A - \vvv^{FW}}.
\end{align}

\end{lemma}

\begin{lemma}[Minimum decrease \citet{tsuji2021sparser}, Lemma 3.4]\label{lemma:minimumdecrease}
Suppose that $t$ is not a drop step, and let $\alpha_t^* = \frac{\innp{\nabla f(\vx_t)}{\vd_t}}{L \norm{\vd_t}^2}$.
\begin{enumerate}[label=(\alph*)]
    \item If step $t$ is either a FW step with $\alpha_t^* < 1$ or a pairwise step, then:
    \begin{align}\label{eq:mindecreaseA}
        f(\vx_t) - f(\vx_{t+1}) \geq \frac{\innp{\nabla f(\vx_t)}{\vd_t}^2}{2LD^2}.
    \end{align}\label{case:a}
    \item If step $t$ is a FW step with $\alpha_t > 1$, we have:
    \begin{align}\label{eq:mindecreaseB}
        f(\vx_t) - f(\vx_{t+1}) \geq \frac{1}{2} \innp{\nabla f(\vx_t)}{\vd_t}.
    \end{align}\label{case:b}
\end{enumerate}
\end{lemma}

\begin{lemma}[Convergence rates from contraction inequalities, \citet{braun2022conditional}, Lemma 2.21]\label{lemma:powersequence}
Let $\{h_t\}_{t\geq 1}$ be a sequence of positive numbers and $c_0, c_1, c_2, \alpha$
such that:

\begin{align*}
\begin{cases}
c_1 < 1 \\
h_1 \leq c_0\\
h_t - h_{t+1} \geq h_t \min \{c_1, c_2 h_t^{\alpha}\} \, \forall t \geq 1,
\end{cases}
\end{align*}
then,
\begin{align*}
    & h_t \leq \begin{cases}
    c_0 (1-c_1)^{t-1} & 1 \leq t \leq t_0 \\
    \frac{(c_1 / c_2)^{1/\alpha}}{ \left(1 + c_1 \alpha (t - t_0)\right)^{1/\alpha}} & t \geq t_0,
\end{cases}
\end{align*}
where
\begin{align}\label{eq:t0def}
& t_0 = \max \left\{\left\lfloor \mathrm{log}_{1-c_1} \left( \frac{(c_1/c_2)^{1/\alpha}}{c_0} \right)\right\rfloor + 2, 1 \right\}.
\end{align}

\end{lemma}

Finally, we can exploit the recent results of \citet{zhao2025new} on AFW for log-homogeneous self-concordant barrier functions to show that on
D-Opt Problem, BPCG converges at a linear rate in primal and FW gaps.
\begin{proposition}[Linear convergence of BPCG on D-Opt \citep{zhao2025new}]\label{prop:linconvdopt}
The blended pairwise conditional gradient contracts the primal gap $h_t$ at a linear rate
$h_t \leq h_0 \exp(-ct)$
with $h_0$ the primal gap of the initial domain-feasible solution and $c$ a constant depending on the design matrix and polytope.
\end{proposition}
\begin{proof}
The core observation required to extend \citet[Theorem 4.1]{zhao2025new} is that the gap produced by BPCG upper-bounds the one produced by AFW.
Let $\vd^{\text{BPCG}},\vd^{\text{AFW}}$ be the search directions produced by BPCG, AFW respectively.
A drop step can occur in both algorithms in at most half of the iterations, since each drop step reduces the size of the active set by one.
We can focus on ``effective'' steps in which no vertex is dropped.
BPCG either takes a local pairwise step with direction:
\begin{align*}
    & \vd^{\text{BPCG}} = \va - \vs\\
    \text{with} \; & \va = \argmax_{\vvv \in \mathcal{S}_t} \innp{\nabla f(\vx)}{\vvv} \\
     \; & \vs = \argmin_{\vvv \in \mathcal{S}_t} \innp{\nabla f(\vx)}{\vvv},
\end{align*}
and $\mathcal{S}_t$ the local active set. Using the same notation, $\vd^{\text{AFW}} = \va - \vx$.
The quantity determining whether a local pairwise step is selected is $\innp{\nabla f(\vx)}{\vd^{\text{BPCG}}}$, which needs to be greater than
the FW gap $\innp{\nabla f(\vx)}{\vx - \vvv}$ at the current point.
Regardless of the direction taken by BPCG, \citet[Lemma 3.5]{tsuji2021sparser} provides:
\begin{align*}
    \innp{\nabla f(\vx)}{\vd} \geq 1/2 \innp{\nabla f(\vx)}{\va - \vvv}
\end{align*}
which is precisely the only inequality required from effective AFW steps.
The rest of the proofs from \citet[Theorem 4.1, Theorem 4.2]{zhao2025new} follow without relying on the AFW descent direction specifically.
\end{proof}
\cref{prop:linconvdopt} unfortunately relies on the fact that the D-Opt objective is a log-barrier in addition to being self-concordant.
Since this is not the case for the trace inverse function, we cannot transpose this result to the A-Opt Problem.

\section{Computational experiments}\label{sec:AppendixComEx}

In this appendix section, we present additional computational experiments.
\cref{tab:SummaryByDifficultyExt} show an extended comparison of the performance between the different solution approaches on the A-Optimal Problem, D-Optimal Problem, A-Fusion Problem and D-Fusion Problem.
In \cref{tab:SummaryByDifficultyGTI,tab:SummaryByDifficultyGTIF}, the results of the computations on the GTI-Optimal and GTI-Fusion Problem are shown.
The \cref{tab:SummaryByDifficultySettings1,tab:SummaryByDifficultySettings2,tab:SummaryByDifficultyLinesearch} compare the performance of Boscia using different settings.

\begin{sidewaystable}[!ht] \scriptsize
    \centering
    \caption{\footnotesize Comparing Boscia, SCIP, Co-BnB and the direct conic formulation and second-order formulation using Pajarito on the different problems and the different data sets, i.e. A-Fusion (AF), D-Fusion (DF), A-Optimal (AO) and D-Optimal (DO). 
    One data set contains independent data, the other has correlated data. \\
    The instances for each problem are split into increasingly smaller subsets depending on their minimum solve time, i.e.~the minimum time any of the solvers took to solve it. 
    The cut-offs are at 0 seconds (all problems), took at least 10 seconds to solve, 100 s, 1000 s and lastly 2000 s. 
    Note that if none of the solvers terminates on any instance of a subset, the corresponding row is omitted from the table. \\
    The average of the relative gap does not include instances that were solved by all methods.
    The average of number nodes/cuts is only taken for the instances solved in that group. 
    The average time is taken using the geometric mean shifted by 1 second. Also note that this is the average time over all instances in that group, i.e.~it includes the time outs.} 
    \label{tab:SummaryByDifficultyExt}
    \vspace{0.1cm}
    \begin{tabular}{lllH Hrrrr Hrrrr Hrrrr Hrrrr Hrrrr} 
        \toprule
        \multicolumn{4}{l}{} & \multicolumn{5}{c}{Boscia} & \multicolumn{5}{c}{Co-BnB} & \multicolumn{5}{c}{\thead{Direct Conic}} & \multicolumn{5}{c}{SOCP} & \multicolumn{5}{c}{SCIP OA}\tabularnewline 
        
        \cmidrule(lr){5-9}
        \cmidrule(lr){10-14}
        \cmidrule(lr){15-19}
        \cmidrule(lr){20-24}
        \cmidrule(lr){25-29}

        \thead{Type} & \thead{Corr.} & \thead{Solved \\ after \\ (s)} & \thead{\#\\ inst.} & \thead{\# \\solved} & \thead{\% \\solved} & \thead{Time (s)} & \thead{relative \\ gap} & \thead{\# \\nodes} & \thead{\# \\solved} & \thead{\% \\solved} & \thead{Time (s)} & \thead{relative \\ gap} & \thead{\# \\nodes} & \thead{\# \\solved} & \thead{\% \\solved} & \thead{Time (s)} & \thead{relative \\ gap} & \thead{\# cuts} & \thead{\# \\solved} & \thead{\% \\solved} & \thead{Time (s)} & \thead{relative \\ gap } & \thead{\# cuts} & \thead{\# \\solved} & \thead{\% \\solved} & \thead{Time (s)} & \thead{relative \\ gap } & \thead{\# cuts} \\
        \midrule
        AO & no & 0 & 50 & 26 & \textbf{52 \%} & \textbf{374.74} & \textbf{0.0851} & \textbf{7340} & 24 & 48 \% & 401.67 & 51.4694 & 22211 & 4 & 8 \% & 2382.52 & 2.4436 & 112 & 10 & 20 \% & 1578.03 & Inf & 2842 & - & - & - & - & - \\
         &  & 10 & 40 & 16 & \textbf{40 \%} & \textbf{1230.45} & \textbf{0.0964} & \textbf{11775} & 14 & 35 \% & 1309.24 & 56.5936 & 37906 & 0 & 0 \% & 3600.09 & 3.3563 & - & 0 & 0 \% & 3600.02 & Inf & - & - & - & - & - & - \\
         &  & 100 & 32 & 8 & \textbf{25 \%} & 2974.04 & \textbf{0.118} & \textbf{22241} & 6 & 19 \% & \textbf{2922.12} & 66.0235 & 81489 & 0 & 0 \% & 3600.11 & 0.6943 & - & 0 & 0 \% & 3600.01 & Inf & - & - & - & - & - & - \\
         &  & 1000 & 29 & 5 & \textbf{17 \%} & \textbf{3164.43} & \textbf{0.1291} & \textbf{19766} & 3 & 10 \% & 3360.8 & 71.1272 & 86515 & 0 & 0 \% & 3600.12 & 0.7501 & - & 0 & 0 \% & 3600.01 & Inf & - & - & - & - & - & - \\
         &  & 2000 & 25 & 1 & \textbf{4 \%} & \textbf{3583.85} & \textbf{0.1482} & \textbf{22096} & 0 & 0 \% & 3602.18 & 74.692 & - & 0 & 0 \% & 3600.09 & 0.8378 & - & 0 & 0 \% & 3600.01 & Inf & - & - & - & - & - & - \\
        \midrule
        AO & yes & 0 & 50 & 34 & \textbf{68 \%} & \textbf{227.69} & \textbf{0.0336} & \textbf{7814} & 30 & 60 \% & 301.8 & 2058.4159 & 22379 & 10 & 20 \% & 1515.23 & 30.3669 & 188 & 10 & 20 \% & 1861.35 & Inf & 3623 & - & - & - & - & - \\
         &  & 10 & 40 & 24 & \textbf{60 \%} & \textbf{481.42} & \textbf{0.0348} & \textbf{10181} & 20 & 50 \% & 864.83 & 2058.4159 & 33031 & 1 & 2 \% & 3149.17 & 34.3286 & 206 & 2 & 5 \% & 3368.07 & Inf & 5752 & - & - & - & - & - \\
         &  & 100 & 29 & 13 & \textbf{45 \%} & \textbf{1421.72} & \textbf{0.0434} & \textbf{16675} & 9 & 31 \% & 2476.41 & 2058.4159 & 59473 & 0 & 0 \% & 3600.05 & 0.4903 & - & 0 & 0 \% & 3600.01 & Inf & - & - & - & - & - & - \\
         &  & 1000 & 19 & 3 & \textbf{16 \%} & \textbf{3544.51} & \textbf{0.061} & \textbf{49147} & 1 & 5 \% & 3583.18 & 2107.5216 & 89805 & 0 & 0 \% & 3600.05 & 1.2972 & - & 0 & 0 \% & 3600.01 & Inf & - & - & - & - & - & - \\
        \midrule
        AF & no & 0 & 50 & 40 & \textbf{80 \%} & 62.83 & \textbf{0.0168} & \textbf{11634} & 40 & \textbf{80 \%} & \textbf{50.92} & 30.35 & 10648 & 2 & 4 \% & 2846.5 & 0.0371 & 58 & 13 & 26 \% & 1686.2 & Inf & 5634 & 19 & 38 \% & 463.65 & 0.0302 & 801 \\
         &  & 10 & 29 & 19 & \textbf{66 \%} & 460.44 & \textbf{0.0213} & \textbf{22887} & 19 & \textbf{66 \%} & \textbf{409.38} & 37.6051 & 21718 & 0 & 0 \% & 3600.03 & 0.0548 & - & 0 & 0 \% & 3600.09 & Inf & - & 2 & 7 \% & 3172.81 & 0.0379 & 1930 \\
         &  & 100 & 19 & 9 & \textbf{47 \%} & 1644.98 & \textbf{0.0272} & \textbf{41058} & 9 & \textbf{47 \%} & \textbf{1374.15} & 42.7763 & 33254 & 0 & 0 \% & 3600.04 & 0.074 & - & 0 & 0 \% & 3600.14 & Inf & - & 0 & 0 \% & 3600.1 & 0.0485 & - \\
         &  & 1000 & 12 & 3 & \textbf{25 \%} & \textbf{3087.98} & \textbf{0.0372} & \textbf{67544} & 2 & 17 \% & 3125.27 & 48.0929 & 67087 & 0 & 0 \% & 3600.06 & 0.0967 & - & 0 & 0 \% & 3600.17 & Inf & - & 0 & 0 \% & 3600.15 & 0.0646 & - \\
         &  & 2000 & 9 & 1 & \textbf{11 \%} & \textbf{3416.18} & \textbf{0.0461} & \textbf{70407} & 0 & 0 \% & 3601.84 & 51.6913 & - & 0 & 0 \% & 3600.06 & 0.1108 & - & 0 & 0 \% & 3600.22 & Inf & - & 0 & 0 \% & 3600.19 & 0.071 & - \\
        \midrule    
        AF & yes & 0 & 50 & 13 & 26 \% & 1368.83 & \textbf{2.1573} & 28578 & 26 & \textbf{52 \%} & \textbf{311.36} & 1407.4062 & 30449 & 12 & 24 \% & 1148.87 & 9.0208 & 318 & 10 & 20 \% & 2131.81 & Inf & 4539 & 7 & 14 \% & 1480.73 & 3.4254 & 552 \\
         &  & 10 & 38 & 4 & 11 \% & 3030.17 & \textbf{2.5452} & 51328 & 14 & \textbf{37 \%} & \textbf{1336.22} & 1468.0559 & 56106 & 1 & 3 \% & 3418.49 & 10.6718 & 1086 & 0 & 0 \% & 3603.12 & Inf & - & 0 & 0 \% & 3600.04 & 3.6558 & - \\
         &  & 100 & 33 & 2 & 6 \% & 3315.04 & \textbf{2.8901} & 58354 & 9 & \textbf{27 \%} & \textbf{2303.15} & 1468.0559 & 84247 & 0 & 0 \% & 3600.08 & 12.2264 & - & 0 & 0 \% & 3603.59 & Inf & - & 0 & 0 \% & 3600.05 & 3.9173 & - \\
         &  & 1000 & 29 & 0 & 0 \% & 3600.06 & \textbf{3.2127} & - & 5 & \textbf{17 \%} & \textbf{3079.19} & 1468.0559 & 128319 & 0 & 0 \% & 3600.09 & 13.8081 & - & 0 & 0 \% & 3604.08 & Inf & - & 0 & 0 \% & 3600.05 & 4.2154 & - \\
         &  & 2000 & 25 & 0 & 0 \% & 3600.07 & \textbf{3.6098} & - & 1 & \textbf{4 \%} & \textbf{3595.22} & 1533.847 & 154005 & 0 & 0 \% & 3600.09 & 13.7875 & - & 0 & 0 \% & 3604.73 & Inf & - & 0 & 0 \% & 3600.06 & 4.5519 & - \\
        \midrule 
        DF & no & 0 & 50 & 47 & \textbf{94 \%} & \textbf{3.49} & \textbf{0.0102} & \textbf{301} & 43 & 86 \% & 22.79 & 0.0382 & 6782 & 6 & 12 \% & 1906.61 & 0.0126 & 433 & 9 & 18 \% & 2213.03 & Inf & 4787 & 25 & 50 \% & 324.11 & 0.0133 & 1155 \\
         &  & 10 & 6 & 3 & \textbf{50 \%} & \textbf{727.47} & \textbf{0.0128} & \textbf{1424} & 1 & 17 \% & 2540.85 & 0.1531 & 32841 & 0 & 0 \% & 3879.62 & 0.0201 & - & 0 & 0 \% & 3600.0 & Inf & - & 0 & 0 \% & 3600.17 & 0.0163 & - \\
         &  & 100 & 5 & 2 & \textbf{40 \%} & \textbf{1176.39} & \textbf{0.0133} & \textbf{94} & 0 & 0 \% & 3602.41 & 0.1825 & - & 0 & 0 \% & 3936.34 & 0.0213 & - & 0 & 0 \% & 3600.0 & Inf & - & 0 & 0 \% & 3600.2 & 0.0176 & - \\
        \midrule
        DF & yes & 0 & 50 & 30 & \textbf{60 \%} & \textbf{54.01} & \textbf{0.0842} & \textbf{7635} & 28 & 56 \% & 153.26 & 0.4948 & 22485 & 3 & 6 \% & 2284.06 & 0.1556 & 319 & 3 & 6 \% & 2934.61 & Inf & 3117 & 14 & 28 \% & 752.15 & 0.092 & 786 \\
         &  & 10 & 24 & 4 & \textbf{17 \%} & \textbf{2634.14} & \textbf{0.1592} & \textbf{56466} & 3 & 12 \% & 2705.02 & 0.948 & 162320 & 0 & 0 \% & 3041.2 & 0.302 & - & 0 & 0 \% & 3600.0 & Inf & - & 0 & 0 \% & 3600.05 & 0.1679 & - \\
         &  & 100 & 22 & 3 & \textbf{14 \%} & \textbf{3251.87} & \textbf{0.161} & \textbf{74805} & 2 & 9 \% & 3265.96 & 0.9635 & 239147 & 0 & 0 \% & 3603.96 & 0.258 & - & 0 & 0 \% & 3600.0 & Inf & - & 0 & 0 \% & 3600.05 & 0.1711 & - \\
         &  & 1000 & 20 & 1 & \textbf{5 \%} & \textbf{3535.72} & \textbf{0.1761} & \textbf{40687} & 0 & 0 \% & 3599.62 & 1.0588 & - & 0 & 0 \% & 3604.36 & 0.2811 & - & 0 & 0 \% & 3600.0 & Inf & - & 0 & 0 \% & 3600.05 & 0.185 & - \\
        \midrule
        DO & no & 0 & 50 & 34 & \textbf{68 \%} & \textbf{136.92} & \textbf{0.0208} & \textbf{3578} & 29 & 58 \% & 216.31 & 2.7678 & 9559 & 10 & 20 \% & 1241.37 & 0.676 & 1077 & 10 & 20 \% & 2169.4 & Inf & 4514 & - & - & - & - & - \\
         &  & 10 & 36 & 20 & \textbf{56 \%} & \textbf{618.39} & \textbf{0.0226} & \textbf{5916} & 15 & 42 \% & 991.36 & 3.3814 & 17925 & 1 & 3 \% & 2355.99 & 0.8846 & 3056 & 0 & 0 \% & 3600.0 & Inf & - & - & - & - & - & - \\
         &  & 100 & 29 & 13 & \textbf{45 \%} & \textbf{1330.76} & \textbf{0.0257} & \textbf{8559} & 8 & 28 \% & 1978.6 & 4.5636 & 27449 & 1 & 3 \% & 2126.78 & 1.3081 & 3056 & 0 & 0 \% & 3600.0 & Inf & - & - & - & - & - & - \\
         &  & 1000 & 12 & 2 & \textbf{17 \%} & \textbf{3261.11} & \textbf{0.0358} & \textbf{18628} & 0 & 0 \% & 3602.01 & 7.4215 & - & 0 & 0 \% & 2741.23 & 4.4622 & - & 0 & 0 \% & 3600.01 & Inf & - & - & - & - & - & - \\
         &  & 2000 & 8 & 1 & \textbf{12 \%} & \textbf{3515.21} & \textbf{0.0294} & \textbf{30726} & 0 & 0 \% & 3601.85 & 3.4692 & - & 0 & 0 \% & 3614.18 & 4.4622 & - & 0 & 0 \% & 3600.01 & Inf & - & - & - & - & - & - \\
        \midrule
        DO & yes & 0 & 50 & 50 & \textbf{100 \%} & \textbf{2.56} & \textbf{0.0094} & \textbf{108} & 35 & 70 \% & 101.89 & 0.0809 & 7973 & 5 & 10 \% & 1398.01 & 0.0206 & 309 & 5 & 10 \% & 2526.74 & Inf & 2871 & - & - & - & - & - \\
         &  & 10 & 10 & 10 & \textbf{100 \%} & \textbf{65.51} & \textbf{0.0099} & \textbf{439} & 0 & 0 \% & 3599.94 & 0.2101 & - & 0 & 0 \% & 1347.23 & Inf & - & 0 & 0 \% & 3600.01 & Inf & - & - & - & - & - & - \\
         &  & 100 & 4 & 4 & \textbf{100 \%} & \textbf{195.36} & \textbf{0.0099} & \textbf{828} & 0 & 0 \% & 3600.17 & 0.1967 & - & 0 & 0 \% & 1204.26 & Inf & - & 0 & 0 \% & 3600.02 & Inf & - & - & - & - & - & - \\
        \bottomrule
    \end{tabular}
\end{sidewaystable}

\begin{table}[!ht] \scriptsize
    \centering
    \caption{\footnotesize Comparing Boscia and Co-BnB on the General-Trace-Inverse (GTI) Optimal Problem for different values of $p$ and under independent and correlated data. Note the the objective function has the form $\phi_p(X(\vx)) = -1/p \log\left(\Tr\left(X(\vx)^p\right)\right)$.\\
    The instances for each problem are split into increasingly smaller subsets depending on their minimum solve time, i.e.~the minimum time any of the solvers took to solve it. The cut-offs are at 0 seconds (all problems), took at least 10 seconds to solve, 100 s, 1000 s and lastly 2000 s. Note that if none of the solvers terminates on any instance of a subset, the corresponding row is omitted from the table. \\
    The relative gap is computed for the instances on which at least one method did not terminate within the time limit. That means it excludes the instances on which all methods terminated.
    The average time is taken using the geometric mean shifted by 1 second. Also note that this is the average time over all instances in that group, i.e.~it includes the time outs.} 
    \label{tab:SummaryByDifficultyGTI}
    \vspace{0.1cm}
    \begin{tabular}{lllr HrrHH HrrHH HrrHH} 
        \toprule
        \multicolumn{4}{l}{} & \multicolumn{5}{c}{Boscia}  & \multicolumn{5}{c}{Boscia Log} & \multicolumn{5}{c}{Co-BnB} \tabularnewline 
        
        \cmidrule(lr){5-9}
        \cmidrule(lr){10-14}
        \cmidrule(lr){15-19}

        \thead{p} & \thead{Corr.} & \thead{Solved \\ after \\ (s)} & \thead{\#\\ inst.} & \thead{\# \\solved} & \thead{\% \\solved} & \thead{Time (s)} & \thead{relative \\ gap} & \thead{\# nodes} & \thead{\# \\solved} & \thead{\% \\solved} & \thead{Time (s)} & \thead{relative \\ gap} & \thead{\# nodes} & \thead{\# \\solved} & \thead{\% \\solved} & \thead{Time (s)} & \thead{relative \\ gap} & \thead{\# nodes} \\ 
        \midrule
        GTI\_25 & no & 0 & 30 & 20 & 67 \% & 162.41 & 0.0229 & 340 & 26 & \textbf{87 \%} & \textbf{27.13} & \textbf{0.0106} & 39 & 15 & 50 \% & 359.92 & 0.6731 & 8495 \\
        GTI\_25 & no & 10 & 16 & 6 & 38 \% & 1880.31 & 0.0238 & 1018 & 12 & \textbf{75 \%} & \textbf{200.69} & \textbf{0.0107} & 64 & 2 & 12 \% & 2864.51 & 0.6715 & 42354 \\
        GTI\_25 & no & 100 & 9 & 1 & 11 \% & 3437.77 & 0.0306 & 2243 & 5 & \textbf{56 \%} & \textbf{930.74} & \textbf{0.0118} & 100 & 0 & 0 \% & 3622.39 & 0.6986 \\
        GTI\_25 & no & 1000 & 5 & 0 & 0 \% & 3606.39 & 0.0414 & NaN & 1 & \textbf{20 \%} & \textbf{3222.96} & \textbf{0.0136} & 63 & 0 & 0 \% & 3640.97 & 0.5756 \\
        \midrule
        GTI\_25 & yes & 0 & 30 & 20 & 67 \% & 110.25 & 0.0205 & 130 & 23 & \textbf{77 \%} & \textbf{85.83} & \textbf{0.0126} & 164 & 18 & 60 \% & 264.22 & 0.987 & 5463  \\
        GTI\_25 & yes & 10 & 16 & 6 & 38 \% & 1461.4 & 0.0205 & 381 & 9 & \textbf{56 \%} & \textbf{896.41} & \textbf{0.0126} & 349 & 4 & 25 \% & 2614.2 & 0.987 & 20102 \\
        GTI\_25 & yes & 100 & 14 & 4 & 29 \% & 2309.53 & 0.0205 & 534 & 7 & \textbf{50 \%} & \textbf{1437.34} & \textbf{0.0126} & 429 & 2 & 14 \% & 3307.62 & 0.987 & 28981  \\
        GTI\_25 & yes & 1000 & 7 & 0 & 0 \% & 3600.83 & 0.027 & NaN & 0 & 0 \% & 3600.9 & 0.0145 & NaN & 0 & 0 \% & 3599.3 & 1.0621 \\
       \midrule
        GTI\_50 & no & 0 & 30 & 16 & 53 \% & 246.61 & 0.0625 & 1786 & 21 & \textbf{70 \%} & \textbf{85.87} & \textbf{0.0225} & 1225 & 15 & 50 \% & 430.08 & 4.7973 & 12614 \\
        GTI\_50 & no & 10 & 17 & 3 & 18 \% & 3139.96 & 0.0625 & 8321 & 8 & \textbf{47 \%} & \textbf{1197.73} & \textbf{0.0225} & 3122 & 2 & 12 \% & 3456.55 & 4.7973 & 79452  \\
        GTI\_50 & no & 100 & 14 & 0 & 0 \% & 3601.04 & 0.0662 & NaN & 5 & \textbf{36 \%} & \textbf{2242.33} & \textbf{0.0234} & 4558 & 0 & 0 \% & 3611.02 & 4.8061 \\
        GTI\_50 & no & 1000 & 10 & 0 & 0 \% & 3601.45 & 0.0856 & NaN & 1 & \textbf{10 \%} & \textbf{3581.72} & \textbf{0.0288} & 9719 & 0 & 0 \% & 3615.73 & 5.1166 \\
        \midrule
        GTI\_50 & yes & 0 & 30 & 18 & \textbf{60 \%} & \textbf{129.86} & \textbf{0.0393} & 1246 & 17 & 57 \% & 245.36 & 0.0836 & 2971 & 17 & 57 \% & 313.7 & 32.2241 & 6271 \\
        GTI\_50 & yes & 10 & 21 & 9 & \textbf{43 \%} & \textbf{675.52} & \textbf{0.0393} & 2379 & 8 & 38 \% & 1340.71 & 0.0836 & 6040 & 8 & 38 \% & 1326.2 & 32.2241 & 12810 \\
        GTI\_50 & yes & 100 & 16 & 4 & \textbf{25 \%} & \textbf{2105.9} & \textbf{0.0393} & 4835 & 3 & 19 \% & 3238.0 & 0.0836 & 14020 & 3 & 19 \% & 3193.81 & 32.2241 & 29773 \\
        \midrule
        GTI\_75 & no & 0 & 30 & 11 & 37 \% & 799.73 & 0.1606 & 349 & 13 & 43 \% & 505.86 & \textbf{0.0541} & 218 & 15 & \textbf{50 \%} & \textbf{471.95} & 35.492 & 14963 \\
        GTI\_75 & no & 10 & 25 & 6 & 24 \% & 2171.71 & 0.1606 & 596 & 8 & 32 \% & 1308.01 & \textbf{0.0541} & 328 & 10 & \textbf{40 \%} & \textbf{1228.42} & 35.492 & 22377 \\
        GTI\_75 & no & 100 & 21 & 2 & 10 \% & 3011.64 & 0.1606 & 934 & 4 & 19 \% & 2210.64 & \textbf{0.0541} & 487 & 6 & \textbf{29 \%} & \textbf{2093.83} & 35.492 & 36214 \\
        GTI\_75 & no & 1000 & 17 & 0 & 0 \% & 3601.73 & 0.1777 & NaN & 0 & 0 \% & 3602.06 & \textbf{0.0593} & NaN & 2 & \textbf{12 \%} & \textbf{3518.93} & 39.6656 & 94975 \\
        \midrule
        GTI\_75 & yes & 0 & 30 & 14 & 47 \% & 541.67 & \textbf{0.081} & 676 & 13 & 43 \% & 557.15 & 0.1052 & 506 & 15 & \textbf{50 \%} & \textbf{357.77} & 1690.0144 & 5438 \\
        GTI\_75 & yes & 10 & 25 & 9 & 36 \% & 1514.11 & \textbf{0.081} & 1025 & 8 & 32 \% & 1626.43 & 0.1052 & 808 & 10 & \textbf{40 \%} & \textbf{904.04} & 1690.0144 & 8127  \\
        GTI\_75 & yes & 100 & 21 & 5 & 24 \% & 2446.72 & \textbf{0.081} & 1475 & 4 & 19 \% & 2739.52 & 0.1052 & 1305 & 6 & \textbf{29 \%} & \textbf{1608.59} & 1690.0144 & 12877 \\
        GTI\_75 & yes & 1000 & 16 & 0 & 0 \% & 3600.46 & \textbf{0.0855} & NaN & 0 & 0 \% & 3600.45 & 0.1108 & NaN & 1 & \textbf{6 \%} & \textbf{3440.58} & 1795.6396 & 57665 \\
        GTI\_75 & yes & 2000 & 15 & 0 & 0 \% & 3600.48 & 0.0895 & NaN & 0 & 0 \% & 3600.48 & 0.1149 & NaN & 0 & 0 \% & 3599.49 & 1915.3473  \\
        \midrule
        GTI\_150 & no & 0 & 30 & 13 & 43 \% & 659.25 & 1.2623 & 2344 & 12 & 40 \% & \textbf{576.51} & \textbf{0.2686} & 1308 & 15 & \textbf{50 \%} & 619.95 & 96345.6059 & 14548  \\
        GTI\_150 & no & 10 & 25 & 8 & 32 \% & 1908.22 & 1.2623 & 3756 & 7 & 28 \% & 1673.37 & \textbf{0.2686} & 2182 & 10 & \textbf{40 \%} & \textbf{1669.28} & 96345.6059 & 21752 \\
        GTI\_150 & no & 1000 & 17 & 0 & 0 \% & 3600.54 & 1.3359 & NaN & 0 & 0 \% & 3601.19 & \textbf{0.2686} & NaN & 2 & \textbf{12 \%} & \textbf{3571.66} & 102012.9939 & 77998  \\
        \midrule
        GTI\_150 & yes & 0 & 30 & 14 & 47 \% & 418.43 & 0.2003 & 4555 & 16 & \textbf{53 \%} & \textbf{306.7} & \textbf{0.0489} & 1474 & 14 & 47 \% & 506.3 & 7.351349315159e8 & 3818\\
        GTI\_150 & yes & 10 & 24 & 8 & 33 \% & 1413.27 & 0.2003 & 7638 & 10 & \textbf{42 \%} & \textbf{1076.7} & \textbf{0.0489} & 2340 & 8 & 33 \% & 1539.25 & 7.351349315159e8 & 6594 \\
        GTI\_150 & yes & 100 & 19 & 3 & 16 \% & 2602.93 & 0.2003 & 9601 & 5 & \textbf{26 \%} & \textbf{2474.23} & \textbf{0.0489} & 4289 & 3 & 16 \% & 2546.07 & 7.351349315159e8 & 12200 \\
        GTI\_150 & yes & 1000 & 16 & 0 & 0 \% & 3600.04 & 0.2003 & NaN & 2 & \textbf{12 \%} & \textbf{3520.88} & \textbf{0.0489} & 6850 & 0 & 0 \% & 3599.27 & 7.351349315159e8  \\
        \midrule
        GTI\_200 & no & 0 & 30 & 13 & 43 \% & 661.29 & 3.8118 & 3386 & 15 & \textbf{50 \%} & \textbf{562.4} & \textbf{0.5984} & 7311 & 15 & \textbf{50 \%} & 630.18 & 2.37363945824e7 & 15719  \\
        GTI\_200 & no & 10 & 25 & 8 & 32 \% & 1880.56 & 3.8118 & 5442 & 10 & \textbf{40 \%} & \textbf{1601.07} & \textbf{0.5984} & 10916 & 10 & \textbf{40 \%} & 1642.44 & 2.37363945824e7 & 23489 \\
        GTI\_200 & no & 100 & 25 & 8 & 32 \% & 1880.56 & 3.8118 & 5442 & 10 & \textbf{40 \%} & \textbf{1601.07} & \textbf{0.5984} & 10916 & 10 & \textbf{40 \%} & 1642.44 & 2.37363945824e7 & 23489  \\
        GTI\_200 & no & 1000 & 18 & 1 & 6 \% & 3510.95 & 3.8118 & 17989 & 3 & \textbf{17 \%} & 3342.37 & \textbf{0.5984} & 28212 & 3 & \textbf{17 \%} & \textbf{3207.77} & 2.37363945824e7 & 61652 \\
        GTI\_200 & no & 2000 & 16 & 0 & 0 \% & 3601.01 & 4.0489 & NaN & 1 & \textbf{6 \%} & 3590.35 & \textbf{0.6352} & 39961 & 1 & \textbf{6 \%} & \textbf{3521.64} & 2.52199192427e7 & 85339 \\
        \midrule
        GTI\_200 & yes & 0 & 30 & 13 & 43 \% & 644.53 & 1.1 & 25219 & 16 & \textbf{53 \%} & \textbf{297.98} & \textbf{0.0373} & 4067 & 14 & 47 \% & 611.8 & 3.672262322023188e12 & 6479 \\
        GTI\_200 & yes & 10 & 24 & 7 & 29 \% & 2068.4 & 1.1 & 42466 & 10 & \textbf{42 \%} & \textbf{1062.97} & \textbf{0.0373} & 6484 & 8 & 33 \% & 1851.58 & 3.672262322023188e12 & 11224 \\
        GTI\_200 & yes & 100 & 20 & 4 & 20 \% & 2632.63 & 1.1 & 46221 & 6 & \textbf{30 \%} & \textbf{2026.91} & \textbf{0.039} & 10402 & 4 & 20 \% & 2734.27 & 3.672262322023188e12 & 18224 \\
        GTI\_200 & yes & 1000 & 16 & 0 & 0 \% & 3600.06 & 1.1 & NaN & 2 & \textbf{12 \%} & \textbf{3547.5} & \textbf{0.039} & 24256 & 0 & 0 \% & 3599.23 & 3.672262322023188e12  \\
        \bottomrule
    \end{tabular}
\end{table}

\begin{table}[!ht] \scriptsize
    \centering
    \caption{\footnotesize Comparing Boscia, Co-BnB and SCIP on the General-Trace-Inverse (GTI) Fusion Problem for different values of $p$ and under independent and correlated data. Note the the objective function has the form $\phi_p(X(\vx)) = -1/p \log\left(\Tr\left(C + X(\vx)^p\right)\right)$.\\
    The instances for each problem are split into increasingly smaller subsets depending on their minimum solve time, i.e.~the minimum time any of the solvers took to solve it. The cut-offs are at 0 seconds (all problems), took at least 10 seconds to solve, 100 s, 1000 s and lastly 2000 s. Note that if none of the solvers terminates on any instance of a subset, the corresponding row is omitted from the table. \\
    The relative gap is computed for the instances on which at least one method did not terminate within the time limit. That means it excludes the instances on which all methods terminated.
    The average time is taken using the geometric mean shifted by 1 second. Also note that this is the average time over all instances in that group, i.e.~it includes the time outs.} 
    \label{tab:SummaryByDifficultyGTIF}
    \vspace{0.1cm}
    \begin{tabular}{lllr HrrHH HrrHH HrrHH HrrHH} 
        \toprule
        \multicolumn{4}{l}{} & \multicolumn{5}{c}{Boscia} & \multicolumn{5}{c}{Boscia Log} & \multicolumn{5}{c}{Co-BnB} & \multicolumn{5}{c}{SCIP OA}\tabularnewline 
        
        \cmidrule(lr){5-9}
        \cmidrule(lr){10-14}
        \cmidrule(lr){15-19}
        \cmidrule(lr){20-24}

        \thead{p} & \thead{Corr.} & \thead{Solved \\ after \\ (s)} & \thead{\#\\ inst.} & \thead{\# \\solved} & \thead{\% \\solved} & \thead{Time (s)} & \thead{relative \\ gap} & \thead{\# nodes} & \thead{\# \\solved} & \thead{\% \\solved} & \thead{Time (s)} & \thead{relative \\ gap} & \thead{\# nodes} & \thead{\# \\solved} & \thead{\% \\solved} & \thead{Time (s)} & \thead{relative \\ gap } & \thead{\# nodes}  & \thead{\# \\solved} & \thead{\% \\solved} & \thead{Time (s)} & \thead{relative \\ gap} & \thead{\# cuts}  \\
        \midrule
        GTIF\_25 & no & 0 & 30 & 27 & 90 \% & 6.57 & 0.0087 & 13 & 30 & \textbf{100 \%} & \textbf{1.71} & \textbf{0.0066} & 6 & 15 & 50 \% & 725.52 & 0.1107 & 309 & 14 & 47 \% & 376.08 & 0.0085 & 671 \\
        GTIF\_25 & no & 10 & 3 & 2 & 67 \% & 470.21 & 0.0104 & 87 & 3 & \textbf{100 \%} & \textbf{56.43} & \textbf{0.007} & 25 & 0 & 0 \% & 3755.69 & 0.0914 & NaN & 0 & 0 \% & 3600.18 & 0.0085 & NaN \\
        GTIF\_25 & no & 100 & 1 & 0 & 0 \% & 3602.77 & 0.0149 & NaN & 1 & \textbf{100 \%} & \textbf{137.03} & 0.0094 & 35 & 0 & 0 \% & 3834.75 & 0.0692 & NaN & 0 & 0 \% & 3600.01 & \textbf{0.0089} & NaN \\
        \midrule
        GTIF\_25 & yes & 0 & 30 & 16 & 53 \% & \textbf{138.78} & 0.2337 & 231 & 17 & \textbf{57 \%} & 184.83 & \textbf{0.1003} & 455 & 8 & 27 \% & 1449.73 & 1.1622 & 408 & 8 & 27 \% & 927.92 & 0.2099 & 1090 \\
        GTIF\_25 & yes & 10 & 19 & 5 & 26 \% & 1299.73 & 0.2812 & 258 & 6 & \textbf{32 \%} & \textbf{1260.16} & \textbf{0.1195} & 863 & 0 & 0 \% & 3601.55 & 1.1415 & NaN & 0 & 0 \% & 3600.03 & 0.249 & NaN \\
        GTIF\_25 & yes & 100 & 16 & 2 & 12 \% & 2722.54 & 0.332 & 569 & 3 & \textbf{19 \%} & \textbf{2572.48} & \textbf{0.1401} & 1666 & 0 & 0 \% & 3601.84 & 1.1215 & NaN & 0 & 0 \% & 3600.04 & 0.292 & NaN \\
        GTIF\_25 & yes & 1000 & 14 & 0 & 0 \% & 3600.55 & 0.378 & NaN & 1 & \textbf{7 \%} & \textbf{3566.88} & \textbf{0.1587} & 4329 & 0 & 0 \% & 3602.03 & 1.0815 & NaN & 0 & 0 \% & 3600.04 & 0.3294 & NaN \\
        \midrule
        GTIF\_50 & no & 0 & 30 & 26 & 87 \% & 9.1 & 0.014 & 140 & 27 & \textbf{90 \%} & \textbf{3.63} & \textbf{0.0096} & 27 & 14 & 47 \% & 787.45 & 0.2604 & 450 & 13 & 43 \% & 349.14 & 0.0164 & 807 \\
        GTIF\_50 & no & 10 & 6 & 2 & 33 \% & 1865.24 & 0.0266 & 1029 & 3 & \textbf{50 \%} & \textbf{310.51} & \textbf{0.0111} & 111 & 0 & 0 \% & 3885.14 & 0.2425 & NaN & 0 & 0 \% & 3600.18 & 0.0245 & NaN \\
        \midrule
        GTIF\_50 & yes & 0 & 30 & 17 & \textbf{57 \%} & \textbf{175.85} & 0.586 & 2649 & 16 & 53 \% & 271.12 & \textbf{0.5412} & 3554 & 9 & 30 \% & 1547.49 & 4.087 & 559 & 6 & 20 \% & 850.41 & 0.583 & 298 \\
        GTIF\_50 & yes & 10 & 21 & 8 & \textbf{38 \%} & \textbf{1010.86} & 0.6683 & 5280 & 7 & 33 \% & 1617.58 & \textbf{0.6171} & 7673 & 0 & 0 \% & 3601.95 & 4.6985 & NaN & 0 & 0 \% & 3600.03 & 0.6649 & NaN \\
        GTIF\_50 & yes & 100 & 16 & 3 & \textbf{19 \%} & \textbf{2753.86} & 0.874 & 13070 & 2 & 12 \% & 3261.97 & \textbf{0.8069} & 22691 & 0 & 0 \% & 3602.31 & 3.9085 & NaN & 0 & 0 \% & 3600.04 & 0.8573 & NaN \\
        GTIF\_50 & yes & 1000 & 14 & 1 & \textbf{7 \%} & \textbf{3489.55} & 0.9974 & 28077 & 1 & 7 \% & 3569.66 & \textbf{0.9202} & 34945 & 0 & 0 \% & 3602.47 & 3.5731 & NaN & 0 & 0 \% & 3600.04 & 0.9672 & NaN \\
        \midrule
        GTIF\_75 & no & 0 & 30 & 21 & 70 \% & 85.39 & 0.0234 & 123 & 26 & \textbf{87 \%} & \textbf{42.47} & \textbf{0.0141} & 120 & 17 & 57 \% & 353.3 & 0.5793 & 191 & 13 & 43 \% & 415.34 & 0.0285 & 725 \\
        GTIF\_75 & no & 10 & 15 & 9 & 60 \% & 362.32 & 0.0321 & 107 & 11 & \textbf{73 \%} & \textbf{144.1} & \textbf{0.017} & 51 & 5 & 33 \% & 1347.24 & 0.568 & 517 & 4 & 27 \% & 1927.57 & 0.0351 & 1855 \\
        GTIF\_75 & no & 100 & 6 & 1 & 17 \% & 2549.51 & 0.0507 & 155 & 2 & \textbf{33 \%} & \textbf{2006.07} & \textbf{0.023} & 87 & 0 & 0 \% & 3832.95 & 0.5164 & NaN & 0 & 0 \% & 3600.04 & 0.0484 & NaN \\
        GTIF\_75 & no & 1000 & 5 & 0 & 0 \% & 3603.61 & 0.059 & NaN & 1 & \textbf{20 \%} & \textbf{3333.88} & \textbf{0.0256} & 89 & 0 & 0 \% & 3881.18 & 0.4339 & NaN & 0 & 0 \% & 3600.03 & 0.0542 & NaN \\
        \midrule
        GTIF\_75 & yes & 0 & 30 & 11 & 37 \% & 741.57 & 1.8886 & 867 & 13 & \textbf{43 \%} & \textbf{434.78} & 1.7079 & 353 & 8 & 27 \% & 1639.63 & 14.5788 & 516 & 7 & 23 \% & 1053.8 & 1.6689 & 877 \\
        GTIF\_75 & yes & 10 & 24 & 5 & 21 \% & 2351.14 & 2.0594 & 1669 & 7 & \textbf{29 \%} & \textbf{1512.43} & 1.9344 & 639 & 3 & 12 \% & 3038.52 & 15.2409 & 994 & 2 & 8 \% & 2837.69 & 1.8195 & 2464 \\
        GTIF\_75 & yes & 100 & 20 & 2 & 10 \% & 3212.49 & 2.2644 & 2294 & 3 & \textbf{15 \%} & \textbf{2855.63} & 2.2306 & 683 & 0 & 0 \% & 3602.37 & 14.5908 & NaN & 0 & 0 \% & 3600.05 & 1.9941 & NaN \\
        GTIF\_75 & yes & 1000 & 18 & 1 & 6 \% & 3567.81 & 2.5146 & 2985 & 1 & \textbf{6 \%} & \textbf{3373.72} & 2.6344 & 343 & 0 & 0 \% & 3602.86 & 13.0088 & NaN & 0 & 0 \% & 3600.06 & 2.2057 & NaN \\
        \midrule
        GTIF\_150 & no & 0 & 30 & 23 & \textbf{77 \%} & \textbf{91.83} & 0.0579 & 1772 & 20 & 67 \% & 129.57 & \textbf{0.056} & 1555 & 14 & 47 \% & 780.56 & 2.0633 & 336 & 12 & 40 \% & 406.94 & 0.0836 & 718 \\
        GTIF\_150 & no & 10 & 19 & 12 & \textbf{63 \%} & \textbf{629.52} & 0.0607 & 3304 & 9 & 47 \% & 821.17 & \textbf{0.0588} & 3081 & 4 & 21 \% & 2030.26 & 2.0633 & 814 & 2 & 11 \% & 2732.57 & 0.0871 & 2076 \\
        GTIF\_150 & no & 100 & 14 & 7 & \textbf{50 \%} & \textbf{1556.16} & 0.0716 & 5127 & 5 & 36 \% & 1922.49 & \textbf{0.0675} & 4962 & 0 & 0 \% & 3712.95 & 2.2099 & NaN & 0 & 0 \% & 3600.09 & 0.0968 & NaN \\
        GTIF\_150 & no & 1000 & 9 & 2 & \textbf{22 \%} & \textbf{3112.4} & 0.106 & 7006 & 1 & 11 \% & 3171.4 & \textbf{0.0962} & 4559 & 0 & 0 \% & 3776.35 & 2.0283 & NaN & 0 & 0 \% & 3600.13 & 0.1173 & NaN \\
        GTIF\_150 & no & 2000 & 8 & 1 & \textbf{12 \%} & \textbf{3488.09} & 0.1181 & 9879 & 0 & 0 \% & 3601.58 & \textbf{0.1069} & NaN & 0 & 0 \% & 3798.53 & 2.1485 & NaN & 0 & 0 \% & 3600.14 & 0.1287 & NaN \\
        \midrule
        GTIF\_150 & yes & 0 & 30 & 11 & 37 \% & 715.74 & 13.4006 & 8008 & 15 & \textbf{50 \%} & \textbf{214.97} & 7.6622 & 2014 & 9 & 30 \% & 1415.97 & 790.4184 & 444 & 7 & 23 \% & 885.83 & 12.5165 & 764 \\
        GTIF\_150 & yes & 10 & 20 & 3 & 15 \% & 2749.79 & 15.403 & 12472 & 5 & \textbf{25 \%} & \textbf{1837.42} & 9.4283 & 5560 & 1 & 5 \% & 2727.35 & 647.8539 & 1 & 0 & 0 \% & 3600.04 & 14.37 & NaN \\
        GTIF\_150 & yes & 100 & 17 & 3 & 18 \% & 2622.1 & 14.9582 & 12472 & 3 & \textbf{18 \%} & \textbf{2813.97} & 11.141 & 9185 & 0 & 0 \% & 3603.41 & 402.7068 & NaN & 0 & 0 \% & 3600.04 & 13.1513 & NaN \\
        GTIF\_150 & yes & 1000 & 15 & 1 & 7 \% & 3379.09 & 16.9513 & 15921 & 1 & \textbf{7 \%} & \textbf{3515.28} & 13.6146 & 17345 & 0 & 0 \% & 3603.8 & 422.537 & NaN & 0 & 0 \% & 3600.05 & 14.8476 & NaN \\
       \midrule
        GTIF\_200 & no & 0 & 30 & 22 & \textbf{73 \%} & \textbf{96.28} & \textbf{0.0877} & 2668 & 21 & 70 \% & 117.1 & 0.323 & 2985 & 17 & 57 \% & 797.56 & 4.3773 & 741 & 12 & 40 \% & 402.28 & 0.1261 & 641 \\
        GTIF\_200 & no & 10 & 20 & 12 & \textbf{60 \%} & \textbf{554.66} & \textbf{0.0877} & 4775 & 11 & 55 \% & 677.02 & 0.323 & 5516 & 7 & 35 \% & 2413.83 & 4.3773 & 1562 & 2 & 10 \% & 2710.21 & 0.1261 & 1742 \\
        GTIF\_200 & no & 100 & 15 & 7 & \textbf{47 \%} & \textbf{1536.15} & \textbf{0.1032} & 7466 & 6 & 40 \% & 1840.11 & 0.3856 & 9004 & 2 & 13 \% & 3313.81 & 5.2507 & 1512 & 0 & 0 \% & 3600.06 & 0.1442 & NaN \\
        GTIF\_200 & no & 1000 & 9 & 1 & \textbf{11 \%} & \textbf{3319.15} & \textbf{0.1654} & 11279 & 0 & 0 \% & 3601.02 & 0.636 & NaN & 0 & 0 \% & 3762.3 & 5.8163 & NaN & 0 & 0 \% & 3600.09 & 0.1857 & NaN \\
        \midrule
        GTIF\_200 & yes & 0 & 30 & 10 & 33 \% & 905.59 & 51.5626 & 19493 & 16 & \textbf{53 \%} & \textbf{202.0} & \textbf{1.4881} & 5179 & 10 & 33 \% & 1305.28 & 11737.6455 & 644 & 6 & 20 \% & 923.45 & 110.7213 & 377 \\
        GTIF\_200 & yes & 10 & 19 & 4 & 21 \% & 2561.54 & 62.4752 & 39052 & 6 & \textbf{32 \%} & \textbf{1658.86} & \textbf{1.9107} & 13405 & 1 & 5 \% & 3586.45 & 12961.237 & 2363 & 0 & 0 \% & 3600.05 & 116.5026 & NaN \\
        GTIF\_200 & yes & 100 & 18 & 4 & 22 \% & 2513.56 & 62.4752 & 39052 & 5 & \textbf{28 \%} & \textbf{2167.04} & \textbf{2.0569} & 16074 & 1 & 6 \% & 3585.71 & 12961.237 & 2363 & 0 & 0 \% & 3600.06 & 116.5026 & NaN \\
        GTIF\_200 & yes & 1000 & 14 & 1 & 7 \% & 3471.17 & 75.8606 & 74623 & 1 & \textbf{7 \%} & \textbf{3513.18} & \textbf{2.9667} & 47291 & 0 & 0 \% & 3600.85 & 15694.4177 & NaN & 0 & 0 \% & 3600.07 & 132.9516 & NaN \\
        \bottomrule
    \end{tabular}
\end{table}

\begin{sidewaystable}[!ht] \scriptsize
    \centering
    \caption{\footnotesize Comparing the different Boscia settings on the different problems and the different data sets, i.e. A-Fusion (AF), D-Fusion (DF), A-Optimal (AO) and D-Optimal (DO). One data set contains independent data, the other has correlated data. \\
    The settings displayed here are the default, so all options are off, using heuristics, using SCIP to model the constraints and tightening the lazifaction for the BPCG algorithm solving the nodes. \\
    The instances for each problem are split into increasingly smaller subsets depending on their minimum solve time, i.e.~the minimum time any of the solvers took to solve it. The cut-offs are at 0 seconds (all problems), took at least 10 seconds to solve, 100 s, 1000 s and lastly 2000 s. Note that if none of the solvers terminates on any instance of a subset, the corresponding row is omitted from the table. A line will also be omitted if nothing changes compared to the previous one. \\
    The relative gap only includes the instances on which at least one run didn't terminate within the time limit. That means it does not include the instances that were solved by all. \\ 
    The average time is taken using the geometric mean shifted by 1 second. Also note that this is the average time over all instances in that group, i.e.~it includes the time outs.} 
    \label{tab:SummaryByDifficultySettings1}
    \vspace{0.1cm}
    \begin{tabular}{lllr Hrrrr Hrrrr Hrrrr Hrrrr} 
        \toprule
        \multicolumn{4}{l}{} & \multicolumn{5}{c}{Default}  & \multicolumn{5}{c}{Heuristics} &  \multicolumn{5}{c}{Shadow Set} & \multicolumn{5}{c}{\thead{MIP SCIP}}\tabularnewline 
        
        \cmidrule(lr){5-9}
        \cmidrule(lr){10-14}
        \cmidrule(lr){15-19}
        \cmidrule(lr){20-24}

        \thead{Type} & \thead{Corr.} & \thead{Solved \\ after \\ (s)} & \thead{\#\\ inst.} & \thead{\# \\solved} & \thead{\% \\solved} & \thead{Time (s)} & \thead{relative \\ gap} & \thead{\# nodes} &\thead{\# \\solved} & \thead{\% \\solved} & \thead{Time (s)} & \thead{relative \\ gap } & \thead{\# nodes}  & \thead{\# \\solved} & \thead{\% \\solved} & \thead{Time (s)} & \thead{relative \\ gap} & \thead{\# nodes} & \thead{\# \\solved} & \thead{\% \\solved} & \thead{Time (s)} & \thead{relative \\ gap} & \thead{\# nodes} \\
        \midrule
        AO & no & 0 & 18 & 9 & \textbf{50 \%} & 358.32 & 7.5254 & 7775 & 9 & \textbf{50 \%} & 363.71 & 0.1521 & 8291 & 9 & \textbf{50 \%} & 365.61 & 5.2159 & 8257 & 8 & 44 \% & 560.05 & 5.2295 & 2368 \\
         &  & 10 & 13 & 4 & \textbf{31 \%} & 1410.24 & 10.4161 & 14871 & 4 & \textbf{31 \%} & 1450.37 & 0.2069 & 16194 & 4 & \textbf{31 \%} & 1450.66 & 7.2184 & 16044 & 3 & 23 \% & 1820.8 & 7.2372 & 2970 \\
         &  & 100 & 10 & 1 & \textbf{10 \%} & 3554.77 & 13.5379 & 50135 & 1 & \textbf{10 \%} & 3587.62 & 0.266 & 55115 & 1 & \textbf{10 \%} & 3594.98 & 9.3809 & 54589 & 0 & 0 \% & 3600.25 & 9.4054 & - \\
         \midrule
        AO & yes & 0 & 18 & 8 & \textbf{44 \%} & 831.95 & 0.1441 & 39425 & 8 & \textbf{44 \%} & \textbf{778.14} & 0.2456 & 33389 & 8 & \textbf{44 \%} & 823.4 & 0.2557 & 35199 & 8 & \textbf{44 \%} & 1434.99 & 0.3364 & 34942 \\
         &  & 10 & 12 & 2 & \textbf{17 \%} & 3066.58 & 0.2783 & 93409 & 2 & \textbf{17 \%} & 2994.35 & 0.3634 & 84460 & 2 & \textbf{17 \%} & \textbf{2944.02} & 0.3785 & 71701 & 2 & \textbf{17 \%} & 3452.15 & 0.4996 & 79349 \\
         \midrule
        AF & no & 0 & 18 & 15 & \textbf{83 \%} & 64.13 & 0.0196 & 15006 & 15 & \textbf{83 \%} & 66.21 & 0.018 & 14433 & 15 & \textbf{83 \%} & \textbf{64.07} & 0.017 & 14059 & 14 & 78 \% & 175.18 & 0.0219 & 10679 \\
         &  & 10 & 11 & 8 & \textbf{73 \%} & 446.0 & 0.0259 & 27310 & 8 & \textbf{73 \%} & 440.79 & 0.0232 & 26154 & 8 & \textbf{73 \%} & \textbf{427.53} & 0.0216 & 25479 & 7 & 64 \% & 1023.65 & 0.0296 & 20399 \\
         &  & 100 & 8 & 5 & \textbf{62 \%} & 1216.92 & 0.0318 & 40277 & 5 & \textbf{62 \%} & 1179.39 & 0.0281 & 38299 & 5 & \textbf{62 \%} & \textbf{1126.97} & 0.0259 & 37212 & 4 & 50 \% & 2145.02 & 0.0369 & 31333 \\
         &  & 1000 & 4 & 1 & \textbf{25 \%} & 3097.25 & 0.0537 & 64107 & 1 & \textbf{25 \%} & \textbf{2905.95} & 0.0463 & 59519 & 1 & \textbf{25 \%} & 3012.72 & 0.0418 & 62081 & 0 & 0 \% & 3600.69 & 0.0639 & - \\
         \midrule
        AF& yes & 0 & 18 & 5 & \textbf{28 \%} & 1330.47 & 2.1036 & 31537 & 5 & \textbf{28 \%} & 1386.95 & 2.1095 & 32697 & 5 & \textbf{28 \%} & 1369.5 & 2.0899 & 32533 & 5 & \textbf{28 \%} & 1919.67 & 2.3779 & 27555 \\
         &  & 10 & 17 & 4 & \textbf{24 \%} & \textbf{1847.36} & 2.2267 & 39036 & 4 & \textbf{24 \%} & 1933.2 & 2.233 & 40487 & 4 & \textbf{24 \%} & 1902.13 & 2.2122 & 40282 & 4 & \textbf{24 \%} & 2481.32 & 2.5172 & 34050 \\
         &  & 100 & 15 & 2 & \textbf{13 \%} & 3107.54 & 2.5223 & 63405 & 2 & \textbf{13 \%} & 3059.65 & 2.5294 & 58879 & 2 & \textbf{13 \%} & 3060.47 & 2.5059 & 60458 & 2 & \textbf{13 \%} & 3431.17 & 2.8515 & 51914 \\
         &  & 1000 & 14 & 1 & \textbf{7 \%} & 3441.35 & 2.7017 & 87165 & 1 & \textbf{7 \%} & 3370.08 & 2.7094 & 75443 & 1 & \textbf{7 \%} & 3399.93 & 2.6842 & 80221 & 1 & \textbf{7 \%} & 3593.22 & 3.0545 & 69063 \\
         \midrule
        DF & no & 0 & 18 & 17 & \textbf{94 \%} & 4.76 & 0.0101 & 219 & 17 & \textbf{94 \%} & 4.54 & 0.0098 & 217 & 17 & \textbf{94 \%} & \textbf{4.43} & 0.0102 & 232 & 17 & \textbf{94 \%} & 11.15 & 0.0103 & 252 \\
         &  & 10 & 3 & 2 & \textbf{67 \%} & 575.13 & 0.0135 & 125 & 2 & \textbf{67 \%} & 419.91 & 0.0119 & 97 & 2 & \textbf{67 \%} & \textbf{356.85} & 0.013 & 89 & 2 & \textbf{67 \%} & 778.56 & 0.0151 & 118 \\
         \midrule
        DF & yes & 0 & 18 & 11 & \textbf{61 \%} & 77.57 & 0.097 & 18067 & 11 & \textbf{61 \%} & 78.71 & 0.097 & 18990 & 11 & \textbf{61 \%} & 79.39 & 0.0968 & 18161 & 9 & 50 \% & 105.04 & 0.1012 & 169 \\
         &  & 10 & 10 & 3 & \textbf{30 \%} & 1942.63 & 0.1673 & 66153 & 3 & \textbf{30 \%} & 2001.29 & 0.1671 & 69540 & 3 & \textbf{30 \%} & 1940.96 & 0.1669 & 66478 & 1 & 10 \% & 2372.28 & 0.1749 & 1243 \\
         &  & 100 & 9 & 2 & \textbf{22 \%} & 3323.61 & 0.1848 & 98596 & 2 & \textbf{22 \%} & 3437.13 & 0.1845 & 103673 & 2 & \textbf{22 \%} & 3320.6 & 0.1844 & 99082 & 0 & 0 \% & 3600.02 & 0.1932 & - \\
         &  & 2000 & 8 & 1 & \textbf{12 \%} & 3455.16 & 0.2067 & 161153 & 1 & \textbf{12 \%} & 3452.91 & 0.2064 & 161139 & 1 & \textbf{12 \%} & 3454.46 & 0.2062 & 161069 & 0 & 0 \% & 3600.02 & 0.2161 & - \\
         \midrule
        DO & no & 0 & 18 & 12 & \textbf{67 \%} & 195.53 & 0.024 & 11231 & 12 & \textbf{67 \%} & 187.77 & 0.0238 & 11117 & 12 & \textbf{67 \%} & 188.2 & 0.0245 & 11062 & 10 & 56 \% & 400.29 & 0.0246 & 3751 \\
         &  & 10 & 11 & 5 & \textbf{45 \%} & 1468.68 & 0.0331 & 25307 & 5 & \textbf{45 \%} & 1394.14 & 0.0327 & 25024 & 5 & \textbf{45 \%} & 1449.37 & 0.0338 & 24941 & 3 & 27 \% & 2057.64 & 0.0341 & 9819 \\
         &  & 100 & 8 & 2 & \textbf{25 \%} & 3389.76 & 0.0418 & 48615 & 2 & \textbf{25 \%} & 3309.23 & 0.0413 & 48268 & 2 & \textbf{25 \%} & 3447.06 & 0.0428 & 48411 & 0 & 0 \% & 3600.31 & 0.0431 & - \\
         &  & 1000 & 6 & 1 & \textbf{17 \%} & 3549.76 & 0.0495 & 54787 & 1 & \textbf{17 \%} & 3470.1 & 0.0483 & 54409 & 1 & \textbf{17 \%} & 3597.51 & 0.0504 & 55215 & 0 & 0 \% & 3600.4 & 0.0506 & - \\
         &  & 2000 & 5 & 1 & \textbf{20 \%} & 3539.8 & 0.0552 & 54787 & 1 & \textbf{20 \%} & 3444.68 & 0.0539 & 54409 & 1 & \textbf{20 \%} & 3597.01 & 0.0567 & 55215 & 0 & 0 \% & 3600.47 & 0.0564 & - \\
         \midrule
        DO & yes & 0 & 18 & 18 & \textbf{100 \%} & 5.47 & 0.0095 & 494 & 18 & \textbf{100 \%} & 5.05 & 0.0095 & 563 & 18 & \textbf{100 \%} & 4.83 & 0.0094 & 385 & 18 & \textbf{100 \%} & 9.75 & 0.0095 & 436 \\
         &  & 10 & 4 & 4 & \textbf{100 \%} & 230.32 & 0.01 & 1980 & 4 & \textbf{100 \%} & 170.88 & 0.0098 & 2232 & 4 & \textbf{100 \%} & 149.39 & 0.01 & 1502 & 4 & \textbf{100 \%} & 333.38 & 0.0099 & 1720 \\
        \bottomrule
    \end{tabular}
\end{sidewaystable}

\begin{sidewaystable}[!ht] \scriptsize
    \centering
    \caption{\footnotesize Comparing the different Boscia settings on the different problems and the different data sets, i.e. A-Fusion (AF), D-Fusion (DF), A-Optimal (AO) and D-Optimal (DO). One data set contains independent data, the other has correlated data. \\
    The settings displayed here are using tightening of the bounds, enabling the shadow set within BPCG and using Strong Branching instead of the most fractional branching rule. \\
    The instances for each problem are split into increasingly smaller subsets depending on their minimum solve time, i.e.~the minimum time any of the solvers took to solve it. The cut-offs are at 0 seconds (all problems), took at least 10 seconds to solve, 100 s, 1000 s and lastly 2000 s. Note that if none of the solvers terminates on any instance of a subset, the corresponding row is omitted from the table. A line will also be omitted if nothing changes compared to the previous one. \\
    The relative gap only includes the instances on which at least one run didn't terminate within the time limit. That means it does not include the instances that were solved by all. \\ 
    The average time is taken using the geometric mean shifted by 1 second. Also note that this is the average time over all instances in that group, i.e.~it includes the time outs.}
    \label{tab:SummaryByDifficultySettings2}
    \vspace{0.1cm}
    \begin{tabular}{lllr Hrrrr Hrrrr Hrrrr Hrrrr Hrrrr} 
        \toprule
        \multicolumn{4}{l}{} & \multicolumn{5}{c}{Strong Branching} & \multicolumn{5}{c}{Tightening} & \multicolumn{5}{c}{\thead{Tighten \\ Lazification}} & \multicolumn{5}{c}{Secant} & \multicolumn{5}{c}{\thead{Blended \\ Conditional \\ Gradient}} \tabularnewline 
        
        \cmidrule(lr){5-9}
        \cmidrule(lr){10-14}
        \cmidrule(lr){15-19}
        \cmidrule(lr){20-24}
        \cmidrule(lr){25-29}

        \thead{Type} & \thead{Corr.} & \thead{Solved \\ after \\ (s)} & \thead{\#\\ inst.} & \thead{\# \\solved} & \thead{\% \\solved} & \thead{Time (s)} & \thead{relative \\ gap} & \thead{\# nodes} & \thead{\# \\solved} & \thead{\% \\solved} & \thead{Time (s)} & \thead{relative \\ gap} & \thead{\# nodes} & \thead{\# \\solved} & \thead{\% \\solved} & \thead{Time (s)} & \thead{relative \\ gap} & \thead{\# nodes} &\thead{\# \\solved} & \thead{\% \\solved} & \thead{Time (s)} & \thead{relative \\ gap } & \thead{\# nodes} & \thead{\# \\solved} & \thead{\% \\solved} & \thead{Time (s)} & \thead{relative \\ gap } & \thead{\# nodes}  \\
        \midrule
        AO & no & 0 & 18 & 8 & 44 \% & 509.39 & 7.7803 & 2170 & 9 & \textbf{50 \%} & 353.06 & 7.5423 & 7847 & 9 & \textbf{50 \%} & 335.12 & 7.562 & 8113 & 8 & 44 \% & 395.39 & 7.5383 & 2369 & 9 & \textbf{50 \%} & \textbf{288.11} & 0.7865 & 8233 \\
         &  & 10 & 13 & 3 & 23 \% & 1785.5 & 10.7691 & 2848 & 4 & \textbf{31 \%} & 1399.5 & 10.4396 & 15068 & 4 & \textbf{31 \%} & 1327.05 & 10.4668 & 15506 & 3 & 23 \% & 1498.71 & 10.434 & 2891 & 4 & \textbf{31 \%} & \textbf{1156.73} & 1.0854 & 15548 \\
         &  & 100 & 10 & 0 & 0 \% & 3635.44 & 13.9969 & - & 1 & \textbf{10 \%} & 3563.21 & 13.5685 & 51487 & 1 & \textbf{10 \%} & 3469.34 & 13.6038 & 52459 & 0 & 0 \% & 3600.14 & 13.5612 & - & 1 & \textbf{10 \%} & \textbf{3226.67} & 1.408 & 51685 \\
        \midrule
        AO & yes & 0 & 18 & 6 & 33 \% & 1411.93 & 0.5182 & 26947 & 8 & \textbf{44 \%} & 895.28 & 0.2547 & 35056 & 8 & \textbf{44 \%} & 1032.46 & 0.3766 & 52251 & 8 & \textbf{44 \%} & 780.03 & 0.255 & 44423 & 7 & 39 \% & 1056.59 & 0.2829 & 50528 \\
         &  & 10 & 12 & 1 & 8 \% & 3482.79 & 0.9416 & 69367 & 2 & \textbf{17 \%} & 2859.77 & 0.377 & 66564 & 2 & \textbf{17 \%} & 3262.57 & 0.5599 & 109156 & 2 & \textbf{17 \%} & 3106.83 & 0.3776 & 98795 & 1 & 8 \% & 3475.97 & 0.5559 & 135093 \\
        \midrule
        AF & no & 0 & 18 & 13 & 72 \% & 224.3 & 0.0402 & 4937 & 15 & \textbf{83 \%} & 69.05 & 0.0184 & 15279 & 14 & 78 \% & 77.05 & 0.0196 & 16534 & 14 & 78 \% & 88.49 & 0.0198 & 12429 & 15 & \textbf{83 \%} & 74.95 & 0.0191 & 23414 \\
         &  & 10 & 11 & 6 & 55 \% & 1382.75 & 0.0595 & 9816 & 8 & \textbf{73 \%} & 471.19 & 0.0239 & 27754 & 7 & 64 \% & 563.16 & 0.0258 & 31716 & 7 & 64 \% & 673.09 & 0.0261 & 23944 & 8 & \textbf{73 \%} & 513.29 & 0.0249 & 42680 \\
         &  & 100 & 8 & 3 & 38 \% & 2711.25 & 0.0781 & 13742 & 5 & \textbf{62 \%} & 1281.81 & 0.0291 & 41038 & 4 & 50 \% & 1540.09 & 0.0317 & 49409 & 4 & 50 \% & 1805.8 & 0.0322 & 37498 & 5 & \textbf{62 \%} & 1471.05 & 0.0305 & 63486 \\
         &  & 1000 & 4 & 0 & 0 \% & 3645.98 & 0.1458 & - & 1 & \textbf{25 \%} & 3011.79 & 0.0485 & 60591 & 1 & 25 \% & 3277.73 & 0.0533 & 94439 & 0 & 0 \% & 3600.18 & 0.0544 & - & 1 & \textbf{25 \%} & 3464.55 & 0.051 & 119747 \\
        \midrule
        AF & yes & 0 & 18 & 5 & 28 \% & 1658.47 & 3.2623 & 30096 & 5 & \textbf{28 \%} & 1325.29 & 2.1184 & 28058 & 5 & \textbf{28 \%} & 1461.5 & 2.4675 & 36446 & 5 & \textbf{28 \%} & \textbf{1296.61} & 2.1909 & 30838 & 5 & \textbf{28 \%} & 1378.94 & 2.1662 & 35027 \\
         &  & 10 & 17 & 4 & 24 \% & 2198.67 & 3.4656 & 37295 & 4 & \textbf{24 \%} & 1824.24 & 2.2424 & 34688 & 4 & \textbf{24 \%} & 2016.49 & 2.6121 & 45073 & 4 & \textbf{24 \%} & 1855.8 & 2.3194 & 38217 & 4 & \textbf{24 \%} & 1899.87 & 2.3009 & 43212 \\
         &  & 100 & 15 & 2 & 13 \% & 3420.23 & 3.9592 & 65163 & 2 & \textbf{13 \%} & 3000.06 & 2.54 & 55034 & 2 & \textbf{13 \%} & 3138.87 & 2.9591 & 68958 & 2 & \textbf{13 \%} & 3062.15 & 2.6273 & 56098 & 2 & \textbf{13 \%} & \textbf{2985.76} & 2.6282 & 62897 \\
         &  & 1000 & 14 & 1 & 7 \% & 3499.51 & 4.263 & 75549 & 1 & \textbf{7 \%} & 3370.29 & 2.7208 & 73985 & 1 & \textbf{7 \%} & 3434.08 & 3.1697 & 87381 & 1 & \textbf{7 \%} & 3412.79 & 2.8142 & 79157 & 1 & \textbf{7 \%} & \textbf{3332.5} & 2.8296 & 78621 \\
        \midrule
        DF & no & 0 & 18 & 16 & 89 \% & 25.69 & 0.0125 & 272 & 17 & \textbf{94 \%} & 4.94 & 0.0104 & 227 & 17 & \textbf{94 \%} & 5.24 & 0.0102 & 333 & 17 & \textbf{94 \%} & 5.27 & 0.0105 & 209 & 17 & \textbf{94 \%} & 5.61 & 0.0103 & 331 \\
         &  & 10 & 3 & 1 & 33 \% & 3300.66 & 0.0281 & 79 & 2 & \textbf{67 \%} & 447.07 & 0.0152 & 101 & 2 & \textbf{67 \%} & 614.7 & 0.0138 & 233 & 2 & \textbf{67 \%} & 451.16 & 0.0142 & 99 & 2 & \textbf{67 \%} & 739.18 & 0.0126 & 232 \\
        \midrule
        DF & yes & 0 & 18 & 10 & 56 \% & 133.22 & 0.1285 & 16635 & 11 & \textbf{61 \%} & 77.67 & 0.0967 & 18104 & 11 & \textbf{61 \%} & 76.96 & 0.0962 & 18141 & 10 & 56 \% & 85.74 & 0.0987 & 16260 & 11 & \textbf{61 \%} & \textbf{72.59} & 0.0938 & 17502 \\
         &  & 10 & 10 & 2 & 20 \% & 2517.61 & 0.2244 & 83020 & 3 & \textbf{30 \%} & 1935.53 & 0.1668 & 66290 & 3 & \textbf{30 \%} & 1838.5 & 0.1656 & 66376 & 2 & 20 \% & 2133.82 & 0.1703 & 81149 & 3 & \textbf{30 \%} & \textbf{1648.16} & 0.1615 & 64030 \\
         &  & 100 & 9 & 1 & 11 \% & 3493.0 & 0.2482 & 163999 & 2 & \textbf{22 \%} & 3323.78 & 0.1842 & 98814 & 2 & \textbf{22 \%} & 3215.14 & 0.1829 & 98904 & 1 & 11 \% & 3495.7 & 0.1882 & 161073 & 2 & \textbf{22 \%} & \textbf{2884.32} & 0.1784 & 95276 \\
         &  & 2000 & 8 & 1 & 12 \% & 3479.82 & 0.2777 & 163999 & 1 & \textbf{12 \%} & 3460.11 & 0.206 & 161077 & 1 & \textbf{12 \%} & 3454.03 & 0.2045 & 161061 & 1 & 12 \% & 3482.87 & 0.2104 & 161073 & 1 & \textbf{12 \%} & \textbf{3362.42} & 0.1995 & 161107 \\
        \midrule
        DO & no & 0 & 18 & 10 & 56 \% & 354.36 & 0.1266 & 3254 & 11 & 61 \% & 196.4 & 0.0243 & 7423 & 12 & \textbf{67 \%} & 191.89 & 0.0248 & 13438 & 11 & 61 \% & 227.48 & 0.0252 & 7050 & 12 & \textbf{67 \%} & \textbf{158.11} & 0.0221 & 14300 \\
         &  & 10 & 11 & 3 & 27 \% & 2133.07 & 0.2009 & 8332 & 4 & 36 \% & 1503.71 & 0.0336 & 18386 & 5 & \textbf{45 \%} & 1429.48 & 0.0343 & 30310 & 4 & 36 \% & 1652.29 & 0.0354 & 17358 & 5 & \textbf{45 \%} & \textbf{1156.8} & 0.03 & 32303 \\
         &  & 100 & 8 & 0 & 0 \% & 3659.93 & 0.2725 & - & 1 & 12 \% & 3473.32 & 0.0425 & 43491 & 2 & \textbf{25 \%} & 3393.4 & 0.0434 & 58562 & 1 & 12 \% & 3577.58 & 0.045 & 42633 & 2 & \textbf{25 \%} & \textbf{3006.11} & 0.0375 & 63716 \\
         &  & 1000 & 6 & 0 & 0 \% & 3679.97 & 0.3538 & - & 0 & 0 \% & 3601.58 & 0.0499 & - & 1 & \textbf{17 \%} & 3540.67 & 0.0516 & 65771 & 0 & 0 \% & 3600.1 & 0.0535 & - & 1 & \textbf{17 \%} & \textbf{3315.64} & 0.0442 & 74469 \\
         &  & 2000 & 5 & 0 & 0 \% & 3696.16 & 0.4197 & - & 0 & 0 \% & 3601.9 & 0.0561 & - & 1 & \textbf{20 \%} & 3528.91 & 0.058 & 65771 & 0 & 0 \% & 3600.12 & 0.0599 & - & 1 & \textbf{20 \%} & \textbf{3261.52} & 0.0499 & 74469 \\
        \midrule
        DO & yes & 0 & 18 & 17 & 94 \% & 18.54 & 0.0094 & 281 & 18 & \textbf{100 \%} & 5.63 & 0.0093 & 378 & 18 & \textbf{100 \%} & 5.5 & 0.0097 & 638 & 18 & \textbf{100 \%} & 5.67 & 0.0094 & 392 & 17 & 94 \% & \textbf{4.76} & 0.0095 & 280 \\
         &  & 10 & 4 & 3 & 75 \% & 1385.8 & 0.0104 & 1206 & 4 & \textbf{100 \%} & 205.16 & 0.0098 & 1456 & 4 & \textbf{100 \%} & 204.36 & 0.01 & 2491 & 4 & \textbf{100 \%} & 221.99 & 0.0098 & 1470 & 3 & 75 \% & \textbf{134.78} & 0.0098 & 1108 \\
        \bottomrule
    \end{tabular}
\end{sidewaystable}

\begin{sidewaystable}[!ht] 
    \centering
    \caption{\footnotesize Comparing performance between A-criterion with and without $\log$. For $\Tr\left(X^{-1}\right)$, we use both the Secant linesearch as well as the Adaptive linesearch wrapped in a \texttt{MonotonicGeneric} linesearch. For $\log\Tr\left(X^{-1}\right)$, we just consider the Secant linesearch since it is numerically better. \\
    The instances for each problem are split into increasingly smaller subsets depending on their minimum solve time, i.e.~the minimum time any of the solvers took to solve it. The cut-offs are at 0 seconds (all problems), took at least 10 seconds to solve, 100 s, 1000 s and lastly 2000 s. Note that if none of the solvers terminates on any instance of a subset, the corresponding row is omitted from the table. \\
    The relative gap is computed for the instances on which at least one method did not terminate within the time limit. That means it excludes the instances on which all methods terminated.
    The average number of nodes is taken over all solved instances for that solver.
    The average time is taken using the geometric mean shifted by 1 second. Also note that this is the average time over all instances in that group, i.e.~it includes the time outs.} 
    \label{tab:SummaryByDifficultyLinesearch}
    \vspace{0.1cm}
    \begin{tabular}{lllr Hrrrr Hrrrr Hrrrr} 
        \toprule
        \multicolumn{4}{l}{} & \multicolumn{5}{c}{\thead{$\Tr\left(X^{-1}\right)$ \\ Monotic + Adaptive}}  & \multicolumn{5}{c}{\thead{$\Tr\left(X^{-1}\right)$ \\ Secant}} &  \multicolumn{5}{c}{\thead{$\log\Tr\left(X^{-1}\right)$ \\ Secant}}\tabularnewline 
        
        \cmidrule(lr){5-9}
        \cmidrule(lr){10-14}
        \cmidrule(lr){15-19}

        \thead{Type} & \thead{Corr.} & \thead{Solved \\ after \\ (s)} & \thead{\#\\ inst.} & \thead{\# \\solved} & \thead{\% \\solved} & \thead{Time (s)} & \thead{relative \\ gap} & \thead{\# nodes} &\thead{\# \\solved} & \thead{\% \\solved} & \thead{Time (s)} & \thead{relative \\ gap } & \thead{\# nodes}  & \thead{\# \\solved} & \thead{\% \\solved} & \thead{Time (s)} & \thead{relative \\ gap} & \thead{\# nodes}  \\
        \midrule
        AO & no & 0 & 18 & 9 & \textbf{50 \%} & 358.32 & 7.5254 & 7775 & 8 & 44 \% & 395.39 & 7.5383 & 2369 & 9 & \textbf{50 \%} & \textbf{279.53} & 0.3128 & 3116 \\
         &  & 10 & 15 & 6 & \textbf{40 \%} & 996.8 & 9.0286 & 11600 & 5 & 33 \% & 1078.18 & 9.0442 & 3712 & 6 & \textbf{40 \%} & \textbf{733.37} & 0.3736 & 4637 \\
         &  & 100 & 10 & 1 & \textbf{10 \%} & 3554.77 & 13.5379 & 50135 & 0 & 0 \% & 3600.14 & 13.5612 & - & 1 & \textbf{10 \%} & \textbf{3259.84} & 0.5553 & 21991 \\
        \midrule
        AO & yes & 0 & 18 & 8 & 44 \% & 831.95 & 0.1441 & 39425 & 8 & 44 \% & 780.03 & 0.255 & 44423 & 11 & \textbf{61 \%} & \textbf{206.45} & 0.0262 & 6508 \\
         &  & 10 & 13 & 3 & 23 \% & 2339.34 & 0.2399 & 45260 & 3 & 23 \% & 2358.33 & 0.3493 & 50312 & 6 & \textbf{46 \%} & \textbf{1051.19} & 0.0326 & 11869 \\
         &  & 100 & 10 & 0 & 0 \% & 3600.03 & 0.4124 & - & 0 & 0 \% & 3600.08 & 0.4511 & - & 3 & \textbf{30 \%} & \textbf{2952.66} & 0.0395 & 22650 \\
         &  & 1000 & 9 & 0 & 0 \% & 3600.03 & 0.4019 & - & 0 & 0 \% & 3600.08 & 0.4318 & - & 2 & \textbf{22 \%} & \textbf{3403.62} & 0.0427 & 31714 \\
        \midrule
        AF & no & 0 & 18 & 15 & \textbf{83 \%} & 64.13 & 0.0196 & 15006 & 14 & 78 \% & 88.49 & 0.0198 & 12429 & 15 & \textbf{83 \%} & \textbf{17.22} & 0.0213 & 355 \\
         &  & 10 & 9 & 6 & \textbf{67 \%} & 908.07 & 0.0294 & 35124 & 5 & 56 \% & 1373.82 & 0.0297 & 31926 & 6 & \textbf{67 \%} & \textbf{152.17} & 0.0329 & 763 \\
        \midrule
        AF & yes & 0 & 18 & 5 & 28 \% & 1330.47 & 2.1036 & 31537 & 5 & 28 \% & 1296.61 & 2.1909 & 30838 & 9 & \textbf{50 \%} & \textbf{270.53} & 5.1204 & 7754 \\
         &  & 10 & 13 & 2 & 15 \% & 3037.99 & 2.8795 & 63405 & 2 & 15 \% & 2986.85 & 2.9945 & 56098 & 4 & \textbf{31 \%} & \textbf{1837.4} & 7.9606 & 17411 \\
         &  & 100 & 12 & 2 & 17 \% & 2995.32 & 2.9574 & 63405 & 2 & 17 \% & 2940.73 & 3.0822 & 56098 & 3 & \textbf{25 \%} & \textbf{2761.53} & 8.9546 & 23174 \\
         &  & 1000 & 11 & 1 & 9 \% & 3399.3 & 3.2253 & 87165 & 1 & 9 \% & 3363.42 & 3.3615 & 79157 & 2 & \textbf{18 \%} & \textbf{3015.56} & 10.2324 & 23230 \\
         \bottomrule
    \end{tabular}
\end{sidewaystable}

\else
\fi

\end{document}